\pdfoutput=1
\RequirePackage{ifpdf}
\ifpdf 
\documentclass[pdftex]{sigma}
\else
\documentclass{sigma}
\fi

\numberwithin{equation}{section}

\newtheorem{Theorem}{Theorem}[section]
\newtheorem{Lemma}[Theorem]{Lemma}
\newtheorem{Proposition}[Theorem]{Proposition}
 { \theoremstyle{definition}
\newtheorem{Remark}[Theorem]{Remark} }

\newcommand{\beq}{\begin{equation}}
\newcommand{\eeq}{\end{equation}}
\newcommand{\bes}{\begin{subequations}}
\newcommand{\ees}{\end{subequations}}

\def\zz{{\mathbb Z}}
\def\cc{{\mathbb C}}

\def\rank{{\rm rk}}
\def\arank{{\rm Rk}}
\def\e{{\rm e}}
\def\d{{\rm d}}
\newcommand{\un}{\underline1}
\def\12{\frac{1}{2}}
\def\32{\frac{3}{2}}
\def\52{\frac{5}{2}}
\def\p{\partial}

\begin{document}
\allowdisplaybreaks

\newcommand{\arXivNumber}{2007.05698}

\renewcommand{\PaperNumber}{056}

\FirstPageHeading

\ShortArticleName{From Heun Class Equations to Painlev\'e Equations}

\ArticleName{From Heun Class Equations to Painlev\'e Equations}

\Author{Jan DEREZI\'NSKI~$^{\rm a}$, Artur ISHKHANYAN~$^{\rm bc}$ and Adam LATOSI\'NSKI~$^{\rm a}$}

\AuthorNameForHeading{J.~Derezi\'nski, A.~Ishkhanyan and A.~Latosi\'nski}

\Address{$^{\rm a)}$~Department of Mathematical Methods in Physics, Faculty of Physics,\\
\hphantom{$^{\rm a)}$}~University of Warsaw, Pasteura 5, 02-093, Warszawa, Poland}
\EmailD{\href{mailto:jan.derezinski@fuw.edu.pl}{jan.derezinski@fuw.edu.pl}, \href{mailto:adam.latosinski@fuw.edu.pl}{adam.latosinski@fuw.edu.pl}}

\Address{$^{\rm b)}$~Russian-Armenian University, 0051 Yerevan, Armenia}
\EmailD{\href{mailto:aishkhanyan@gmail.com}{aishkhanyan@gmail.com}}
\Address{$^{\rm c)}$~Institute for Physical Research of NAS of Armenia, 0203 Ashtarak, Armenia}

\ArticleDates{Received August 25, 2020, in final form May 25, 2021; Published online June 07, 2021}

\Abstract{In the first part of our paper we discuss linear 2nd order differential equations in the complex domain, especially Heun class equations, that is, the Heun equation and its confluent cases. The second part of our paper is devoted to Painlev\'e I--VI equations. Our~philosophy is to treat these families of equations in a unified way. This philosophy works especially well for Heun class equations. We~discuss its classification into 5 supertypes, subdivided into 10 types (not counting trivial cases). We~also introduce in a unified way deformed Heun class equations, which contain an additional nonlogarithmic singularity. We~show that there is a direct relationship between deformed Heun class equations and all Painlev\'e equations. In~particular, Painlev\'e equations can be also divided into 5 supertypes, and subdivided into 10 types. This relationship is not so easy to describe in a completely unified way, because the choice of the ``time variable'' may depend on the type. We~describe unified treatments for several possible ``time variables''.}


\Classification{34A30; 34B30; 34M55; 34M56}

{\small \tableofcontents}

\section{Introduction}

There are many types of differential equations and special functions. Typically, within a given class there is one generic type and many confluent types.
This is the case of Riemann (hypergeometric) class equations, Heun class equations, as well as Painlev\'e equations. For instance, the generic type of Riemann
class equations can always be reduced to the Gauss hypergeometric equation, but there are also confluent types such as Kummer's confluent equation, the $F_1$ equation, and the Hermite equation, see, e.g., \cite{De,De2,Mason,Ronveaux,Sl_Lay}.

One can try to understand the process of confluence by considering
equations depending holomorphically on parameters.
Some properties of the whole class can be described in a uniform way,
without splitting the class into types.
For instance, one can identify various transformations (``symmetries'') that leave the class invariant.

In~order to study the equations
in more detail, one needs to split the class into types.
Within a given type one can simplify the equation by symmetries and convert it to a {\em normal form},
thereby reducing the number of parameters. This has to be done case by case.

In~the case of hypergeometric class equations, this idea was successfully applied in the book by~Nikiforov--Uvarov \cite{NU}. It works especially well for hypergeometric polynomials, that is Jacobi, Laguerre, Bessel and Hermite polynomials, which can be elegantly treated in a unified way.

In~this paper, we try to apply this idea to the derivation of Painlev\'e equations from Heun class equations by the method of isomonodromic deformation.
We~will see that each type of~Heun class equation corresponds to a
(properly understood) type of Painlev\'e equation. The~passage from
Heun class to Painlev\'e can be accomplished in a fairly uniform way, although one has to~consider
several similar but distinct cases.

We~start our paper by Section~\ref{s1} containing basic theory of
singularities of 2nd order scalar ordinary differential equations.
This is a classic subject with several well-known textbooks, such as \cite{Ince,Poole}.
 We~follow to a large extent the
treatment
described in the monograph by Slavyanov--Lay \cite{Sl_Lay} and the
appendix to \cite{Ronveaux}, written by Slavyanov. Sometimes we
introduce new notation and terminology to
make precise some concepts which in \cite{Sl_Lay} are implicit.

The central concept of the theory of linear differential equations in the complex domain is the {\em rank of a singularity}. In~our paper we introduce several kinds of the rank.
In~particular, we distinguish between the usual rank and the {\em absolute rank}
(the infimum over the ranks of all possible transformed forms of a given equation).
The rank can be an integer or a half-integer. We~also introduce the {\em rounded rank}, which has always an integer value.
Thus if the rank is $m$ or $m-\frac12$, where $m$ is an integer, then
we say that its rounded rank is $m$. In~particular, the rounded rank
is $1$ if the singularity is {\em Fuchsian} (also called {\em regular}).
We~believe that all these concepts clarify the theory of ordinary differential equations. We~also discuss formal power series solutions of these equations (the so-called {\em Thom\'e solutions}). We~introduce the concept of {\em indices}
of a singular point. This is of course well-known for Fuchsian singularities. For non-Fuchsian singularities this concept is not so well known, although it is implicit in \cite{Sl_Lay}.

In~Section~\ref{s2} we discuss equations with rational coefficients.
Such equations have a finite number of singularities on the Riemann sphere.
Following~\cite{Sl_Lay}, the class of equations with $n$ singularities in
$\cc$ and a singularity at infinity, and their confluent cases are
called {\em $M_n$ class} equations. red It is easy to see that these equations can be
written as
 \begin{gather}
 \big(\sigma(z)\partial_z^2+\tau(z)\partial_z+\eta(z)\big)u(z)=0, \label{pde.}
 \end{gather}
where $\sigma$ is a polynomial of degree $\leq n$, $\tau$ of degree
$\leq n-1$ and $\sigma\eta$ of degree $\leq 2n-2$.
We~introduce also a closely related {\em grounded $M_n$ class}, for
which one of indices of all finite singularities is zero~--- hence the
name ``grounded''. (In~\cite{Sl_Lay} such equations are called {\em
 canonical}. In~our opinion, the word canonical is overused, hence
not appropriate for this meaning.) They can be written as~\eqref{pde.}
with the same conditions on $\sigma$, $\tau$, and with $\eta$ a
polynomial of degree $\leq n-2$.

The best known classes of equations with rational coefficients are the
$M_2$ class and the grounded $M_2$ class. We~call $M_2$ the {\em
 Riemann class}, since its generic representative is the {\em Riemann equation} with one singularity at $\infty$.
The grounded $M_2$ class is especially often
encountered in the literature. It is the main subject of the textbook
by Nikiforov--Uvarov~\cite{NU}, where its elements are called
``hypergeometric type equations''. Note that the difference between
the full and grounded Riemann class is minor~--- the only type of
equations contained in the full Riemann class but not represented in
the grounded Riemann class is the {\em Airy equation}.

 One of the central objects of our paper is the $M_3$ class, called
also the {\em Heun class}, for which~$\sigma$ is a polynomial of degree $\leq3$, $\tau$ of degree $\leq2$
and $\sigma\eta$ of degree $\leq4$.
The main type within the Heun class is the standard {\em Heun type},
that is the equation with 4 Fuchsian singularities in~the Riemann
sphere, one of which put at $\infty$. It was first analyzed by Heun \cite{Heun}, see also~\cite{Maier} for a~recent study.

In~the literature the name Heun class equations is employed in two meanings: the meaning that we have just described is used in \cite{Ronveaux,Sl_Lay}. It is also common to use it for what we call the {\em grounded Heun class}, see, e.g.,
\cite{FID}. The grounded Heun class (not counting types of the Riemann class) is divided into 5 types: {\em standard, confluent, biconfluent, doubly confluent} and {\em triconfluent}.
The full Heun class, beside the above five types has five more types, which we call {\em degenerate confluent, degenerate biconfluent, degenerate doubly confluent, doubly degenerate doubly confluent} and {\em degenerate triconfluent}. (\cite{Sl_Lay} uses the word ``reduced'' instead of ``degenerate''.)

Sometimes a singularity of an equation does not lead to a singularity
of its solutions. We~call such points {\em non-logarithmic
 singularities}.
 Non-logarithmic singularities play an important role in the
 following construction.

 Let us start with an
equation \eqref{pde.}. We~introduce the {\em deformed form of \eqref{pde.}}
to be the equation
\begin{gather}
\bigg(\sigma(z)\partial_z^2+
\bigg(\tau(z)-\frac{\sigma(z)}{z-\lambda}\bigg)\partial_z
+\eta(z)-\eta(\lambda)-\mu^2\sigma(\lambda)-\mu\big(\tau(\lambda)-\sigma'(\lambda)\big)
+\frac{\mu\sigma(\lambda)}{z-\lambda}\bigg)\nonumber
\\ \qquad
{}\times v(z)=0.
\label{h-bis-1}
\end{gather}
Note that all finite singularities of \eqref{h-bis-1} are the same as
in \eqref{pde.}, except that there is an additional non-logarithmic
singularity with indices $0$, $2$. $\lambda$ is the position of this
singularity and $\mu= \frac{v'}{v}(\lambda)$.
If~\eqref{pde.} is a Heun class equation, then \eqref{h-bis-1} is called
a {\em deformed Heun class equation} \cite{Sl_Lay}.

Let us note that non-logarithmic singularities of an
equation can be produced by a {\em change of gauge}, that is by replacing
the unknown $v(z)$ with $c_0(z)v(z)+c_1(z)v'(z)$, where $c_0$, $c_1$ are
some rational functions.
This procedure applied to grounded Heun class equations produces the so-called {\em derivative Heun class
 equations}, see~\cite{FID,Ishkhanyan, Sl-etal}. Kimura proved in~\cite{Ki} that in deformed Riemann class equations one can
remove all non-logarithmic singularities by a change of gauge. We~do not know if the same is
possible for all kinds of deformed Heun class equations.

{\em Painlev\'e equations} is a famous class of nonlinear differential equations with the so called {\em Painlev\'e property}--the absence of moving essential and branch singularities in its solutions. They were discovered in the beginning of 20th century \cite{Gam,Pa1,Pa2}.
Traditionally, Painlev\'e equations are divided into 6 types, called
Painlev\'e I, II, III, IV, V and VI \cite{Iwasaki}.
As noted by Ohyama--Oku\-mura~\cite{Ohyama}, it is actually natural to subdivide some of them into smaller types, obtaining altogether 10 types. Each of them corresponds
to one of types of the Heun class equations. We~obtain the following correspondence between types of the Heun class and Painlev\'e:
\begin{center}\renewcommand{\arraystretch}{1.3}
{\small\begin{tabular}{cll}\hline
 $(\un\un\un\un)$& (standard) Heun&Painlev\'e VI
 \\\hline
 $(\un\un2)$& confluent Heun& Painlev\'e nondegenerate V
 \\\hline
 $\big(\un\un\frac32\big)$ & degenerate confluent Heun& Painlev\'e degenerate V
 \\\hline
 $(22)$& doubly confluent Heun& Painlev\'e nondegenerate III$'$
 \\\hline
 $\big(\frac322\big)$ & degenerate doubly confluent Heun& Painlev\'e degenerate III$'$
 \\\hline
$\big(\frac32\frac32\big)$& doubly degenerate doubly confluent Heun&Painlev\'e doubly degenerate III$'$
\\\hline
 $(\un3)$& bi-confluent Heun&Painlev\'e IV
 \\\hline
 $\big(\un\frac52\big)$ & degenerate bi-confluent Heun& Painlev\'e 34
 \\\hline
 $(4)$ & tri-confluent Heun&Painlev\'e II
 \\\hline
 $\big(\frac72\big)$ & degenerate tri-confluent Heun&Painlev\'e I
 \\\hline
 \end{tabular}}
 \end{center}

\medskip
Above, the symbols such as $(\un\un2)$ indicate the ranks of the singularities in the Heun class equation.
We~use underline to indicate the rounded rank: thus
\begin{alignat*}{3}
& \underline1 \text{ denotes either $1$ or $\frac12$ (a Fuchsian singularity);}\qquad &&
 \underline2 \text{ denotes either $2$ or $\frac32$;}& \\
& \underline3 \text{ denotes either $3$ or $\frac52$;}\qquad &&
\underline4 \text{ denotes either $4$ or $\frac72$.}&
 \end{alignat*}

Following the literature, instead of Painlev\'e III we prefer to use an equivalent equation Painlev\'e III$'$.

Note that the Painlev\'e deg-V is equivalent to Painlev\'e III$'$
and Painlev\'e 34 is equivalent to Painlev\'e II by a relatively complicated change of variables.

As noted by Ohyama--Okumura \cite{Ohyama}, there also exists another, coarser
classification of Pain\-lev\'e equations into five supertypes. It corresponds to a coarser classification of Heun class equations into supertypes,
where we use rounded ranks instead of the usual ranks. We~obtain the following table, which is discussed in detail in Section~\ref{FivesupertypesofPainleveequation}:
\begin{center}\renewcommand{\arraystretch}{1.3}
{\small\begin{tabular}{cll}\hline
 $(\un\un\un\un)$& (standard) Heun&Painlev\'e VI
 \\\hline
 $(\un\un\underline2)$& confluent Heun& Painlev\'e V
 \\\hline
 $(\underline2\underline2)$& doubly confluent Heun& Painlev\'e III$'$
 \\\hline
 $(\un\underline3)$& bi-confluent Heun&Painlev\'e IV-34
 \\\hline
 $(\underline4)$ & tri-confluent Heun&Painlev\'e II--I
 \\\hline
 \end{tabular}}
 \end{center}

The main topic of our paper, described in Section~\ref{section: Painleve},
is a derivation of
Painlev\'e equations from
Heun class equations.
The first step of this derivation is a choice of a family of Heun
class equations
depending on a
parameter denoted $t$ and called the {\em time}.
Then we consider then the corresponding deformed Heun class equation,~\eqref{h-bis-1}, which depends on two additional vari\-ab\-les:~$\lambda$,~$\mu$.
The conditions for a constant monodromy lead to a set of nonlinear
differential equation for $\lambda$, $\mu$ in terms of $t$. These
equations can be interpreted as Hamilton equations generated by
$H(t,\lambda,\mu)$, which we
will call {\em Painlev\'{e} Hamiltonians}.

The most difficult ingredient of the passage from Heun class to
Painlev\'e is the choice of the time variable and of the so-called
{\em compatibility functions} $a,b$, which control the isomonodromic
deformations. The main idea of our paper is to present this passage in
a unified way. Our first attempt in this direction is described in Theorem~\ref{main0}. The description of compatibility
functions and the corresponding Hamiltonian contained in this theorem
is unified, however it is not very transparent.
Theorem~\ref{main1}, which can be viewed as the first part of the main
result of this paper, is more satisfactory. The method of
isomodronomic deformation is here divided into two slightly different
ansatzes, called Cases~A and~B. The main condition on~$\sigma$, $\tau$, $\eta$, $t$ for the applicability of Ansatz A is the existence of a zero of
 $\sigma$, so that we can write $\sigma(z)=(z-s)\rho(z)$, where $\rho$ is a
polynomial of degree $\leq2$. Under these conditions
there exists a function $t\mapsto m(t)$ such that
\begin{gather}
H(t,\lambda,\mu)=m\big(\sigma(\lambda)\mu^2+(\tau(\lambda) -(\lambda-s)\rho'(\lambda))\mu+\eta(\lambda)\big)
\label{paino1}
\end{gather}
is a Painlev\'{e} Hamiltonian.

Among the conditions on
$\sigma$, $\tau$, $\eta$, $t$ needed for
Ansatz B the most important is $\deg\sigma\leq2$.
Then we can show that
there exists $t\mapsto m(t)$ such that
\begin{gather}
H(t,\lambda,\mu)=m\big(\sigma(\lambda)\mu^2+(\tau(\lambda) -\sigma'(\lambda))\mu+\eta(\lambda)\big) \label{paino2}
\end{gather}
is a Painlev\'{e} Hamiltonian.

It seems impossible to implement Theorem~\ref{main1} in a
fully unified way, and one has to subdivide Cases A and B into
several subcases. Case A splits into Subcases A1, Ap,
Aq and Case B splits into Subcases~Bp and Bq.
We~describe these subcases in Theorem~\ref{main}. The main difference
between the subcases is the choice of the time variable $t$:

In~Subcase~A1 the variable $t$ is the position of one of Fuchsian
singularities.

In~Subcases Ap and Bp the variable $t$ is contained in $\tau$.

In~Subcases Aq and Bq the variable $t$ is contained in $\eta$.

In~the following list we informally describe when we can apply various subcases.
\begin{itemize}\itemsep=0pt\setlength{\leftskip}{0.75cm}
{\samepage\item[A1.] $\sigma(t)=0$, $\sigma'(t)\neq0$, $\deg(z-t)\eta\leq2$.

\item[Ap.] $\sigma(s)=\sigma'(s)=0$, $\tau(s)\neq0$, $\sigma\eta(s)=(\sigma\eta)'(s)=0$.
}

\item[Aq.] $\sigma(s)=\sigma'(s)=0$, $\tau(s)=0$, $\sigma\eta(s)=0$, $(\sigma\eta)'(s)\neq0$.

\item[Bp.] $\deg\sigma\leq2$, $\deg\tau=2$, $\deg\sigma\eta\leq2$.

\item[Bq.] $\deg\sigma\leq2$, $\deg\tau\leq 1$, $\deg\sigma\eta=3$.
 \end{itemize}

Note that Subcase~A1 works in the
standard Heun
type $(\un\un\un\un)$, leading to Painlev\'{e} VI. Hence it works in
under generic conditions.
However, it does not works for some confluent types. For instance, for most degenerate
types one needs to use either Subcase~Aq or Subcase~Bq.
Altogether, Theorem~\ref{main} works for certain normal forms of all
types of Heun class equations. This
 allows us to derive all types of Painlev\'e equations.

The derivation of the
Painlev\'e VI equation from the Heun equation can be traced back to a~paper by Fuchs \cite{Fu} from the early 20th century (written by the son of Fuchs from whose name the adjective ``Fuchsian'' comes). The approach was generalized to other Painlev\'e equations by~Okamoto \cite{Okamoto} and refined by Ohyama--Okumura \cite{Ohyama}. A discussion of the relationship between the biconfluent Heun type and the Painlev\'e IV equation can be also found in \cite{CLY}.
Thus derivations described in Sections~\ref{painleve6}--\ref{painleve1} are known. They are in particular
described by Ohyama--Okumura \cite{Ohyama}, where they were checked case by case.
Our approach allows us to automatize these derivations and view them as implementations of a unified algorithm.
In~fact, our paper can be to some extent viewed as an explanation of the principles that underly
the results of \cite{Ohyama}.

Slavyanov and Lay devote in \cite{Sl_Lay} a whole chapter to the Heun class~--- Painlev\'{e}
correspondence. In~particular, they stress that Painlev\'{e}
Hamiltonians can be viewed as ``dequantizations'' of the corresponding
Heun class equations. In~fact, the symbol of Heun class operators,
that is the expression obtained from \eqref{pde.} by
replacing $\partial_z$ with $\mu$ and $z$ with $\lambda$, is very
similar to~\eqref{paino1} and~\eqref{paino2}. (The difference is a
``lower order term'', which can be interpreted as a result of~an~``ordering prescription''.) Nevertheless, to our understanding,
\cite{Sl_Lay} contains only a sketch of~a~program. The details of this
program are quite involved and \cite{Sl_Lay} does not
provide its full description.

Clearly, every 2nd order scalar equation can be rewritten as a system
of two 1st order equations. For instance, we can rewrite \eqref{pde.} as
\begin{gather*}
 \sigma(z)\partial_z\begin{bmatrix}u(z)\\v(z)\end{bmatrix}=\begin{bmatrix}0&1
 \\
 -\sigma(z)\eta(z)&-\tau(z)+\sigma'(z)\end{bmatrix}
 \begin{bmatrix}u(z)\\v(z)\end{bmatrix}
 \end{gather*}
(or in other ways). In~particular, Heun class equations are essentially equivalent to
certain systems of 1st
order equations called sometimes {\em Heun connections}. This suggests
a different approach to deriving Painlev\'{e} equations
starting from Heun connections.
This alternative approach was studied, e.g., by Jimbo and Miwa
\cite{Ji,JiM}. A recent exposition of this approach can be found in
\cite{IP}.

Note that in the above references each type of Painlev\'{e} equations
were treated separately. Besides, degenerate cases
were usually left out.
It would be interesting to investigate to what extent a derivation of
all types
Painlev\'{e} equations from Heun connections can be treated in a~uniform way, in the spirit of Theorems~\ref{main1} and~\ref{main}.
Anyway, in this paper we do not consider this question and we stick to
(scalar 2nd order)
deformed Heun class equations.

\section[Second order linear differential equations in the complex domain]
{Second order linear differential equations \\in the complex domain}\label{s1}

\subsection{Differential equation and operator}
Let us recall basic concepts of ordinary 2nd order linear
differential equations in the complex domain with holomorphic
coefficients. They have the form \eqref{pde.},
 where $\sigma$, $\tau$ and $\eta$ are holomorphic functions.
 We~will often describe the equation~(\ref{pde.}) by specifying the corresponding operator
\begin{gather}
\sigma(z)\partial_z^2+\tau(z)\partial_z+\eta(z).\label{pde-op}
\end{gather}

Clearly, by
multiplying an equation
from the left by an arbitrary nonzero holomorphic function we obtain an equivalent equation. However, we change the corresponding operator.
Speaking of operators instead of equations, which we will often do, has two advantages. First it saves a little space, since we do not need to write the function $u$.
Besides, an operator contains more information than the corresponding equation, therefore sometimes allows for making more precise statements.

Divide~(\ref{pde.}) and~(\ref{pde-op}) by $\sigma(z)$ and set
\begin{gather*}
p(z):=\frac{\tau(z)}{\sigma(z)}, \qquad
q(z):=\frac{\eta(z)}{\sigma(z)}.
\end{gather*}
This leads to the so-called {\em principal form of the equation}
and the corresponding {\em principal operator}:
\begin{gather}
\big(\partial_z^2+p(z)\partial_z+q(z)\big)u(z)=0,
\label{pde}
\\
A:= \partial_z^2+p(z)\partial_z+q(z).
\label{pde-op.}
\end{gather}
The principal form is often used as the standard form.
However, we will often prefer different forms.

 \subsection{Singularities of functions}

 Let $p$ be a function holomorphic on an open subset of the Riemann
 sphere $\cc\cup\{\infty\}$.
 Let $z_0\in\cc\cup\{\infty\}$ be its singularity, so that
\begin{gather*}
 p(z)=\sum_{k=-\infty}^\infty p_{z_0,k}(z-z_0)^k,\qquad
 |z-z_0|<r\qquad \text{for some}\quad r>0,\quad \text{if}\quad z_0\in\cc,
 \\
 p(z)=\sum_{k=-\infty}^\infty p_{\infty,k} z^k,\qquad
 |z|>R \qquad \text{for some}\quad R\geq0,\quad \text{if}\quad z_0=\infty.
\end{gather*}
We~define the {\em degree of the singularity} of $p$ at $z_0$ by
\begin{gather*}
 \deg(p,z_0):=-\min\{k\mid p_{z_0,k}\neq0\},\qquad z_0\in\cc,
 \\
 \deg(p,\infty):=\max\{k\mid p_{\infty,k}\neq0\},\qquad z_0=\infty.
\end{gather*}
(If $p$ is a polynomial, then $\deg(p,\infty)=\deg(p)$ is its usual degree.)

Note that if $\phi$ is a biholomorphic transformation of a neighborhood of $z_0$ onto a neighborhood of $\phi(z_0)$, then
\begin{gather*}
\deg(p,z_0)=\deg\big(p\circ\phi^{-1},\phi(z_0)\big).
 \end{gather*}

 \subsection{Singularities of equations}
Consider now the equation~(\ref{pde}) (in the principal form) with holomorphic coefficients represented by the operator~(\ref{pde-op.}), which we denote by $A$.
We~say that $z_0$ is a {\em regular point} of $A$ if{\samepage
\begin{gather*}
\begin{split}
& \deg(p,z_0)\leq0,\qquad \deg(q,z_0)\leq0,\qquad z_0\in\cc,\\
& \deg\bigg(p-\frac{2}{z},\infty\bigg)\leq-2,\qquad\deg(q,\infty)\leq-4,\qquad z_0=\infty.
\end{split}
\end{gather*}
Otherwise we say that it is a {\em singular point} of $A$.}

We~say that the singular point $z_0$ is {\em regular} or {\em Fuchsian} if
\begin{gather*}
 \deg(p,z_0)\leq1,\qquad \deg(q,z_0)\leq2,\qquad z_0\in\cc,
 \\
 \deg(p,\infty)\leq -1,\qquad \deg(q,\infty)\leq-2,\qquad z_0=\infty.
\end{gather*}
It is standard to introduce two {\em indices} of a Fuchsian singular point:
\begin{gather}
 \text{the roots of}\quad \lambda(\lambda-1)+p_{z_0,-1}\lambda+q_{z_0,-2}\quad \text{are called\quad{\em indices of $z_0\in\cc$},}\label{rho4}
 \\
 \text{the roots of}\quad \lambda(\lambda+1)-p_{\infty,-1}\lambda+q_{\infty,-2}\quad
 \text{are called\quad{\em indices of $z_0=\infty$}}.\label{rho5}
\end{gather}

The {\em rank of $A$ at $z_0$} is defined as follows. If $z_0$ is a regular point, we set $\rank(A,z_0)=0$. The case of rank equal to $\frac12$ is somewhat special:
\begin{gather*}
\rank(A,z_0):=\frac12\qquad \text{if}\quad
\deg\bigg(p-\frac1{2(z-z_0)},z_0\bigg)\leq 0,\quad \deg(q,z_0)\leq1,\quad z_0\in\cc,
\\
\rank(A,\infty):=\frac12\qquad \text{if}\quad
\deg\bigg(p-\frac3{2z},\infty\bigg)\leq-2,\quad \deg(q,\infty)\leq-3,\quad z_0=\infty.
\end{gather*}
If $\rank(A,z_0)\neq0,\frac12$, then we set
\begin{gather*}
\rank(A,z_0):=\max\bigg\{\deg(p,z_0),\, \frac12\deg(q,z_0),\,1\bigg\},\qquad z_0\in\cc,
\\
\rank(A,\infty):=\max\bigg\{\deg(p,\infty)+2,\, \frac12\deg(q,\infty)+2,\,1\bigg\},\qquad z_0=\infty.
\end{gather*}

The rank and indices of a singularity are invariants of biholomorphic transformations. For instance,
this is the case of homographies, that is $w=\frac{az+b}{cz+d}$, or $z=\frac{dw-b}{-cw+a}$,
where $ad-bc=1$. We~obtain
 \begin{gather} \label{pde3}
 \partial_z^2+p(z)\partial_z+q(z) =(-cw+a)^4 \partial_w^2\nonumber
 \\ \hphantom{\partial_z^2+p(z)\partial_z+q(z)}
 {}+\! \bigg(\!{-}2c(-cw+a)^3\!+p\bigg(\frac{dw-b}{-cw\!+a}\bigg)(-cw\!+a)^2\bigg)\partial_w
 +q\bigg(\frac{dw-b}{-cw\!+a}\bigg).\!\!\!\!
 \end{gather}
 In~order to obtain the principal form we need to divide~(\ref{pde3})
 by $(-cw+a)^4$.

Note that the rank is always an integer or a half-integer.
A singularity of rank $m-\frac12$ with $m\in\{1,2,\dots\}$ can be often treated
as a degeneration of a singularity of rank $m$. This motivates us to
introduce the {\em rounded rank}, denoted
$ \lceil\rank\rceil$:
\begin{gather*}
 \lceil\rank\rceil(A,z_0):=
 \begin{cases}\rank(A,z_0),&\text{if}\quad\rank(A,z_0)\in\{0,1,2,\dots\},
 \\[.5ex]
 \rank(A,z_0)+\frac12,&\text{if}\quad\rank(A,z_0)\in\big\{\frac12,\frac32,\frac52,\dots\big\}.
 \end{cases}
\end{gather*}
Equivalently,
\begin{gather*}
\lceil\rank\rceil(A,z_0):=
\bigg\lceil \max\bigg\{\deg(p,z_0),\, \frac12\deg(q,z_0),\,0\bigg\}\bigg\rceil,\qquad z_0\in\cc,
\\[.5ex]
\lceil\rank\rceil(A,\infty):=
\bigg\lceil\max\bigg\{\deg(p,\infty)+2,\, \frac12\deg(q,\infty)+2,\,0\bigg\}\bigg\rceil,\qquad z_0=\infty.
\end{gather*}
Here $\lceil\cdot\rceil$ is the {\em ceiling function}, that is
\begin{gather*}
\lceil x\rceil:=\inf\{n\in\zz\mid x\leq n\}.
\end{gather*}

The singularity $z_0$ is Fuchsian if its rank is $\frac12$ or $1$.
Thus $z_0$ is a Fuchsian singularity iff its rounded rank is $1$.

According to our definition, a Fuchsian singularity has rank $\frac12$ if its indices are $0$, $\frac12$.
The splitting of the Fuchsian case into two subcases, the rank $\12$ and $1$, is quite useful, even if its definition is not obvious. Nevertheless, in our paper we will not make much use of this splitting and both will be usually treated as one case, denoted $\un$.

\subsection{Sandwiching with functions}

The family of equations~(\ref{pde}) is preserved by several kinds of transformations of the form
\begin{gather}
\nonumber 
\e^{-r(z)}\big( \partial_z^2\!+p(z)\partial_z\!+q(z)\big)\e^{r(z)}
= \partial_z^2\!+\big(p(z)\!+2r'(z)\big)\partial_z\!+q(z)\!+ p(z)r'(z)\!+ r'(z)^2\!+r''(z)
\\ \hphantom{\e^{-r(z)}\big( \partial_z^2\!+p(z)\partial_z\!+q(z)\big)\e^{r(z)}}
{} =: \partial_z^2+\tilde p(z)\partial_z+\tilde q(z)=\tilde A.
 \label{pde-tra1}
\end{gather}

\noindent
{\bf Sandwiching with powers.}
For $r(z)=\kappa\log(z)$
\begin{gather}
(z-z_0)^{-\kappa} \big(\partial_z^2+p(z)\partial_z+q(z)\big) (z-z_0)^{\kappa}
 =\partial_z^2+\big(p(z)+2\kappa(z-z_0)^{-1}\big)\partial_z+q(z)\nonumber
 \\ \hphantom{(z-z_0)^{-\kappa} \big(\partial_z^2+p(z)\partial_z+q(z)\big) (z-z_0)^{\kappa} =}
{} +p(z)\kappa(z-z_0)^{-1}+(\kappa^2-\kappa)(z-z_0)^{-2}.\label{tra1}
\end{gather}
If $z_0\in\cc\cup\{\infty\}$ is a singularity of~(\ref{pde}) and $\rank(z_0)>1$, then the
transformation~(\ref{tra1}) preserves its rank.
If $z_0$ is a Fuchsian singularity, then after the transformation it is also Fuchsian.

If $\rho_1$, $\rho_2$ are the indices of~(\ref{pde}) at $z_0\in\cc$, then
 $\rho_1+\kappa$, $\rho_2+\kappa$ are the indices of~(\ref{tra1}) at $z_0$.
The same is true for $\infty$, except that the indices
of~(\ref{tra1}) at $\infty$ are $\rho_1-\kappa$, $\rho_2-\kappa$.

\medskip
\noindent{\bf Sandwiching with exponentials.}
Let $k=2,3,\dots$. We~have
\begin{gather}\nonumber
\exp\bigg(\frac{\kappa(z-z_0)^{-k+1}}{k-1}\bigg)
 \big(\partial_z^2+p(z)\partial_z+q(z)\big) \exp\bigg({-}\frac{\kappa(z-z_0)^{-k+1}}{k-1}\bigg)
 \\ \qquad\nonumber
{} =\partial_z^2+\big(p(z)+2 \kappa(z-z_0)^{-k}\big)\partial_z
 +q(z)+p(z)\kappa(z-z_0)^{-k}
 \\ \qquad\phantom{=}
 {}+\kappa^2(z-z_0)^{-2k}-\kappa k(z-z_0)^{-k-1}.\label{pde1}
 \end{gather}
Hence this transformation preserves $\rank(z_0)$ if it~is~$>k$,
and preserves or decreases $\rank(z_0)$ if it~is~$=k$. The transformation does not change the coefficients $p_{z_0,-1}$, $q_{z_0,-1}$, $q_{z_0,-2}$. Therefore, it~also does not change the indices of $z_0$.

The same is true for $\infty$ under the transformation
\begin{gather}
\exp\bigg({-}\frac{\kappa z^{k+1}}{k+1}\bigg)
\big(\partial_z^2+p(z)\partial_z+q(z)\big)
\exp\bigg(\frac{\kappa z^{k+1}}{k+1}\bigg)\nonumber
\\ \qquad
{}=\partial_z^2+\big(p(z)+2\kappa z^{k}\big)\partial_z
+q(z)+p(z)\kappa z^k+\kappa^2z^{2k}+\kappa k z^{k-1}.\label{pde2}
\end{gather}

More generally, we have transformations of the form~(\ref{pde-tra1}),
where
\begin{gather}
r(z)=\sum_{k=-m+1}^{-1}\frac{w_k}{k}(z-z_0)^{k}+w_{0}\log(z-z_0),\label{tra1.}
\\
\text{so that}\quad
r'(z)=\sum_{k=-m+1}^{0}w_k(z-z_0)^{k-1} \qquad \text{if}\quad z_0\in\cc;\nonumber
\\
r(z)=-\sum_{k=1}^{m-1}\frac{w_k}{k}z^{k}-w_{0}\log(z),\label{tra2.}
\\
\text{so that}\quad
r'(z)=-\sum_{k=0}^{m-1}w_kz^{k-1} \qquad \text{if}\quad z_0=\infty.\nonumber
\end{gather}
We~define the {\em absolute rank} of $A$ at $z_0$ as
\begin{gather*}
\arank(A,z_0):=\inf\big\{\rank(\tilde A,z_0)\big\},
\end{gather*}
where $\tilde A$ are all possible transforms of $A$ of the form~(\ref{pde-tra1}) with $r$ as in~(\ref{tra1.}) or~(\ref{tra2.}).

 \subsection{Half-integer rank}
\label{Half-integer rank}

 In~this subsection we discuss singular points with a half-integer rank. They are in a sense exceptional and have special properties.

 Suppose that the equation~(\ref{pde}) has a singular point at $z_0$ and $ \rank(A,z_0)=m+\frac12$, where $m=0,1,\dots$. It is easy to see that this implies
 \begin{gather*}
 \arank(A,z_0)= \rank(A,z_0) .
\end{gather*}

 Without loss of generality we can assume that
 the singularity is at $0$. This is equivalent to
 \begin{gather*}
 \deg(p,0)\leq m,\qquad \deg(q,0)=2m+1,\qquad m\geq1,
 \\
 \deg\bigg(p-\frac1{2z},0\bigg)\leq 0,\qquad \deg(q,0)\leq1,\qquad m=0.
\end{gather*}
 Let us make the substitution
 \begin{gather}\label{quadra}
 z=y^2,\qquad y=\sqrt{z}.
 \end{gather}
 Using $\partial_z=\frac{1}{2y}\partial_y$ we transform~(\ref{pde-op.}) into
 \begin{gather} \label{pde-sq}
 \frac{1}{4y^2}\partial_y^2-\frac{1}{4y^3}\partial_y+\frac{p(y^2)}{2y}\partial_y +q\big(y^2\big).
\end{gather}
 Multiplying~(\ref{pde-sq}) by $4y^2$ we obtain
 an equation in the principal form
 \begin{gather} \label{pde-sq1}
 \partial_y^2+2y\bigg(p\big(y^2\big)-\frac{1}{2y^2}\bigg)\partial_y+4y^2q\big(y^2\big).
 \end{gather}
 Now
 \begin{gather*}
 \deg\bigg(2y\bigg(p\big(y^2\big)-\frac{1}{2y^2}\bigg),0\bigg)\leq 2m-1,
 \\
 \deg\big(4y^2q\big(y^2\big),0\big)=2(2m+1)-2=4m,\qquad m\geq1,
 \\
 \deg\big(4y^2q\big(y^2\big),0\big)\leq0,\qquad m=0.
\end{gather*}
 Thus the rank of~(\ref{pde-sq1}) at zero is $2m$.

 Thus we have shown that by a quadratic substitution
 we can reduce a singularity of a half-integer rank to a singularity of integer rank. The resulting equation~(\ref{pde-sq1}) is in addition invariant with respect to the substitution $y\to-y$ and the rank of the singularity is even.

 Note that our definition of rank $\frac12$ has been chosen so that
 the above quadratic reduction works for all half-integer ranks.

 \subsection{Simplifying the equation}

 Suppose that the equation~(\ref{pde}) has a singular point at $z_0$.
 Obviously, we have 3 exclusive possibilities:
 \begin{enumerate}\itemsep=-.5ex
 \item[(0)] $\arank(A,z_0)\leq1$;
 \item[(1)] $\arank(A,z_0)$ is an integer and $\geq2$;
 \item[(2)] $\arank(A,z_0)$ is a half-integer and $\geq\frac32$.
\end{enumerate}
 We~would like to simplify the equation around this singularity by
 sandwiching with $\e^r$, where $r$ is given by~(\ref{tra1.}) or~(\ref{tra2.}). The transformed operator will be, as usual, denoted $\tilde A$.
 We~will see that the simplification will be quite different
 depending on Case (1) and (2). (Case (0) is simple enough, therefore we do not discuss it in the following proposition).

\begin{Proposition}\label{ground}\qquad
\begin{enumerate}\itemsep=0pt\item[$(1)$]
 If $\arank(A,z_0)$ is an integer and $\geq2$, then there exist exactly two transformations
 such that
 \begin{gather}
 \arank(A,z_0)= \rank(\tilde A,z_0)= \deg(\tilde p,z_0)\geq\deg(\tilde q,z_0),\qquad
 z_0\in\cc,\nonumber
 \\
 \label{(1)}
 \arank(A,\infty)= \rank(\tilde A,\infty)= \deg(\tilde p,\infty)+2\geq\deg(\tilde q,\infty)+4,\qquad z_0=\infty.
 \end{gather}
\item[$(2)$]
 If $\arank(A,z_0)$ is a half-integer and $\geq\frac32$, then there exists a unique transformation such that
 \begin{gather}
\deg(\tilde p,z_0)\leq0\quad\ \text{and}\quad\ \arank(A,z_0)= \rank(\tilde A,z_0)=
 \frac12\deg(\tilde q,z_0),\quad z_0\in\cc,\nonumber
 \\
\label{(2)}
\deg(\tilde p,\infty)\leq-2\quad \text{and}\quad \arank(A,\infty)=
 \rank(\tilde A,\infty)= \frac12\deg(\tilde q,\infty)+2,\quad z_0=\infty.
 \end{gather}
\end{enumerate} \end{Proposition}

\begin{proof}
 Without loss of generality we can assume that
 the singularity is at $0$.
 We~will use the identity\vspace{-2ex}
 \begin{gather}
 \tilde q(z) =\!\!\sum_{n=-2m}^\infty
 z^n\Bigg(\sum_{k=\max(-m,n+1)}^{\min(-1,n+m)}\!\!\!\!\!\!\!\!\!\! w_{k+1}w_{n-k+1}
+(n+1)w_{n+2}+\!\!\! \sum_{k=-m}^{\min(-1,n+m)}\!\!\!\!\!\!\!\!\!\! w_{k+1}p_{n-k}+q_n\Bigg).
\label{recu}
\end{gather}

 We~will apply one of the following three transformations, denoted I, II and III.

{\it Transformation I.} Suppose that the rank of the initial equation is $m-\frac12$, $m=2,3,\dots$.
Then $\deg (q,0)=2m-1$ and\vspace{-1.5ex}
\begin{gather*}
p(z)=\sum_{j=-m+1}^\infty p_jz^j.\vspace{-1ex}
\end{gather*}
 We~choose $r(z)$ such that\vspace{-1.5ex}
\begin{gather*}
r'(z)=-\frac12\sum_{j=-m+1}^{-1} p_jz^j.\vspace{-1ex}
\end{gather*}
 Then $\deg(\tilde p,0)\leq0$ and $\deg(\tilde q,0)=\deg(q,0)$.
The transformed equation satisfies~(\ref{(2)}).

For transformations II and III we suppose that the
rank of the initial equation is $m=2,3,\dots$.

{\it Transformation II.} Assume that
 \begin{gather*}
 p_{-m}^2\neq 4q_{-2m}.
 \end{gather*}
 Let $w_{-m+1}$ be one of two solutions of
 \begin{gather}\label{qqq}
0= w_{-m+1}^2+w_{-m+1}p_{-m}+q_{-2m}.
\end{gather}
 Then $\tilde q_{-2m}=0$.
 Equating $\tilde q_{-m+j}=0$ for $j=-m+1,\dots,-1$ we obtain from~(\ref{recu})
 the recurrence relations\vspace{-1ex}
\begin{gather*}
0= w_{j+1}(2w_{-m+1}+p_{-m})
 \\ \hphantom{0=}
 {}
 +\sum_{k=-m+1}^{j-1}w_{k+1}w_{-m+j-k+1}+(-m+j+1)w_{-m+j+2}+\sum_{k=-m}^{j-1}w_{k+1}p_{-m+j-k}+q_{-m+j}.
\end{gather*}
Using
\begin{gather*}
2w_{-m+1}+p_{-m}=\sqrt{p_{-m}^2-4q_{-2m}}\neq0
\end{gather*}
we can solve the recurrence relations. The transformed equation
satisfies~(\ref{(1)}).

 {\it Transformation III.} If
\begin{gather*}
 p_{-m}^2= 4q_{-2m},
 \end{gather*}
 then we sandwich with $\e^r$, where $r(z)=\frac{p_{-m}z^{-m+1}}{2(-m+1)}$. The transformed operator $\tilde A$ has $\deg(\tilde p,0)\leq m-1$ and $\deg(\tilde q,0)\leq 2m-1$. Thus after the transformation
 $\rank\big(0,\tilde A\big)\leq m-\frac12$. If the resulting rank is half-integer, then we apply I and stop. If the resulting rank is an integer, we apply II and stop, or III and we iterate.

 We~have thus two possibilities:
 \begin{enumerate}\itemsep=-1pt
\setlength{\leftskip}{0.65cm}
\item[{\it Case} (1)] First a finite number of III, and then II. At the end we obtain equation satisfying~(\ref{(1)}). Steps III are uniquely determined and II has two possibilities.

\item[{\it Case} (2)] First a finite number of III, and then I obtaining~(\ref{(2)}). All steps are uniquely determined.
\hfill \qed
\end{enumerate}\renewcommand{\qed}{}
 \end{proof}

We~will say that the operator $A$ has a {\em grounded form at $z_0$}
if $ \deg(p,z_0)\geq\deg( q,z_0)$.
If $z_0$ is Fuchsian, then this is equivalent to one of the indices being $0$.

It follows from Proposition~\ref{ground} that if $\arank(A,z_0)$ is an integer or $\frac12$, then
the equation can be brought to a grounded form at $z_0$.

 \subsection{Solutions in terms of formal power series}

 We~consider the equation~(\ref{pde}) and try to solve it
 in terms of a nontrivial, not necessarily convergent power series
 \begin{gather*}
 \sum_{k=0}^\infty v_kz^k.
\end{gather*}
As we will see, this is not always possible.

\begin{Proposition}\label{series1}
 Set $m:=\deg(p,0)$ and $l:=\deg(q,0)$.
 \begin{enumerate}\label{item}\itemsep=0pt
 \item[$1.$] If $m\leq0$ and $l\leq0$, then for any $v_0$, $v_1$ there exists a power series solution.
 \item[$2.$] If $m\leq1$ and $l\leq2$, then there are no power series solutions unless for some $n=1,2,\dots$ we have
 \begin{gather}
 n(n-1)+np_{-1}+q_{-2}=0.\label{pwp}
 \end{gather}
 \item[$3.$]
 If $m\geq2$ and $l\leq m$, then for any $v_0$ there is a unique power series solution.
 \item[$4.$]
 If $m\geq2$ and $l=m+1$, then there is no power series solution unless for some $n=1,2,\dots$
 \begin{gather*}
 p_{-m}n+q_{-m-1}=0.
 \end{gather*}
 \item[$5.$]
 If $m\geq1$ and $l\geq m+2$, then there are no power series solutions.
 \end{enumerate}
 \end{Proposition}

 \begin{proof} By equating
 the terms at $z^n$ to zero in
 \begin{gather*}
 \bigg(\partial_z^2+\sum_{j=-m}^\infty p_j z^j\partial_z+\sum_{j=-l}^\infty q_j z^j\bigg) \sum_{k=0}^\infty v_k z^k=0
 \end{gather*}
 we obtain the following equations:
\begin{gather*}
 (n+2)(n+1)v_{n+2}+\sum_{k=1}^{n+1+m}p_{n+1-k}kv_k +\sum_{k=0}^{n+l}q_{n-k}v_k=0,\qquad n\in\zz.
 \end{gather*}

Let us prove (5). For $j:=n+l=0,\dots,l-m-1$ we obtain the equations
\begin{gather*}
0=q_{-l}v_0,
\\
\cdots\cdots\cdots\cdots\cdots\cdots\cdots\cdots
\\
0=q_{-l}v_j+\cdots+q_{-l+j}v_0,
\\
\cdots\cdots\cdots\cdots\cdots\cdots\cdots\cdots\cdots\cdots
\\
0=q_{-l}v_{l-m-1}+\cdots+q_{-m-1}v_0.
 \end{gather*}
 Thus $0=v_0=\cdots=v_{l-m-1}$.

 If $j\geq l-m$, there are additional terms coming from the 1st order derivative. $v_k$ with the highest $k$ contained in such a term has $k=m-l+j+1$. But because of $l\geq m+2$ we have $k\leq j-1$. Thus $v_k=0$ by one of previous recursion steps.

 If $j\geq l$, there is an additional term coming from the 2nd order derivative, involving $v_k$ with $k=-l+j+2$. But $l\geq3$ implies $k\leq j-1$.
 Again, this $v_k=0$ by one of previous recursion steps.

 Let us prove (4). We~have recursion relations
 \begin{gather*}
 0=q_{-l}v_0,
 \\
 0=(p_{-l+1}+q_{-l})v_1+q_{-l+1}v_0,
 \\
 \cdots\cdots\cdots\cdots\cdots\cdots\cdots\cdots\cdots\cdots
 \\
 0=(jp_{-l+1}+q_{-l})v_j+\cdots,
 \end{gather*}
 where dots denote terms depending on $v_{j-1},\dots,v_0$.
 If~(\ref{pwp}) has no solutions, then we can solve the recurrence obtaining
 $0=v_0=v_1=\cdots$.

 Let us prove (3). We~have the recursion relations
 \begin{gather}
 0=p_{-m}v_1+q_{-l}v_0,\nonumber
 \\
 0=p_{-m}2v_2+p_{-m+1}v_1+q_{-m}v_1+q_{-m+1}v_0,\label{w2}
 \\
 \cdots\cdots\cdots\cdots\cdots\cdots\cdots\cdots\cdots\cdots\cdots\cdots\cdots\cdots\nonumber
 \\
 0=jp_{-m}v_j+\cdots\nonumber,
 \end{gather}
 where $\cdots$ involve $v_0,\dots,v_{j-1}$.
If $m=2$, there is an additional term $2v_0$ in~(\ref{w2}).
 Note that $p_{-m}\neq0$.
 Hence, for any $v_0$ we can solve the recurrence obtaining $v_1,v_2,\dots$.

(2) follows immediately from the well-known theory of solutions around a Fuchsian singular point.

 (1) is the well-known fact about the Cauchy problem in the regular case. \end{proof}

\subsection{Thom\'e solutions}

By the so-called Frobenius method, if $z_0$ is a Fuchsian singularity
and $\rho_1,\rho_2$ are its indices such that
$\rho_1-\rho_2\not\in\zz$, then solutions of the equation~(\ref{pde})
are spanned by two convergent power series indexed by $i=1,2$ with $v_{i,0}\neq0$:
\begin{gather*}
\sum_{j=0}^\infty v_{i,j}(z-z_0)^{j+\rho_i},\qquad z_0\in \cc,
\\
\sum_{j=-\infty}^0 v_{i,j}z^{j-\rho_i},\qquad z_0=\infty.
\end{gather*}
If $\rho_1-\rho_2\in\zz$ this is not always true. One can then assume that $\rho_1-\rho_2\geq0$. There exists one solution as above with $i=1$
and the second has the form
\begin{gather*}
\sum_{j=0}^\infty v_{2,j}(z-z_0)^{\rho_2+j}+\log z
\sum_{j=0}^\infty v_{1,j}(z-z_0)^{\rho_1+j},\qquad z_0\in \cc,
\\
\sum_{j=-\infty}^0 v_{2,j}z^{-\rho_2+j} +\log z
\sum_{j=-\infty}^0 v_{1,j}z^{-\rho_1+j},\qquad z_0=\infty.
 \end{gather*}

 If the singular point is not Fuchsian, then we can also look for solutions in a similar form, however the resulting power series are usually no longer convergent. One obtains the so-called {\em Thom\'e solutions}. Note that in some way the situation is simpler, because we do not have the logarithmic case. On the other hand, half-integer ranks need to be treated separately and lead to power series in $\sqrt{z}$.

\begin{Proposition} \label{thome}
Let $z_0$ be a singular point of $A$ with $\rank(A,z_0)=n$.
 Then there exist two formal solutions of $A$, indexed by $i=1,2$.
 \begin{enumerate} \label{thome1}\itemsep=0pt
\item[$1.$] If $\arank(A,z_0)$ is an integer $\geq2$,
 they have the form
 \begin{gather*}
 \exp\Bigg(\sum_{k=-n+1}^{-1}\frac{(z-z_0)^k}{k}w_{i,k}\Bigg)
 \sum_{j=0}^\infty v_{i,j}(z-z_0)^{w_{i,0}+j},\qquad z_0\in \cc,
 \\
 \exp\Bigg({-}\sum_{k=1}^{n-1}\frac{z^k}{k}w_{i,k}\Bigg)
\sum_{j=-\infty}^0 v_{i,j}z^{-w_{i,0}+j},\qquad z_0=\infty.
 \end{gather*}
 \item[$2.$]
 Let $\arank(A,z_0)$ be a half integer $\geq\frac32$.
 The two formal solutions have the form
 \begin{gather*}
 \exp\Bigg(\mathop{{\sum}'}_{k=-n+1}^{-\frac12}\frac{(z-z_0)^k}{k}w_{i,k} \Bigg)
 \mathop{{\sum}'}_{j=0}^\infty v_{i,j}(z-z_0)^{w_{i,0}+j},\qquad z_0\in \cc,
 \\
 \exp\Bigg({-}\mathop{{\sum}'}_{k=\frac12}^{n-1}\frac{z^k}{k}w_{i,k}\Bigg)
\mathop{{\sum}'}_{j=-\infty}^0 v_{i,j}z^{-w_{i,0}+j},\qquad z_0=\infty.
 \end{gather*}
 Here $\mathop{{\sum}'}$ denotes the sum where the index $k$ within its range runs over
 both integers and half-integers. We~have
 \beq w_{0,k}=(-1)^{2k}w_{1,k},\qquad v_{0,k}=(-1)^{2k}v_{1,k}.
\eeq
 \end{enumerate}
 \end{Proposition}

\begin{proof} For simplicity, assume that $z_0=0$.

Suppose that $\arank(A,0)$ is an integer $\geq2$.
By Proposition~\ref{ground}(1) we can transform the equation to
a grounded form in two distinct ways. By Proposition~\ref{series1}(3), the grounded form has a solution in terms of the power series.

If $\arank(A,0)=m+\frac12$ is a half-integer $\geq \frac32$, first we reduce the equation to
the form with a half-integer rank, see Proposition~\ref{ground}(2). Then we apply the quadratic transformation, as described in Section~\ref{Half-integer rank}.
We~obtain
an even equation in $\sqrt{z}$ of the rank $2m$, with $\tilde p_{-2m}=0$ and $\tilde q_{-4m}\neq0$.
We~already know that it has a solution of the form described in 1:
\begin{gather}\label{qqq1}
\exp\Bigg(\sum_{k=-2m+1}^{-1}\frac{\big(\sqrt{z}\big)^k}{k}\tilde w_{i,k}\Bigg)
\sum_{j=0}^\infty \tilde v_{j}\big(\sqrt{z}\big)^{\tilde w_{i,0}+j}.
\end{gather}
Using $2m\geq2$ and~(\ref{qqq}) we obtain
\begin{gather}
\tilde w_{-2m+1}=\sqrt{-\tilde q_{-4m}}\neq0.\label{qqq2}
\end{gather}

Clearly,~(\ref{qqq1}) with $\sqrt{z}$ replaced by $-\sqrt{z}$ is also a solution. By~(\ref{qqq2}) both solutions are not proportional to one another. This proves~(\ref{thome1}).
\end{proof}

Note that Proposition~\ref{thome} is also true in the Fuchsian case,
except that for $\arank(A,z_0)=1$, one has to make an obvious modification
in the logarithmic case, and for $\arank(A,z_0)=\frac12$~(\ref{thome1}) does not have to be true.

In~the Fuchsian case $w_{i,0}$, $i=1,2$, coincide with the indices of $z_0$.
In~what follows, the numbers $w_{i,0}$, $i=1,2$, will be called {\em indices of $z_0$} in the general case as well. We~also introduce the alternative notation
\begin{gather}
\rho_{z_0,i}:=w_{i,0},\qquad i=1,2.\label{rho6}
\end{gather}

\begin{Proposition} \label{fuchsi}
 Let $z_0$ be a singularity of $A$.
 \begin{enumerate}\itemsep=0pt
 \item[$1.$] We~have
 \begin{gather}
\label{rho1}
\rho_{z_0,1}+\rho_{z_0,2}=-p_{z_0,-1}+\arank(A,z_0),\qquad z_0\in\cc,
\\[1ex]
\label{rho2}
\rho_{\infty,1}+\rho_{\infty,2}=p_{\infty,-1}-2+\arank(A,\infty),\qquad z_0=\infty.
 \end{gather}
 \item[$2.$] If $\arank(A,z_0)\in\big\{\frac32,\frac52,\dots\big\}$, then
 \begin{gather*}
\rho_{z_0,1}=\rho_{z_0,2}=\frac12\big({-}p_{z_0,-1}+\arank(A,z_0)\big),\qquad z_0\in\cc,
\\[1ex]
\rho_{\infty,1}=\rho_{\infty,2}=\frac12\big(p_{\infty,-1}-2+\arank(A,\infty)\big),\qquad z_0=\infty.
 \end{gather*}
\item[$3.$] If $z_0\in\cc$ is grounded, then
 \begin{gather*}
 \{\rho_{z_0,1},\rho_{z_0,2}\}=\big\{0,-p_{z_0,-1}+\arank(A,z_0)\big\} .
 \end{gather*}
 \end{enumerate}
\end{Proposition}

\begin{proof} (1) Without loss of generality we can assume that $z_0=0$.
When we apply sandwiching with $\e^r$, where $r(z)$ has the
form~(\ref{tra1.}),
then $p_{-1}$ and $\rho_i$ are transformed into $p_{-1}-2w_0$ and~\mbox{$\rho_i+w_0$}. This does not affect the identity~(\ref{rho1}).

{\samepage
Assume first that $\arank(A,0)$ is an integer. Then by a sandwiching transformation we can reduce the equation to a grounded form at $0$.
The first $m$ recurrence relations of~(\ref{recu}) read then
\begin{gather*}
0 = \sum_{k=-m}^{n+m} w_{k+1}(w_{n-k+1}+p_{n-k}),\qquad n=-2m,\dots,-m-2,
 \\
 0 = \sum_{k=-m}^{-1} w_{k+1}(w_{-m-k}+p_{-m-1-k})-mw_{-m+1}.
\end{gather*}}
Apart from the solution $w_{k+1}=0$, $k=-m,\dots,-1$, this is solved by
\begin{gather*}
 w_{k+1}:=-p_k,\qquad k=-m,\dots,-2,
 \\
 w_{0}:=-p_{-1}+m.
 \end{gather*}
Thus $\{\rho_1,\rho_2\}=\{0,-p_{-1}+m\}$. Hence~(\ref{rho1}) is satisfied.

 Assume next that $\arank(A,0)$ is a half integer equal $m+\frac12$.
 After an appropriate sandwiching transformation we can assume that
 $\arank(A,0)= \rank(A,0)$. Then we can apply the quadratic transformation~(\ref{quadra}) obtaining an equation
 \begin{gather*}
 \tilde A=\partial_y^2+\tilde p(y)\partial_y+\tilde q(y).
 \end{gather*}
 Let $\tilde\rho_i$, $i=1,2$ be the indices of $\tilde A$ at zero.
 Now $\rank\big(\tilde A,0\big)=2m$, which is an integer. Hence we can apply the formula~(\ref{rho1})
 \begin{gather*}
 \tilde\rho_1+\tilde\rho_2=-\tilde p_{-1}+2m.
 \end{gather*}
 But the indices of $A$ at $0$ are $\rho_i=\frac{\tilde\rho_i}{2}$, $i=1,2$ and $\tilde p_{-1}=2p_{-1}-1$. Hence
 \begin{gather*}
 \rho_1+\rho_2=\frac12( \tilde \rho_1+\tilde\rho_2) = \frac12(-2p_{-1}+1+2m)=-p_{-1}+m+\frac12.
 \end{gather*}
 This ends the proof~(\ref{rho1}).

 To prove~(\ref{rho2}) we apply the transformation $w=\frac1z$. We~note that $\arank(A,\infty)$ is transformed to~$\arank(A,0)$ and $\rho_{\infty,i}$ are transformed to $\rho_{0,i}$, $i=1,2$. Besides, $p_{\infty,-1}$
 is transformed to $-p_{0,-1}+2$.

 (1) and~(\ref{thome1}) imply (2).

 (1) and~(\ref{item}) of Proposition~\ref{series1} imply (3).
 \end{proof}

\subsection{Nonlogarithmic singularities}

Let $z_0$ be a singular point of the equation~(\ref{pde}). We~say that $z_0$
is {\em nonlogarithmic} or {\em apparent} iff all solutions of the equation are meromorphic around this singularity. If $z_0$ is a nonlogarithmic Fuchsian singular point, then both its indices are integers.

The following proposition shows how to deform a given equation
so that one obtains an addi\-tional nonlogarithmic singularity with indices $0,2$.
This deformation depends on two parame\-ters~$\lambda$, $\mu$: the
additional singularity is located at $\lambda$,
and solutions of the deformed equation satisfy $\mu v(\lambda)=v'(\lambda)$.
It will play the central role in the derivation of Painlev\'e equations from Heun class equations in Section~\ref{section: Painleve}.

\begin{Proposition}\label{appa}
Let $\sigma$, $\tau$, $\eta$ be analytic at $\lambda$ and $\sigma(\lambda)\neq0$. Let $\mu\in\cc$.
Then all solution of the equation given by
\begin{gather}
\sigma(z)\partial_z^2+
\bigg(\tau(z)-\frac{\sigma(z)}{z-\lambda}\bigg)\partial_z
+\eta(z)-\eta(\lambda)-\mu^2\sigma(\lambda)-\mu\big(\tau(\lambda)-\sigma'(\lambda)\big)
+\frac{\mu\sigma(\lambda)}{z-\lambda}\label{h-bis-}
\end{gather}
are analytic at $\lambda$. Thus the equation given by~\eqref{h-bis-}
has a nonlogarithmic singularity at $z=\lambda$. The singularity is Fuchsian with indices $0$, $2$.
\end{Proposition}

\begin{proof}
We~look for a solution analytic around $\lambda$:
\begin{gather*}
v(z)=\sum_{n=0}^\infty v_n(z-\lambda)^n.
\end{gather*}
We~obtain
\begin{gather*}
\big({-}\sigma(\lambda)v_1+\mu\sigma(\lambda)v_0\big)(z-\lambda)^0 +\big(\sigma(\lambda)2v_2+\big(\tau(\lambda)-\sigma'(\lambda)\big)v_1-\sigma(\lambda)2v_2 -\big(\mu\tau(\lambda)
\\ \qquad
{}-\mu\sigma'(\lambda)+\mu^2\sigma(\lambda)\big)v_0 +\mu\sigma(\lambda)v_1\big)(z-\lambda)^1
\\ \qquad
{} +\sum_{n=3}^\infty\big(\sigma(\lambda)n(n-2)v_n+\cdots\big)(z-\lambda)^{n-1}=0.
\end{gather*}
We~have $\sigma(\lambda)\neq0$. Hence the first line implies
$v_1=\mu v_0$. Then the second line is identically zero, and $v_2$ is left unspecified. The next terms yield recurrence relations for $v_n$, $n=3,\dots$.
\end{proof}

\section{Equations with rational coefficients}\label{s2}

\subsection[The $M_n$ class and the grounded $M_n$ class]
{The $\boldsymbol{M_n}$ class and the grounded $\boldsymbol{M_n}$ class}

Consider an equation given by the operator
\begin{gather}
A:= \partial_z^2+p(z)\partial_z+q(z),
\label{pde-}
\end{gather}
where $p(z)$, $q(z)$ are rational functions.

If $z_1,\dots,z_k\in\cc\cup\{\infty\}$ are its singularities and their ranks are
$m_1,\dots,m_k$, then we will say that the equation~(\ref{pde-}) is of type $(m_1m_2\cdots m_k)$

 Often we will need a more precise
 description of~(\ref{pde-}), which gives information what is the rank of the singularity at $\infty$. We~will then put it at the end of the sequence, so that $z_k=\infty$, and precede it with a semicolon. We~will write that~(\ref{pde-}) is of type $(m_1m_2\cdots m_{k-1};m_\infty)$.

 By writing $\underline{m_i}$ instead of $m_i$ we will mean
the rounded rank. We~will use it especially often for $1$. Thus $\underline{1}$ means a Fuchsian singularity $\big(1$ or $\frac12\big)$.

Every equation having no more than $n+1$ singular points in the Riemann sphere, all of them Fuchsian
and at most $n$ finite, is given by an operator of the form
\begin{gather}
 \partial_z^2+\sum_{j=1}^{n}\frac{a_j}{z-z_j}\partial_z
 +\sum_{j=1}^{n}\frac{b_j}{z-z_j}
 +\sum_{j=1}^{n}\frac{c_j}{(z-z_j)^2}\label{req1-}
\qquad \text{with}\quad
\sum_{j=1}^{n}b_j=0,
\end{gather}
where $z_1,\dots,z_{n}$ are distinct points in $\cc$.
The family of equations~(\ref{req1-}) will be called {\em the $M_n$ type}.
 The corresponding symbol is
$\big(\!\underset{n-1\text{ times }}{\underline{1}\cdots\underline{1}};\underline{1}\big)$.

Each finite singularity has at least one index equal $0$ if and only if
$c_1=\cdots=c_{n-1}=0.$
Thus such equations are given by operators
\begin{gather}
 \partial_z^2+\sum_{j=1}^{n}\frac{a_j}{z-z_j}\partial_z
 +\sum_{j=1}^{n}\frac{b_j}{z-z_j}
\qquad \text{with}\quad
 \sum_{j=1}^{n}b_j=0.\label{req2}
 \end{gather}
The family of equations given by~(\ref{req2}) will be called
{\em the grounded $M_n$ type}.

\begin{Proposition}
 By sandwiching with powers, as in~\eqref{tra1}, we can always transform an $M_n$ type equation into a grounded $M_n$ type equation.
 \end{Proposition}

We~say that a differential equation belongs to the {\em $M_n$ class} if it is given by
\begin{gather}
 \partial_z^2+\frac{\tau(z)}{\sigma(z)}\partial_z+\frac{\xi(z)}{\sigma(z)^2},
 \label{req3}
 \end{gather}
where $\sigma$, $\tau$, $\xi$ are polynomials satisfying
\begin{gather}\label{req3/}
 \sigma\neq0,\qquad \deg\sigma\leq n,\qquad\deg\tau\leq n-1,\qquad
\deg\xi\leq2n-2.
\end{gather}
We~will often use the shorthand
\begin{gather*}
 \eta(z):=\frac{\xi(z)}{\sigma(z)},\label{short}
 \end{gather*}
where $\eta$ does not have to be a polynomial.

We~say that a differential equation belongs to the {\em grounded $M_n$ class} if it is given by
\begin{gather}
\partial_z^2+\frac{\tau(z)}{\sigma(z)}\partial_z+\frac{\eta(z)}{\sigma(z)},
\label{req3-}
\end{gather}
where $\sigma$, $\tau$, $\eta$ are polynomials satisfying
\begin{gather*}
\sigma\neq0,\qquad\deg\sigma\leq n,\qquad\deg\tau\leq n-1,\qquad
\deg\eta\leq n-2.
\end{gather*}

The name ``the $M_n$ class'' is borrowed from Lay--Slavyanov \cite{Sl_Lay}.

\begin{Proposition}\qquad
\begin{enumerate}\itemsep=0pt
\item[$1.$]
The $M_n$ type is contained in the $M_n$ class. An equation of the $M_n$ class is of the $M_n$ type iff $\sigma$ possesses $n$ distinct roots.
\item[$2.$]
The grounded $M_n$ type is contained in the grounded~$M_n$ class. An equation of the grounded $M_n$ class is of the grounded $M_n$ type iff $\sigma$ possesses $n$ distinct roots.
\end{enumerate} \end{Proposition}

\begin{proof} Let us prove (1).
 Consider~(\ref{req1-}). Set \begin{gather*}\sigma(z):=(z-z_1)\cdots(z-z_{n}).\end{gather*}
 Then clearly $\sigma$ is a nonzero polynomial with $n$ distinct roots. We~easily see that~(\ref{req1-}) can be rewritten as~(\ref{req3}) with~(\ref{req3/}) satisfied.

 Conversely, consider~(\ref{req3}) such that $\sigma$ has $n$ distinct roots, namely, $z_1,\dots,z_{n}$. Then we can decompose $\frac{\tau(z)}{\sigma(z)}$ and $\frac{\xi(z)}{\sigma(z)^2}$ into simple fractions, obtaining~(\ref{req1-}).

 The proof of (2) is analogous.
 \end{proof}

 \begin{Proposition}
 Let $(z_1,\dots,z_k)$ be the singularities of an equation $A$ of the $M_n$ class. Then
 \begin{gather*}
 \lceil\rank\rceil(A,z_1)+\cdots+\lceil\rank\rceil(A,z_k)
 \leq n+1.
 \end{gather*}
\end{Proposition}

 \begin{proof} Without loss of generality we can assume that $z_1,\dots,z_{k-1}\in\cc$ and $z_k=\infty$. Let \begin{gather*}\sigma(z)=(z-z_1)^{m_1}\cdots(z-z_{k-1})^{m_{k-1}}\end{gather*} with distinct $z_i$'s. Then $z_1,\dots,z_{k-1}$ are the finite singular points.
 Clearly, $\lceil\rank\rceil(A,z_i)\leq m_i$.
 Therefore,
 \begin{gather}\label{we1}
 \lceil\rank\rceil(A,z_1)+\cdots + \lceil\rank\rceil(A,z_{k-1})\leq\deg\sigma.
 \end{gather}
 Now
\begin{gather}\nonumber
\lceil\rank\rceil(A,\infty)=\bigg\lceil\!\max\bigg(\deg\tau-\deg\sigma+2, \frac12(\deg\xi-2\deg\sigma)+2\bigg)\bigg\rceil
\\ \hphantom{\lceil\rank\rceil(A,\infty)}
{}\leq \bigg\lceil\!\max\bigg(n-1-\deg\sigma+2,\frac12(2n-2-2\deg\sigma)+2\bigg)\!\bigg\rceil= n+1-\deg\sigma.
\label{we2}
 \end{gather}
Then we sum~(\ref{we1}) and~(\ref{we2}).
 \end{proof}

We~will often represent $M_n$ class equations by operators obtained by multiplying~(\ref{req3}) or~(\ref{req3-}) from the right by $\sigma(z)$:
\begin{gather}
 \sigma(z) \partial_z^2+\tau(z)\partial_z+\eta(z).
 \label{req3a}
 \end{gather}

Obviously, $M_n$ class equations and operators are defined by coefficients of the polynomi\-als~$\sigma$, $\tau$, $\xi$. Therefore, they form a complex manifold parameterized by
\begin{gather}\label{opendense1}
 \big(\cc^{n}\backslash\{0\}\big)\times\cc^{3n-3}.\end{gather}
The condition saying that $\sigma$ has $n$ distinct roots defines an open dense subset in~(\ref{opendense1}).
Thus the $M_n$ type is an open dense subset of the $M_n$ class.
Hence the $M_n$ class consists of the $M_n$ type and its limiting points in the topology of~(\ref{opendense1}). These limiting points are traditionally called {\em confluent cases}.

Similarly,
 grounded $M_n$ class equations and operators are defined by $\sigma$, $\tau$, $\eta$. Therefore, they form a complex manifold parameterized by
\begin{gather*}
 \big(\cc^{n}\backslash\{0\}\big)\times\cc^{2n-3}.
 \end{gather*}
Clearly,
 the grounded $M_n$ type is an open dense subset of the grounded $M_n$ class.
One can say that the grounded $M_n$ class consists of the grounded $M_n$ type and its confluent cases.

\subsection{Generalized Fuchs relation}

Recall that for any singular point $z_0$ of an equation $A$ we defined
its two indices $\rho_{z_0,1}$ and $\rho_{z_0,2}$. For Fuchsian singularities they were defined in~(\ref{rho4}),~(\ref{rho5}) and for non-Fuchsian singularities in~(\ref{rho6}). If all singularities are regular then the well-known Fuchs relation says that the sum of all indices equals the number of singularities minus $2$. In~the following proposition we describe its generalization which is valid if some of
the singularities are non-Fuchsian.

\begin{Proposition}
 Let $z_1,\dots,z_k$ be the singular points of an equation $A$. Then
 \begin{gather}
 \sum_{j=1}^k(\rho_{z_j,1}+\rho_{z_j,2})=\sum_{j=1}^k\arank(A,z_j)-2.
 \label{rho3}
 \end{gather}
\end{Proposition}

 \begin{proof}
 Without loss of generality we can assume that $z_k=\infty$.
We~have
\begin{gather}\label{fu1}
\rho_{z_i,1}+\rho_{z_i,2}=-p_{z_j,-1}+\arank(A,z_j),\qquad j=1,\dots,k-1,
\\
\label{fu2}
\rho_{\infty,1}+\rho_{\infty,2}=p_{\infty,-1}-2+\arank(A,\infty),
\\
p_{\infty,-1}=\sum_{j=1}^{k-1}p_{z_j,-1}\nonumber,
\end{gather}
where~(\ref{fu1}) and~(\ref{fu2}) follows from~Proposition~\ref{fuchsi}.
 Summing up the above three relations we obtain~(\ref{rho3}).
 \end{proof}

\subsection{Riemann class equations}

The simplest nontrivial $M_n$ class is the $M_2$ class.
We~call it the {\em Riemann class} since it consists of the
Riemann equation with one singularity at $\infty$ and its confluent cases.
Thus Riemann class operators have the form~(\ref{req3}), where
\begin{gather*}
\sigma\neq0,\qquad\deg\sigma\leq 2,\qquad\deg\tau\leq 1,\qquad
\deg\xi\leq 2.
\end{gather*}
The {\em grounded Riemann class} operators
has the form~(\ref{req3-}), where
\begin{gather*}
\sigma\neq0,\qquad\deg\sigma\leq 2,\qquad\deg\tau\leq 1,\qquad
\eta\quad\text{is a number.}
\end{gather*}

Note that grounded Riemann class equations appear in the literature very often. They are often called {\em hypergeometric type equations},
see \cite{De,DeWr,NU}.

It is well known that by a division by a constant, transformations $z\mapsto az+b$, sandwiching with powers and exponentials all Riemann class operators can be transformed into one of the following types:
\begin{center}\renewcommand{\arraystretch}{1.2}
{\small\begin{tabular}{lcl}
 \hline
the ${}_2F_1$ operator& $(\underline{1}\underline{1};\underline{1})$&$z(1-z)\p_z^2+\big(c-(a+b+1)z\big)\p_z-ab$
\\ \hline
the ${}_2F_0$ operator& $(2;\underline{1})$&$z^2\p_z^2+\big({-}1+(a+b+1)z\big)\p_z+ab$
\\ \hline
the ${}_1F_1$ operator& $(\underline{1};2)$&$z\p_z^2+(c-z)\p_z-a$
\\ \hline
the ${}_0F_1$ operator& $(\underline{1};\frac32)$&$z\p_z^2+c\p_z-1$
\\ \hline
the Hermite operator& $ (;3)$&$\p_z^2-2z\p_z-2a$
\\ \hline
the Airy operator &$ (;\frac52)$&$\p_z^2+z$
\\ \hline
the Euler II operator & $(\underline{1};\underline{1})$&$z^2\p_z^2+cz\p_z$
\\ \hline
the Euler I operator & $(\underline{1};\underline{1})$&$z\p_z^2+c\p_z$
\\ \hline
the 1d Helmholtz operator & $(;2)$&$\p_z^2+1$
\\ \hline
the 1d Laplace operator & $(;\underline{1})$&$\p_z^2$
\\
\hline
\end{tabular}}
\end{center}

Let us make some remarks.
\begin{enumerate}\itemsep=-2pt
\item The Euler II and Euler I operators yield the same equations.
\item The last four equations from the table can be solved in elementary functions.
\item
In~this table, only the Airy equation cannot be brought to the grounded form.
\item
When we take into account the transformation $z\mapsto z^{-1}$, then the
types $(2;\underline{1})$ and $(\underline{1};2)$ are equivalent.
\item There are more relations between various types when we consider more complicated transformations.
 \end{enumerate}

\subsection{Heun class equations}

$M_3$ type equations were studied by Heun in \cite{Heun}. Therefore, it is natural to call the $M_3$ type the {\em Heun type}.
Consequently, the $M_3$ class will be called the {\em Heun class}.
The grounded $M_3$ class will be called the {\em grounded Heun class}.

Our terminology is consistent with \cite{Ronveaux,Sl_Lay}. However, in some publications the name {\em Heun class} is used to denote what we call the {\em grounded Heun class}, see, e.g., \cite{FID}.

We~will represent Heun class equations by Heun class operators. More precisely,
we will say that
\begin{gather}
 \sigma(z) \partial_z^2+\tau(z)\partial_z+\eta(z)
 \label{req3aa}
\end{gather}
is a {\em Heun class operator} if
$\eta(z)=\frac{\xi(z)}{\sigma(z)}$ and $\sigma$, $\tau$, $\xi$ are polynomials such that
\begin{gather}
\sigma\neq0,\qquad\deg\sigma\leq 3,\qquad\deg\tau\leq 2,\qquad
\deg\xi\leq 4.
\label{con-heun}
\end{gather}
(\ref{req3aa}) is a {\em grounded Heun class operator}
if $\sigma$, $\tau$, $\eta$ are polynomials such that
\begin{gather*}
\sigma\neq0,\qquad\deg\sigma\leq 3,\qquad\deg\tau\leq 2,\qquad
\deg\eta\leq1.
\end{gather*}
If in addition $\sigma$ has 3 distinct roots, then~(\ref{req3aa}) is a (grounded) Heun type operator.

Clearly, the Heun class and the grounded Heun class are preserved
by transformations $z\mapsto az+b$.

The Heun class is also preserved by sandwiching with powers and exponentials, see
(\ref{tra1}),~(\ref{pde1}) and~(\ref{pde2}).

Heun class operators are invariant with respect to swapping a finite singularity with the infinity.
More precisely, Heun class operators of the form
(\ref{req3aa})
after the transformations
\begin{gather*}
 w=(z-z_0)^{-1},\text{ where $z_0$ is one of finite singular points}
 \end{gather*}
remain in the Heun class.
Indeed, without loss of generality, we can suppose that $z_0=0$. Thus $\sigma(0)=0$, so that $\sigma(z)=z\rho(z)$, where $\rho$ is a polynomial with $\deg\rho\leq2$.
Substitute $w=z^{-1}$, which transforms~(\ref{req3aa}) into
\begin{gather}\label{swap}
 w^3\rho\big(w^{-1}\big)\partial_w^2+w^2\big(2\rho \big(w^{-1}\big)-\tau\big(w^{-1}\big)\big)\partial_w+\eta\big(w^{-1}\big).
\end{gather}
It is easy to see that{\samepage
\bes\begin{gather}
 \tilde\sigma(w):= w^3\rho\big(w^{-1}\big),\\
 \tilde\tau(w):=w^2\big(2\rho \big(w^{-1}\big)-\tau\big(w^{-1}\big)\big),\label{rp2}\\
 \tilde\eta(w):=\eta\big(w^{-1}\big)
\end{gather}\ees
still satisfy the condition~(\ref{con-heun}).}

Note that we do not need to put any prefactor in~(\ref{swap}).
Remarkably, the analogous property does not hold for the $M_n$ classes with $n\neq3$: for them after swapping a finite singularity with~$\infty$ an additional prefactor is needed.

Swapping a finite singularity with $\infty$ is possible also if we want to stay within the grounded Heun class, except that the transformation
$w=(z-z_0)^{-1}$ needs to be followed by sandwiching with a power, that is a transformation~(\ref{tra1}).
Indeed, assume~(\ref{req3a}) is a grounded Heun class operator
and $z_0=0$ is a singularity. Let $\alpha$ satisfy the generalized indicial equation at $z=\infty$:
\begin{gather*}
\frac{\rho''}{2}\alpha(\alpha+1)-\frac{\tau''}{2}\alpha+\eta'=0.
\end{gather*}
Then
\begin{gather}\label{heun9}
w^{-\alpha}\big(\tilde\sigma(w)\partial_w^2+\tilde\tau(w)\partial_w+\tilde\eta(w)\big)w^\alpha
= \tilde\sigma(w)\partial_w^2+\tilde\tau_1(w)\partial_w+\tilde\eta_1(w),
\end{gather}
where
\begin{gather*}
\tilde\tau_1(w):=2(\alpha+1)w^2\rho\big(w^{-1}\big)-w^2\tau\big(w^{-1}\big),
\\
\tilde\eta_1(w):=w\alpha\big((\alpha+1)\rho(0)-\tau(0)\big)+ \alpha\big((\alpha+1)\rho'(0)-\tau'(0)\big)+\eta(0).
 \end{gather*}
Clearly,~(\ref{heun9}) is a grounded Heun class operator.

\subsection{Deformed Heun class equations}

Consider $\sigma$, $\tau$, $\eta$ satisfying
the conditions~(\ref{con-heun}), so that \eqref{req3aa}
is a Heun class operator. Let $\lambda,\mu\in\cc$. The corresponding {\em deformed Heun class operator}
is defined as
\begin{gather}
\sigma(z)\partial_z^2+
\bigg(\tau(z)-\frac{\sigma(z)}{z-\lambda}\bigg)\partial_z
+\eta(z)-\eta(\lambda)-\sigma(\lambda)\mu^2
-\big(\tau(\lambda)-\sigma'(\lambda)\big)\mu
+\frac{\sigma(\lambda)\mu}{z-\lambda}.
\label{h-bis--}
\end{gather}
By Proposition~\ref{appa}, the equation defined by~(\ref{h-bis--}) has a nonlogarithmic singularity at $z=\lambda$ with indices $0,2$. All the remaining finite singularities have the same type (the rank, the indices), as for the original Heun class operator~(\ref{req3aa}).

Thus to every Heun class operator~(\ref{req3aa}) there corresponds a family of deformed Heun class operators~(\ref{h-bis--}) depending on two new parameters: $\lambda$ and $\mu$. Note that one of the parameters of $\eta$ in the original operator~(\ref{h-bis-}) is lost~---~(\ref{h-bis--}) does not depend on the free (zeroth order) term of $\eta$.

The family of deformed Heun class operators is preserved by the same transformations as the family of Heun class operators.
Clearly, it is preserved by $z\mapsto az+b$, division by a constant and sandwiching with powers and exponentials, as described in~(\ref{tra1}),~(\ref{pde1}) and~(\ref{pde2}).

It is also invariant with respect to swapping the singularity at $\infty$ with
finite singularities.
Thus assume that $\sigma(z)=z\rho(z)$.
Then substitution $w=z^{-1}$ transforms~(\ref{h-bis--}) into
\begin{gather*}
\tilde\sigma(w)\partial_w^2+
\bigg(\tilde\tau(w)-\frac{\tilde\sigma(w)}{w-\lambda^{-1}}\bigg)\partial_w
+\tilde\eta(w)-\tilde\eta\big(\lambda^{-1}\big)-\tilde\sigma\big(\lambda^{-1}\big)\mu^2
\\ \hphantom{\tilde\sigma(w)\partial_w^2}
{}-\big(\tilde\tau\big(\lambda^{-1}\big)\tilde\sigma'\big(\lambda^{-1}\big)\big)\mu
+\frac{\tilde\sigma\big(\lambda^{-1}\big)\mu}{w-\lambda^{-1}},
\end{gather*}
where{\samepage
\bes\begin{gather}
\tilde\sigma(w):= w^3\rho\big(w^{-1}\big),
\\
\tilde\tau(w):=w^2\big(3\rho\big(w^{-1}\big)-\tau\big(w^{-1}\big)\big),\label{rr2}
\\
\tilde\eta(w):=\eta\big(w^{-1}\big),
\\
\tilde\mu:=-\lambda^2\mu.
\end{gather}\ees}

Transformation $w=z^{-1}$ transforms a deformed Heun class operator
satisfying $\sigma(0)=0$ into another deformed Heun class operator
satisfying $\tilde\sigma(0)=0$, similarly as for undeformed Heun class operators.
Note however a subtle difference between~(\ref{rp2}) and~(\ref{rr2}).

The following proposition describes mapping properties of the above described
transformation in more detail:

\begin{Proposition}\qquad\label{symmo}
\begin{enumerate}\itemsep=0pt
\item[$1.$]
$\sigma\eta(0)=0 \Leftrightarrow \deg\tilde\sigma\tilde\eta\leq3$.
\item[$2.$]
$\sigma'(0)=0 \Leftrightarrow \deg\tilde\sigma\leq2$.
\item[$3.$]
$\sigma'(0)=0$ and $\tau(0)=0 \Leftrightarrow\deg\tilde\sigma\leq2$ and $\deg\tilde\tau\leq1$.
\item[$4.$]
$\sigma'(0)=0$ and $(\sigma\eta(0)=(\sigma\eta)'(0)=0 \Leftrightarrow
\deg\tilde\sigma\leq2$ and $\deg\tilde\sigma\tilde\eta\leq2$.
\end{enumerate}
\end{Proposition}

\subsection{Classification of Heun class equations}
\label{Classification of Heun class equations}

In~this subsection we discuss two
classifications of Heun class equations and operators.

The first is based
on the rank of singularities. We~classify half-integer and integer ranks separately, except for the rank 1, where we use, as usual, the rounded rank.
Not counting the types reducible to the Riemann class, which are treated as ``trivial'', it partitions
the Heun class into ten types.

There exists also a coarser classification, which uses rounded singularity ranks. It groups the ten nontrivial types of the Heun class into five {\em supertypes}.

 In~the following list we give both classifications of the Heun class:
\begin{itemize}\itemsep=0pt
\item (standard) Heun or $(\un \un \un \un )$.
\item confluent Heun or $(\un \un \underline2)$.
 \begin{itemize}\itemsep=0pt
 \item non-degenerate confluent Heun or $(\un \un 2)$.
 \item degenerate confluent Heun or $\big(\un \un \frac32\big)$.
 \end{itemize}
\item doubly confluent Heun or $(\underline2\underline2)$.\vspace{-1ex}
 \begin{itemize}\itemsep=-1pt
 \item non-degenerate doubly confluent Heun or $(22)$.
 \item degenerate doubly confluent Heun or $\big(\frac322\big)$.
 \item doubly degenerate doubly confluent Heun or $\big(\frac32\frac32\big)$.\vspace{-1.5ex}
 \end{itemize}
 \item biconfluent Heun or $\big(\un \underline3\big)$.\vspace{-1.5ex}
\begin{itemize}\itemsep=-1pt
\item non-degenerate biconfluent Heun or $(\un 3)$.
 \item degenerate biconfluent Heun $\big(\un \frac52\big)$.\vspace{-1ex}
\end{itemize}
 \item triconfluent Heun $(\underline4)$.\vspace{-1.5ex}
\begin{itemize}\itemsep=-1pt
\item non-degenerate triconfluent Heun $(4)$.
 \item degenerate triconfluent Heun $\big(\frac72\big)$.
\end{itemize}
 \end{itemize}

In~the above list we use names similar to those proposed by \cite{Sl_Lay}.

Some of the types in this list
have two distinct varieties, which are equivalent by swapping a finite singularity with infinity. The variety where the higher rank singularity is put at $\infty$ is sometimes called the {\em natural}. For instance, $(\un \un \underline2)$ has the natural variety
$(\un \un ;\underline2)$ and the alternative variety $(\underline2\un ;\un )$.

For some varieties we give more than one
 normal form~--- they are labelled a) and b).

In~the following theorem we describe normal forms of various types of Heun class operators.
Note that there is some arbitrariness in the choice of a normal form.
We~allow the following transformations:
 $z\mapsto az+b$, division by a constant, sandwiching with powers and exponentials.

\begin{Theorem}\label{heuni}
Each Heun class operator can be transformed
into a Riemann class operator or one of the following normal forms:
\begin{center}\renewcommand{\arraystretch}{1.2}
{\small\begin{tabular}{cclll}
\hline
\multicolumn{1}{c}{type}&\multicolumn{1}{c}{$\sigma(z)$}&\multicolumn{1}{c}{$\tau(z)$}
&\multicolumn{1}{l}{\qquad\quad$\eta(z)$}
\\ \hline
$(\underline{1}\underline{1}\underline{1};\underline{1})$&$z(z-1)(z-t)$&$a_2z^2+a_1z+a_0$
&$\begin{array}{l}a)\ b_1z+b_0\\[-.5ex]b)\ b_0+b_{-1}z^{-1}\end{array}$&$t\neq0,1$
\\ \hline
$(\underline{1}\underline{1};2)$&$z(z-1)$&$a_2z^2+a_1z+a_0$&$\begin{array}{l}a)\
b_1z+b_0\\[-.5ex]b)\ b_0+b_{-1}z^{-1}\end{array}$&$a_2\neq0$
\\
$(2\underline{1};\underline{1}$)&$z^2(z-1)$& $a_2z^2+a_1z+a_0$&$\begin{array}{l}a)\
b_1z+b_0\\[-.5ex] b)\ b_0+b_{-1}z^{-1}\end{array}$& $a_0\neq0$
\\ \hline
$\big(\underline{1}\underline{1};\frac32\big)$&$z(z-1)$& $a_1z+a_0$&\hphantom{a)\ \ }$b_1z+b_0$ &$b_1\neq0$
\\
$\big(\frac32\underline{1};\underline{1}\big)$&$z^2(z-1)$& $a_2z^2+a_1z$&\hphantom{a)\ \ }$b_0+b_{-1}z^{-1}$& $b_{-1}\neq0$
\\ \hline
$(2;2)$&$z^2$& $a_2 z^2+a_1z+a_0$&$\begin{array}{l}a)\
b_1z+b_0\\[-.5ex] b)\ b_0+b_{-1}z^{-1}\end{array}$& $a_2\neq0$, $a_0=c$
\\ \hline
$\big(\frac32;2\big)$&$z^2$& $a_2z^2+ a_1z$&\hphantom{a)\ \ }$b_0+b_{-1}z^{-1}$&$b_{-1}\neq0$, $a_2=c$
\\
$\big(2;\frac32\big)$&$z^2$& $a_1z+a_0$&\hphantom{a)\ \ }$b_1z+b_0$&$b_1\neq0$, $a_0=c$
\\ \hline
$\big(\frac32;\frac32\big)$&$z^2$& $0$&\hphantom{a)\ \ }$b_1z+b_0+b_{-1}z^{-1}$& $b_{-1}\neq0$, $b_1=c$
\\ \hline
$\big(1;\frac32\big)$&$z^2$&$a_1z$&\hphantom{a)\ \ }$b_1z$& $b_1=c$
\\
$\big(\frac32;1\big)$&$z^2$& $a_1z$&\hphantom{a)\ \ }$b_{-1}z^{-1}$&$b_{-1}=c$
\\\hline
$(\underline{1};3)$&$z$& $a_2 z^2+a_1z+a_0$&$\begin{array}{l}a)\
b_1z+b_0\\[-.5ex]b)\ b_0+b_{-1}z^{-1}\end{array}$&$a_2=c$
\\
$(3;\underline{1})$&$z^3$& $a_2z^2+a_1z+a_0$&$\begin{array}{l}a)\
b_1z+b_0\\[-.5ex] b)\ b_0+b_{-1}z^{-1}\end{array}$&$a_0=c$
\\
\hline
$\big(\underline{1};\frac52\big)$&$z$& $a_0$&\hphantom{a)\ \ }$b_2z^2+b_1z+b_0$&$b_2=c$
\\
$\big(\frac52;\underline{1}\big)$&$z^3$& $a_2z^2$&\hphantom{a)\ \ }$b_0+b_{-1}z^{-1}+b_{-2}z^{-2}$&$b_{-2}=c$
\\
\hline
$(;4)$&$1$& $a_2 z^2+a_0$&\hphantom{a)\ \ }$b_1z+b_0$&$a_2=c$
\\ \hline
$\big(;\frac72\big)$&$1$&$0$& \hphantom{a)\ \ }$b_3 z^3+b_1z+b_0$&$b_3=c$
\\
\hline
\end{tabular}}
\end{center}
In~the above table $c$ denotes an arbitrary nonzero constant.
 \end{Theorem}

\begin{proof}
If $\sigma$ has 3 distinct roots, it can be transformed to $z(z-1)(z-t)$, $t\neq0,1$. By sandwiching with powers at each finite singularity we can make one of indices 0.
Then $\eta$ becomes a polynomial and $\deg\eta\leq1$. We~obtain the normal form of $(\underline{1}\underline{1}\underline{1};\underline{1})$.

Let $\sigma$ have degree 2 and 2 distinct roots. It can be transformed to $z(z-1)$. At each finite singularity we can make one of indices 0.
Then $\eta$ becomes a polynomial and $\deg\eta\leq2$. By the transformation $\e^{-\kappa z}\cdot\e^{\kappa z}$ with $\kappa$ solving
\begin{gather*}\kappa^2+a_2\kappa+b_2=0\end{gather*}
we can make $b_2=0$. If $a_2\neq0$ we obtain the normal form of $(\underline{1}\underline{1};2)$.

Assume that $a_2=0$.
If $b_1=0$, we get the ${}_2F_1$ operator, which belongs to the Riemann class.
Otherwise we obtain the normal form of $\big(\underline{1}\underline{1};\frac32\big)$.

Let $\sigma$ have degree 2 and one root. It can be transformed to $z^2$.
We~have
\begin{gather*}\eta(z)=b_2z^2+b_1z+b_0+b_{-1}z^{-1}+b_{-2}z^{-2}.\end{gather*}
By $\e^{-\kappa z}\cdot\e^{\kappa z}$ with $\kappa$ solving
\begin{gather*}\kappa^2+a_2\kappa+b_2=0\end{gather*}
we can kill $b_2$.
By $\e^{\kappa z^{-1}}\cdot\e^{-\kappa z^{-1}}$ with $\kappa$ solving{\samepage
\begin{gather*}\kappa^2+a_0\kappa+b_{-2}=0\end{gather*}
we can kill $b_{-2}$.}

Let $a_0\neq0$. By scaling we can make $a_0=1$.
Then by $z^{-\lambda}\cdot z^{\lambda}$ with $\lambda=-b_{-1}$ we can kill~$b_{-1}$, keeping $a_0=1$.
If $a_2\neq0$,
we obtain the normal form of $(2;2)$.
If $a_2=0$ and $b_1=0$, we obtain~${}_2F_0$ or Euler II type, both of the Riemann class.
 If $a_2=0$ and $b_1\neq0$, we obtain the normal form of $\big(2;\frac32\big)$.

 Let $a_0=0$. If $a_2\neq0$,
 by scaling we can make $a_2=1$ Then by $z^{-\lambda}\cdot z^\lambda$
 with $\lambda=-b_1$ we can kill $b_1$ keeping $a_2=1$. We~obtain the normal form of $\big(\frac32;2\big)$.

 Let $a_0=a_2=0$.
 If $b_1=b_{-1}=0$, the operator is of the Riemann class.
 If $b_{-1}=0$, $b_1\neq0$, then with
 $z^{-\lambda}\cdot z^\lambda$ we kill $b_0$ and we obtain $z\big(z\partial_z^2+a_1\partial_z+b_1\big)$. The operator in brackets can be reduced to the $F_1$ operator. If $b_1=0,$ $b_{-1}\neq0$, we similarly kill $b_0$ obtaining
 $\frac1z\big(z^3\partial_z^2+a_1z^2\partial_z+b_{-1}\big)$. The operator in brackets, after the transformation $z\mapsto\frac1z$ can be transformed to a $F_1$ operator. If $b_{-1},b_1\neq0$
we apply $z^{-\lambda}\cdot z^\lambda$
 with $\lambda=-\frac{a_1}{2}$ to kill $a_1$. We~obtain the normal form of
 $\big(\frac32;\frac32\big)$.

 Let $\sigma$ have degree 1. It can be transformed to $z$.
 One of indices at $0$ can be made $0$.
 Then~$\eta$ becomes a polynomial of degree $\leq3$. By applying
 $\e^{-\kappa z^2}\cdot \e^{\kappa z^2}$ with $\kappa=\frac{\sqrt{-b_3}}{2}$ we can kill $b_3$. If~$a_2\neq0$, applying
 $\e^{-\kappa z}\cdot \e^{\kappa z}$ with $\kappa=-\frac{b_2}{a_2}$ we kill $b_2$. By scaling we can make $a_2=1$ and we obtain the normal form of $(\underline{1};3)$.
 If $a_2=0$ and $b_2\neq0$,
by applying $\e^{-\kappa z}\cdot \e^{\kappa z}$ with $\kappa=-a_1$ we kill $a_1$. If $b_2\neq0$, by scaling we can make $b_2=1$ and
we obtain the normal form of $\big(\underline{1};\frac52\big)$. If~$b_2=0$, by applying
$\e^{-\kappa z}\cdot \e^{\kappa z}$ with $\kappa$ solving
\begin{gather*}
\kappa^2+a_1\kappa+b_1=0
\end{gather*}
we obtain an operator which can degenerate to the ${}_1F_1$, ${}_0F_1$ or Euler I type, all of the Riemann class.

 Let $\sigma$ have degree $0$.
 We~can assume that it is $1$. $\eta$ is a polynomial of degree $4$.

 By applying $\e^{-\kappa z^3}\cdot \e^{\kappa z^3}$ with
$
 9\kappa^2+a_23\kappa+b_4=0
$
 we kill $b_4$.

 Let $a_2\neq0$. By applying $\e^{-\kappa z^2}\cdot \e^{\kappa z^2}$ with
$ \kappa=-\frac{b_3}{2a_2}$
 we kill $b_3$. By applying $\e^{-\kappa z}\cdot \e^{\kappa z}$ with
 $\kappa=-\frac{b_2}{a_2}$
 we kill $b_2$.
After a transformation $z\mapsto az+b$ we can assume that $\tau(z)=z^2+a_0$.
We~obtain the normal form of $(;4)$

Let $a_2=0$. By applying $\e^{-\kappa z^2}\cdot \e^{\kappa z^2}$ with
$ \kappa=-\frac{a_1}{4}$
 we kill $a_1$. By applying
$\e^{-\kappa z}\cdot \e^{\kappa z}$ with
 $\kappa=-\frac{a_0}{2}$
we kill $a_0$. Thus $\tau=0$. If $b_3\neq0$, then
after a transformation $z\mapsto az+b$ we can assume that $\tau(z)=z^3+a_1z+a_0$. We~obtain the normal form of $\big(;\frac72\big)$.
If $b_3=0$, we obtain an operator that can degenerate to the Hermite, Airy, 1d Helmholtz or 1d Laplace type, all of the Riemann class.

In~$(\underline{1}\underline{1};2)$, $(\underline{1};3)$ and $(2;2)$
the transformation
$z^\lambda\cdot z^{-\lambda}$ with $\lambda=-\frac{b_1}{a_2}$ kills
$b_1z$ and produ\-ces~$b_{-1}z^{-1}$. Thus it makes the normal form b)
out of the normal form a).

If $\sigma$ has degree 3 and 2 distinct roots, it can be transformed to $z^2(z-1)$. The transformation $z\mapsto z^{-1}$ leads to $\sigma(z)=z(z-1)$.

If $\sigma$ has degree 3 and only 1 root,
it can be transformed to $z^3$.
Then $z\mapsto z^{-1}$ yields $\sigma(z)=z$.
\end{proof}

\begin{Remark} The operators listed in the table of Theorem~\ref{heuni} as
 $\big(1;\frac32\big)$ and $\big(\frac32;1\big)$ are strictly speaking not of the Riemann class: they are $z$ times an operators of the Riemann class. Hence they
 yield equations of the Riemann class. So they can be considered as ``trivial'' and were ignored in the table at the beginning of the subsection.
 \end{Remark}

\section{From Heun class to Painlev\'e equations}\label{section: Painleve}

\subsection{Method of isomonodromic deformations}
Let us review the theory of isomonodromic deformations of linear second order differential equations following \cite{Iwasaki, Ohyama, Okamoto}. We~shall use a notation similar to~\cite{Ohyama}.

Let $p$, $q$ be rational functions of $z$, depending on some parameters.
Among these parameters we single out a parameter $t$.
We~will write $v'$ for $\frac{\partial}{\partial z}v$ and $\dot v$ for $\frac{\partial}{\partial t}v$.
Consider a family of linear second order differential equations of the form
\begin{gather}\label{eq}
v''(z)+p(z)v'(z)+q(z)v(z)=0.
\end{gather}
 We~assume that when we deform the equation~(\ref{eq}), we can also deform its certain solution $v$ so that the following condition is satisfied:
\begin{gather}\label{eq1}
\dot v=a(z,t)v'+b(z,t)v.
\end{gather}
This essentially means that when we deform the equation, its solutions
``live'' on the same Riemann surface. In~particular, if there are singularities, then one should expect that the monodromy of solutions stays the same.

The compatibility of~(\ref{eq}) and~(\ref{eq1}) imposes a strong condition on the deformation. Indeed, differentiating~(\ref{eq1}) in $z$ we obtain
\begin{gather*}
 \dot v'=(a'-ap+b)v'+(-aq+b')v.
 \end{gather*}
Differentiating~(\ref{eq}) once in $t$ and~(\ref{eq1}) twice in $z$ we get
\begin{gather}
\dot v''= \big({-}\dot p-pa'+ap^2-pb-qa\big)v'+\big(paq-pb'-\dot q-qb)v, \label{eq.2}
\\
\dot v''= \big({-}2qa'-aq'+b''+apq-bq\big)v+\big(a''-2pa'-ap'+2b'-aq+ap^2-bp)v'.\label{eq.1}
\end{gather}
Equating~(\ref{eq.1}) and~(\ref{eq.2}) we obtain
\begin{gather}
\dot p-ap'+2b'-pa'+a''=0, \label{compa2.}
\\
\dot q+pb'-2qa'-aq'+b''=0.\label{compa1.}
\end{gather}

When applying this method to a concrete family of equations one needs to divide its parameters into two categories. The first category should contain all parameters responsible for the monodromy around singular points. For example, the coefficients $p_{-1}$ and $q_{-2}$ of the Laurent series of $p$, resp.~$q$ around singular points.
In~the second category we have parameters that do not influence the monodromy, typically denoted $\mu$, $\lambda$, $t$. The variable $t$ is called the ``time variable''.

\subsection{Isomonodromic deformations in presence of a
 non-logarithmic singularity}

Let
 $\sigma$, $\tau$, $\eta$ be rational functions. (At the moment we do not assume the conditions~(\ref{con-heun}) for the Heun class).
Consider the differential equation given by
\begin{gather}\label{heudef1}
 \partial_z^2+p_0(z)\partial_z+q_0(z) =\partial_z^2+\frac{\tau(z)}{\sigma(z)}\partial_z+
 \frac{\eta(z)}{\sigma(z)}.
\end{gather}
We~assume that $\sigma$, $\tau$, $\eta$ depend on a parameter $t$.
Let $\lambda$, $\mu$ be additional parameters. Following the prescription of~(\ref{h-bis-}), we introduce
the deformed equation corresponding to~(\ref{heudef1}):
\begin{gather}
\partial_z^2+p(\lambda,z)\partial_z+q(\lambda,\mu,z)
=\partial_z^2+\bigg(\frac{\tau(z)}{\sigma(z)}-\frac{1}{z-\lambda}\bigg)\partial_z \nonumber
\\ \qquad
{}+\frac{1}{\sigma(z)}\bigg(\eta(z)-\eta(\lambda)-\mu\big(\tau(\lambda)
-\sigma'(\lambda)\big)-\mu^2\sigma(\lambda) +\frac{\mu\sigma(\lambda)}{(z-\lambda)}\bigg).
\label{heudef}
\end{gather}

 The following theorem is devoted to monodromic
 deformations of \eqref{heudef}. It unifies a large family of cases in
 a single formulation. Unfortunately, this unification has one drawback:
relatively complicated conditions~\eqref{muu1}, \eqref{muu2} and~\eqref{muu3} constraining the choice
 of the time variable~$t$ and the auxiliary polynomial $c$.
 Probably this drawback is impossible to avoid. The main results
of our paper, described in the next subsection, will be corollaries
of Theorem~\ref{main0}.

\begin{Theorem}\label{main0}
 Suppose
that $c(z)=c(t,\lambda,z)$ is a
$t,\lambda$-dependent polynomial of degree $\leq2$.
 Suppose that the following conditions are satisfied:
\begin{gather}
0= \frac{\dot\tau}{\sigma}(z)-\frac{\tau\dot\sigma}{\sigma^2}(z)+ \frac{\frac{c\tau}{\sigma}(z)-\frac{c\tau}{\sigma}(\lambda)-(z-\lambda)\big(\frac{c\tau}{\sigma}\big)'(z)}
{(z-\lambda)^2},
\label{muu1}
\\
0=-\frac{\dot\sigma}{\sigma}(z)\big(\eta(z)-\eta(\lambda)\big)+\dot\eta(z)-\dot\eta(\lambda)\nonumber
\\ \hphantom{0=}
{} +\frac{\eta(z)-\eta(\lambda)}{(z-\lambda)}\bigg(\frac{c\sigma'}{\sigma}(z)-c'(z)\bigg)
-\eta'(\lambda)\bigg(\frac{\sigma'c}{\sigma}(\lambda)-c'(\lambda)\bigg) \nonumber
 \\ \hphantom{0=}
{} + \frac{2c\eta(z)-2c\eta(\lambda)-\big((c\eta)'(z)+(c\eta)'(\lambda)\big)(z-\lambda)} {(z-\lambda)^2},
 \label{muu2}
 \\
 0= \frac{\dot\sigma}{\sigma}(z)- \frac{\dot\sigma}{\sigma}(\lambda)-
 \frac{\frac{c\sigma'}{\sigma}(z)-\frac{c\sigma'}{\sigma}(\lambda)
-\big(\frac{c\sigma'}{\sigma}\big)'(\lambda)(z-\lambda)} {(z-\lambda)}.
\label{muu3}
 \end{gather}
 Define
the {\em compatibility functions}
\begin{gather}
 a(t,\lambda,z):=\frac{c(z)}{z-\lambda},\qquad
 b(t,\lambda,\mu,z)=-\frac{c(\lambda)\mu}{z-\lambda}\label{compa3}.
 \end{gather}
Then the following equations for $\lambda$, $\mu$
 \begin{gather}
 \dot\lambda=2 c(\lambda)\mu-c'(\lambda)+\frac{c\tau}{\sigma}(\lambda),
 \label{fir0}
 \\
 \dot\mu=-\frac{c\eta'}{\sigma}(\lambda)-\mu\bigg(\frac{c\tau'}{\sigma}(\lambda)-
 \frac{c\sigma''}{2\sigma}(\lambda)-\frac{c''}{2}\bigg)-\mu^2\frac{\sigma'c}{\sigma}(\lambda)
 \label{second0}
 \end{gather}
are
equivalent to the compatibility conditions~\eqref{compa1.} and~\eqref{compa2.}.
\end{Theorem}

 The proof of Theorem~\ref{main0} is deferred to Appendixes~\ref{ap0},~\ref{ap1} and~\ref{ap2}.

\subsection{Isomonodromic deformations of Heun class equations}

This subsection contains the main results of our paper. We~will suppose
 that $\sigma$ is a polynomial of degree $\leq3$,
$\tau$ is a polynomial of degree $\leq2$ and
$\eta\sigma$ is a polynomial of degree $\leq4$.
In~other words, we will assume that
\eqref{heudef1} is a Heun class equation.
We~will show that Theorem~\ref{main0} can be applied
to a large family of Heun class equations, including normal forms of
all its types. As~a~result we obtain all types of Painlev\'e equation.

Our main results will be formulated in two theorems. In~Theorem~\ref{main1} we still try to give a~unified treatment. More precisely,
we consider two closely related ansatzes, which we call A and B.
Ansatz A is applicable if $\sigma$ has a
zero. Ansatz B can be used if the degree of $\sigma$ is $\leq2$. The time
variable is not specified, it is only constrained by certain
conditions.

In~Theorem~\ref{main} the time variable is always
explicit. Unfortunately, it seems impossible to do it in a unified
way~--- we are compelled to consider 5 distinct cases. (Note that 5 is
still less than the number of types of Heun class equations. Besides,
some of these cases are applicable to more than one type).

\begin{Theorem}\label{main1}
{\bf Case A.} Assume that $s\in\cc$ and $\sigma(s)=0$, so that we
 can write $\sigma(z)=(z-s)\rho(z)$ for a polynomial $\rho$ of
 degree $\leq2$. We~assume that $\sigma$, $\tau$, $\eta$
 depend on~$t$. Let $m$ be a function of~$t$ satisfying the following conditions
 \begin{gather}
 \partial_t \frac{\tau}{\sigma}(z)=\frac{m\tau(s)}{(z-s)^2},\label{moo1}
 \\
 \nonumber
 \frac{\dot\sigma}{\sigma}(z)\big(\eta(z)-\eta(\lambda)\big)-\dot\eta(z)+\dot\eta(\lambda)
 = m\bigg(\frac{\big(\eta(z)-\eta(\lambda)\big)(\lambda-s)\rho(z)}{(z-\lambda)(z-s)}
 -\eta'(\lambda)\rho(\lambda)\bigg) %
 \\ \hphantom{\frac{\dot\sigma}{\sigma}(z)\big(\eta(z)-\eta(\lambda)\big)}
{}+\frac{(\lambda-s)}{(z-\lambda)^2}\big(2\rho\eta(z)-2\rho\eta(\lambda) -\big((\rho\eta)'(z)+(\rho\eta)'(\lambda)\big)(z-\lambda)\big),\label{moo2}
\\
\frac{\dot\sigma}{\sigma}(z)- \frac{\dot\sigma}{\sigma}(\lambda)=
 m\rho(s)\frac{(z-\lambda)}{(z-s)(\lambda-s)}.\label{moo3}
 \end{gather}
 Define the compatibility functions
 \begin{gather*}
 a(t,\lambda,z)=\frac{m(\lambda-s)\rho(z)}{z-\lambda},\qquad
 b(t,\lambda,z)=-\frac{m(\lambda-s)\rho(\lambda)\mu}{z-\lambda},
 \end{gather*}
 and the Hamiltonian
 \begin{gather}
 \label{haha1}
 H(t,\lambda,\mu)=m\big(\eta(\lambda)+\big(\tau(\lambda)-(\lambda-s)\rho'(\lambda)\big)\mu
 + \sigma(\lambda)\mu^2\big).
\end{gather}
 Then $\lambda$, $\mu$ satisfy the Hamilton equations with respect to $H$, that is
\begin{gather*}
\frac{\d\lambda}{\d t}=\frac{\partial H}{\partial \mu}(t,\lambda,\mu),
\\
\frac{\d\mu}{\d t}=-\frac{\partial H}{\partial \lambda}(t,\lambda,\mu),
\end{gather*}
 if and only if~\eqref{compa1.} and~\eqref{compa2.} hold.

 \medskip

\noindent {\bf Case B.} Assume that $\deg\sigma\leq2$ and $\deg\sigma\eta\leq3$. Suppose that
 $\tau$, $\eta$, but not $\sigma$ depend on $t$. Let~$m$ be a function of $t$ satisfying
\begin{gather}\label{mii1}
 \frac{\dot \tau(z)}{\sigma(z)}=m\frac{\tau''}{2},
 \\ \label{mii2}
 \dot\eta(z)-\dot\eta(\lambda)= m\frac{(\sigma\eta)'''}{6}(z-\lambda).
\end{gather}
 Define the compatibility functions
 \begin{gather*}
 a(t,\lambda,z)=\frac{m\sigma(z)}{z-\lambda},\qquad
 b(t,\lambda,z)=-\frac{m\sigma(\lambda)\mu}{z-\lambda},
 \end{gather*}
and the Hamiltonian
 \begin{gather}\label{haha2}
 H(t,\lambda,\mu):=m\big(\eta(\lambda)+\big(\tau(\lambda)-\sigma'(\lambda)\big) \mu
 +\sigma(\lambda)\mu^2\big).
 \end{gather}
 Then $\lambda$, $\mu$ satisfy the Hamilton equations with
 respect to $H$ if and only if~\eqref{compa1.} and~\eqref{compa2.} hold.
 \end{Theorem}

 \begin{proof} We~set
 \begin{gather}
 c(z)=m(\lambda-s)\rho(z),\qquad\text{in Case A},\label{caso1}
 \\
 c(z)=m\sigma(z),\qquad\text{in Case B}.\label{caso2}
 \end{gather}
 We~apply Theorem~\ref{main0}.
 More precisely, we check that conditions \eqref{muu1}, \eqref{muu2} and \eqref{muu3}
are equi\-valent to \eqref{moo1}, \eqref{moo2} and \eqref{moo3},
resp.~to~\eqref{mii1}, \eqref{mii2}.
We~also verify that equations~\eqref{fir0} and \eqref{second0}
coincide with the Hamilton equations for $H(t,\lambda,\mu)$.

Details of computations are given in Appendixes~\ref{apa},~\ref{apa1}
and~\ref{apb}. \end{proof}

Note that it is possible to unify the formulas \eqref{haha1} and
\eqref{haha2} for the Hamiltonian in a single formula using the
polynomial $c$ from \eqref{caso1} and \eqref{caso2}:
\begin{gather*}
 H(t,\lambda,\mu)=\frac{\eta(\lambda)c(\lambda)}{\sigma(\lambda)}
+\mu\bigg(\frac{\tau(\lambda)c(\lambda)}{\sigma(\lambda)}-c'(\lambda)\bigg)
 + \mu^2 c(\lambda).
 \end{gather*}

Ansatzes A and B of Theorem~\ref{main1} will still be subdivided into
several subcases that differ by the choice of the time variable.
All these subcases are described in Theorem~\ref{main} below, which together with Theorem~\ref{main1} describes the main
result of our paper.

The main subcase of Ansatz A is A1, where the time variable is the
position of a nondegenerate zero of $\sigma$. This is of course
generically true, however there are situations when $\sigma$ does not
have nondegenerate zeros.
Subcases Ap and Aq can be applied when $\sigma$ has a degenerate (that
is, at least double) zero.

In~Subcases Ap and Bp time is contained in the function
$p_0$ (the coefficient of the first order term) of \eqref{heudef1}
and in Aq and Bq it is contained in $q_0$. Subcases Ap and Bp are typically used when the
rank at $\infty$ is an integer. Subcases Aq and Bq are more appropriate for
degenerate types, when the rank at $\infty$ is half-integer.

In~the following theorem,
first we describe $\sigma$, $\tau$ and $\eta$ that belong to a given
subcase, specifying explicitly the dependence on $t$. Next we write the corresponding (undeformed) Heun class
operator in the principal form. Then we give the corresponding compatibility
functions $a$, $b$ and the Hamiltonian $H$.

By writing $\deg \beta\leq n$ we mean that $\beta$ is a polynomial in
$z$ of degree $\leq n$. Unlike in~Theo\-rem~\ref{main}, the dependence on the parameter $t$ will be always explicitly given.
 \begin{Theorem}\label{main}${}$
 \begin{itemize}\itemsep=0pt
 \item {\bf Subcase A1.} Let
 \begin{gather*}
 \sigma(z)=(z-t)\rho(z),\qquad\deg\rho\leq2;
 \\
 \tau(z)=(1-\kappa)\rho(z)+\phi(z)(z-t),\qquad\kappa\in\cc,\qquad\deg\phi\leq1;
 \\
\eta(z)=\frac{\alpha\rho(t)}{z-t}+\eta_0(z)\qquad\alpha\in\cc,\qquad\deg\eta_0\leq1.
 \end{gather*}
 Consider
 \begin{gather*}
 \partial_z^2+
 \Big(\frac{1-\kappa}{z-t}+\frac{\phi(z)}{\rho(z)}\Big)\partial_z+
 \frac{\alpha\rho(t)}{(z-t)^2\rho(z)}+\frac{\eta_0(z)}{(z-t)\rho(z)}.
 \end{gather*}
If $ \rho(t)\neq0$, then Theorem~$\ref{main1}$A holds with $m=\rho(t)^{-1}$,
\begin{gather*}
 a(z):=\frac{(\lambda-t)\rho(z)}{\rho(t)(z-\lambda)},\qquad
b(z):=-\frac{\sigma(\lambda)\mu}{\rho(t)(z-\lambda)},
\\
\rho(t)H:=(\lambda-t)\rho(\lambda)\mu^2
+\big((1-\kappa)\rho(\lambda)+(\lambda-t)(\phi(\lambda)-\rho'(\lambda)\big)\mu
+\frac{\alpha\rho(t)}{\lambda-t}+\eta_0(\lambda).
\end{gather*}
\item {\bf Subcase Ap.} Let
 \begin{gather*}
 \sigma(z)=(z-s)^2\rho_1(z),\qquad s\in\cc,\qquad\deg\rho_1\leq1;
 \\
 \tau(z)=t\rho_1(z)+\tau_0(z),\qquad \deg\tau_0\leq2 ;
 \\
\eta(z)=\frac{\psi(z)}{\rho_1(z)},\qquad \deg\psi\leq2.
 \end{gather*}
 Consider
 \begin{gather*}
 \partial_z^2+ \bigg( \frac{t}{(z-s)^2}+\frac{\tau_0(z)}{(z-s)^2\rho_1(z)}\bigg)\partial_z
 +\frac{\psi(z)}{(z-s)^2\rho_1(z)^2}.
 \end{gather*}
 If $\tau_0(s)\neq0$ or $\rho_1(s)\neq0,$
 then Theorem~$\ref{main1}$A holds with $m=\big(\tau_0(s)+t\rho_1(s)\big)^{-1}$,
\begin{gather*}
 a(z):=\frac{(\lambda-s)(z-s)\rho_1(z)}{\big(\tau_0(s)+t\rho_1(s)\big)(z-\lambda)},\qquad
b(z):=-\frac{(\lambda-s)^2\rho_1(\lambda)\mu}{\big(\tau_0(s)+t\rho_1(s)\big)(z-\lambda)},
\\
\big(\tau_0(s)+t\rho_1(s)\big)H:=(\lambda-s)^2\rho_1(\lambda)\mu^2
\\ \hphantom{\big(\tau_0(s)+t\rho_1(s)\big)H:=}
{}+\big(t\rho_1(\lambda)+\tau_0(\lambda)-(\lambda-s)\rho_1(\lambda)-(\lambda-s)^2\rho_1'
\big)\mu+\frac{\psi(\lambda)}{\rho_1(\lambda)}.
 \end{gather*}
\item
 {\bf Subcase Aq.} Let
 \begin{gather*}
 \sigma(z):=(z-s)^2\rho_1(z),\qquad s\in\cc,\qquad\deg\rho_1\leq1;
 \\
 \tau(z)=(z-s)\phi(z),\qquad \deg\phi\leq1;
 \\
 \eta(z)=\frac{t}{z-s}+\eta_0(z),\qquad \deg\rho_1\eta_0\leq2.
 \end{gather*}
 Consider
 \begin{gather*}
 \partial_z^2+ \frac{\phi(z)}{(z-s)\rho_1(z)}\partial_z
 +\frac{1}{(z-s)^2\rho_1(z)}\bigg(\frac{t}{z-s}+\eta_0(z)\bigg).
 \end{gather*}
If $(\sigma\eta_0)'(s)\neq0$ or $\rho_1(s)\neq0$, then Theorem~$\ref{main1}$A holds with $m=\big((\sigma\eta_0)'(s)+\rho_1(s)t\big)^{-1}{:}$
 \begin{gather*}
 a(z):=\frac{(\lambda-s)(z-s)\rho_1(z)}{\big((\sigma\eta_0)'(s)+\rho_1(s)t\big)(z-\lambda)},\qquad
b(z):=-\frac{(\lambda-s)^2\rho_1(\lambda)\mu}{\big((\sigma\eta_0)'(s)+\rho_1(s)t\big) (z-\lambda)},
\\
\big((\sigma\eta_0)'(s)+\rho_1(s)t\big) H:=(\lambda-s)^2\rho_1(\lambda)\mu^2\nonumber
+\big((\lambda-s)(\phi(\lambda)-\rho_1(\lambda))-(\lambda-s)^2\rho_1'\big)\mu
\\ \hphantom{\big((\sigma\eta_0)'(s)+\rho_1(s)t\big) H:=}
{}+\frac{t}{\lambda-s}+\eta_0(\lambda).
 \end{gather*}

\item
 {\bf Subcase Bp.} Let
 \begin{gather*}
 \sigma(z),\qquad\deg\sigma\leq2;
 \\
 \tau(z)=t\sigma(z)+\tau_0(z),\qquad\deg\tau_0\leq2;
 \\
 \eta(z),\qquad\deg\sigma\eta\leq2.
 \end{gather*}
 Consider
 \begin{gather*}
 \partial_z^2+ \bigg(t+\frac{\tau_0(z)}{\sigma(z)}\bigg)\partial_z
 +\frac{\eta(z)}{\sigma(z)}.
 \end{gather*}
If $\sigma''\neq0$ or $\tau_0''\neq0$, then Theorem~$\ref{main1}$B holds with
 $m=\big(t\frac{\sigma''}2+\frac{\tau_0''}2\big)^{-1}{:}$
\begin{gather*}
a(z):=\frac{\sigma(z)}{
\big(t\frac{\sigma''}2+\frac{\tau_0''}2\big) (z-\lambda)},\qquad b(z):=-\frac{\sigma(\lambda)\mu}{\big(t\frac{\sigma''}2+\frac{\tau_0''}2\big)(z-\lambda)},
\\
\bigg (t\frac{\sigma''}2+\frac{\tau_0''}2\bigg) H:=\sigma(\lambda)\mu^2+
 \big(t\sigma(\lambda)+\tau_0(\lambda)-\sigma'(\lambda)\big)\mu+
\eta(\lambda).
 \end{gather*}

\item {\bf Subcase Bq.} Let
 \begin{gather*}
 \sigma(z),\qquad \deg\sigma\leq2;\\
 \tau(z),\qquad\deg\tau\leq1;\\
 \eta(z)=tz+\eta_0(z),\qquad \deg\sigma\eta_0\leq3.
 \end{gather*}
 Consider
 \begin{gather*}
 \partial_z^2+ \frac{ \tau(z)}{\sigma(z)}\partial_z
 +\frac{tz}{\sigma(z)}+\frac{\eta_0(z)}{\sigma(z)}.
 \end{gather*}
If $\sigma''\neq0$ or $ (\sigma\eta_0)'''\neq0$, then Theorem~$\ref{main1}$B holds with $m=\big(t\frac{\sigma''}{2}+\frac{(\sigma\eta_0)'''}{6}\big)^{-1}$,
\begin{gather*}
a(z):=\frac{\sigma(z)}{\big(t\frac{\sigma''}{2}+\frac{(\sigma\eta_0)'''}{6}\big) (z-\lambda)},\qquad b(z):=-\frac{\sigma(\lambda)\mu}{\big(t\frac{\sigma''}{2}+\frac{(\sigma\eta_0)'''}{6}\big)(z-\lambda)},
\\
\bigg(t\frac{\sigma''}{2}+\frac{(\sigma\eta_0)'''}{6}\bigg) H:=\sigma(\lambda)\mu^2+\big(\tau(\lambda)-\sigma'(\lambda)\big)\mu
 +t\lambda+\eta_0(\lambda).
\end{gather*}
\end{itemize}
\end{Theorem}

Note that one can deduce Subcases Ap and Aq from Subcases Bp resp.~Bq by applying the symmetry described above Proposition~\ref{symmo}.
However, in Appendices~\ref{apa},~\ref{apa1} and~\ref{apb}, where we
prove Theorem~\ref{main}, we will give independent proofs of all subcases.

Note also that the union of subcases of Theorem~\ref{main} does not cover
 the whole Heun class. However, it covers all appropriately
 interpreted normal forms listed
 in Theorem~\ref{heuni}. This will be further discussed in the
 following subsection.

\begin{Remark}
 In~our applications we will sometimes use rescaled versions of the above
 constructions. In~fact, if $\epsilon\neq0$, we replace $t$ with $\epsilon t$ and multiply $a$, $b$, $H$ with $\epsilon$, then the above theorem remains true.
\end{Remark}

\subsection{Correspondence between Heun class and Painlev\'e equations}

Traditionally, Painlev\'e equations are divided into 6 types: Painlev\'e I--VI. However, one can argue that some of their degenerate cases should be treated as separate types.

Thus Painlev\'e V~(\ref{painleve5}) splits into the nondegenerate Painlev\'e V with $\delta\neq0$ and the degenerate Painlev\'e V with $\delta=0$. We~denote the former simply by ndeg-V and the latter by deg-V.
One can show that deg-V Painlev\'e is equivalent to Painlev\'e III$'$, however
it is natural to treat it as a separate type.
All that is explained in Section~\ref{PainleveV}.

With Painlev\'e III~(\ref{painleve3}) the situation is more complicated.
First of all, following various authors, we prefer to use the Painlev\'e III$'$ equation, which is equivalent to Painlev\'e III by a simple transformation. Beside the nondegenerate case we have the degenerate case
and the doubly degenerate case.
We~denote them respectively, ndeg-III$'$, deg-III$'$ and ddeg-III$'$. $\big($Ohyama--Oku\-mura denote them $\big(D_6^{(1)}\big)$, $\big(D_7^{(1)}\big)$, $\big(D_8^{(1)}\big).\big)$
One can also consider an alternative degenerate case $\gamma\neq0$, $\delta=0$, which is however equivalent to deg-III$'$.
$\big($Ohyama--Okumura denotes it $\big(D_7^{(1)}\big){-}2.\big)$ All of that is explained in
Section~\ref{PainleveIII'}.

Finally, it is natural to consider the Painlev\'e 34 equation~(\ref{P34}), which can be viewed as~a degenerate case of the Painlev\'e IV equation~(\ref{painleve434}).
One can show that Painlev\'e 34 is equivalent to Painlev\'e II, however it is natural to keep it as a separate type.
See Section~\ref{Painleve IV-34} for more comments.

This way we obtain 10 types of Painlev\'e equations. Recall that we
also have 10 types of Heun class equations. In~fact, each type of Painlev\'e can be derived from one of the types of deformed Heun.
Here is the list of correspondences:
\begin{center}\renewcommand{\arraystretch}{1.3}
{\small\begin{tabular}{lcllc}
\hline
(standard) Heun &$(\un \un \un ;\un )$ & A1 a),b) &Painlev\'e VI &$(\un \un \un \un )$
\\ \hline
\raisebox{-2.5mm}[0mm][0mm]{ndeg. confluent Heun} & $(\un \un;2)$&A1 a),b), Bp b)
& \raisebox{-2.5mm}[0mm][0mm]{Painlev\'e ndeg-V} & \raisebox{-2.5mm}[0mm][0mm]{$(\un \un 2)$}
\\
&$(2\un ;\un )$&A1 a), Ap a)
\\\hline
\raisebox{-2.5mm}[0mm][0mm]{deg. confluent Heun} & $\big(\un \un ;\frac32\big)$&A1, Bq
& \raisebox{-2.5mm}[0mm][0mm]{Painlev\'e deg-V} & \raisebox{-2.5mm}[0mm][0mm]{$\big(\un \un \frac32\big)$}
\\
&$\big(\frac32\un ;\un \big)$& Aq
\\\hline
db. confluent Heun&$(2;2)$& Ap a), Bp b)& Painlev\'e ndeg-III$'$ & $(22)$
\\\hline
\raisebox{-2.5mm}[0mm][0mm]{deg. db. confluent Heun}& $\big(\frac32;2\big)$&Aq, Bp
& \raisebox{-2.5mm}[0mm][0mm]{Painlev\'e deg-III$'$}& \raisebox{-2.5mm}[0mm][0mm]{$\big(\frac322\big)$}
\\
&$\big(2;\frac32\big)$&Ap, Bq
\\ \hline
ddeg. db. confluent Heun&$\big(\frac32;\frac32\big)$ &Bq&Painlev\'e ddeg-III$'$ &$\big(\frac32\frac32\big)$
\\ \hline
\raisebox{-2.5mm}[0mm][0mm]{ndeg. bi-confluent Heun}& $(\un ;3)$& A1 a),b), Bp a),b)
&\raisebox{-2.5mm}[0mm][0mm]{Painlev\'e IV} & \raisebox{-2.5mm}[0mm][0mm]{$(\un 3)$}
\\
&$(3;\un)$&Ap b)
\\ \hline
\raisebox{-2.5mm}[0mm][0mm]{deg. bi-confluent Heun}& $\big(\un ;\frac52\big)$&Bq
& \raisebox{-2.5mm}[0mm][0mm]{Painlev\'e 34} & \raisebox{-2.5mm}[0mm][0mm]{$\big(\un \frac52\big)$}
\\
&$(\frac52;\un )$&Aq
\\ \hline
ndeg. tri-confluent Heun& $(;4)$ &Bp&Painlev\'e II&$(4)$
\\ \hline
deg. tri-confluent Heun& $\big(;\frac72\big)$ & Bq&Painlev\'e I & $\big(\frac72\big)$
\\ \hline
\end{tabular}}
\end{center}

In~the first column we give the name of the Heun class type. For typographical reasons we abbreviate ``nondegenerate'' to ndeg, ``degenerate'' to deg, ``doubly degenerate'' to
ddeg and ``doubly'' to db.

In~the second column we give the symbol of the type in terms of the ranks of singularities. We~also indicate which singularity is at $\infty$. In~several cases there are two possibilities~--- we give both of them.

In~the third column we indicate subcases of Theorem~\ref{main} which can
be applied to certain normal forms of a given type. Normal forms
are taken from the table in Theorem~\ref{heuni}.
If in that table more than one normal form is given,
 a) or b) indicates which normal form is considered.
 (In~some cases, the normal forms from Theorem~\ref{heuni} need
to be slightly modified: When we
apply A1 to the form b) of
$(\un\un\un;\un)$ we change the roles of the roots; for
$(\un\un;2)$ and $(\un;3)$ we shift one root from $0$ to $t$;
finally for $(2\un;\un)$ and $\big(\un\un;\frac32\big)$ we shift a root from
$1$ to $t$.)

In~the fourth column we give the name of the Painlev\'e type that can be obtained by the isomonodromic deformation.

In~the fifth column we list the symbol in terms of ranks of singularities without the indication of the position of $\infty$. We~will often use it in the sequel as the name of the given type
of the Painlev\'e equation. Thus, e.g., the Painlev\'e
$\big(\frac322\big)$ equation is an alternative name for the degenerate Painlev\'e III$'$ equation.
We~will actually prefer these names to the traditional ones, similarly as
 Slavyanov--Lay in \cite{Sl_Lay}.

\looseness=-1 Occasionally, we will also use the names for Painlev\'e equation involving the
 position of $\infty$. For instance,
 the Painlev\'e
$\big(\frac32;2\big)$ equation will mean the form of the degenerate Painlev\'e III$'$ equation
 obtained from
the Heun $\big(\frac32;2\big)$ equation. The Painlev\'e
$\big(2;\frac32\big)$ equation will denote the equation
obtained from
the Heun $\big(2;\frac32\big)$ equation.
Both forms of Painlev\'e equation are equivalent.

We~will discuss further the classification of Painlev\'e equations in Section~\ref{FivesupertypesofPainleveequation}, where we will see how to group the 10 types into 5 supertypes,
parallel to the grouping of 10 types of Heun class equations into 5 supertypes.

In~the following subsections we describe how to obtain all types of Painlev\'e equations from deformed Heun class equations.
First we give the functions $\sigma$, $\tau$, $\eta$ describing one of possible normal forms of a given type of the Heun class equation.
We~indicate explicitly the dependence of
$\sigma$, $\tau$, $\eta$ on the time variable $t$.
Then we present this equation in its principal form. Next
we give the corresponding deformed equation.
Next we give the compatibility functions $a$, $b$ and the Painlev\'e Hamiltonian. Finally, we describe the resulting Painlev\'e equation.

Note that the whole procedure is determined by
 $\sigma$, $\tau$, $\eta$, by the choice of the time variable $t$ and the functions $a,b$. The latter are restricted by Theorem~\ref{main}. We~always indicate which case
of Theorem~\ref{main} we use.

\looseness=1
In~our derivations we follow 
the paper of Ohyama--Okumura \cite{Ohyama}.
We~have slightly changed their notation for some of the parameters.
We~parametrize the equations by the differences of indices at singular points of the deformed equation. In~particular, if the rounded rank at~$z_0$ is~1, the parameter is called $\kappa_{z_0}$, if the rank is $2$, it is called $\chi_{z_0}$ and if the rank is $3$ it is called~$\theta_{z_0}$.

 One of the parameters of the initial Heun class equation~--- the free term in $\eta$~--- does not enter in the deformed equation, and therefore is not used by \cite{Ohyama}. We~denote it simply by $c$.

Note that there is some arbitrariness in the choice of Hamiltonians, where a term depending on $t$, but not on $\lambda$, $\mu$, can always be added. We~always choose Hamiltonians coinciding with those of \cite{Ohyama}.

 If in a given type $\sigma$ has a root of multiplicity $1$,
one can usually use Case A1 of Theorem~\ref{main}. Following \cite{Ohyama}, we use it only
for Heun $(\un\un\un\un)$--Painlev\'e VI. For the other types we use Ap, Aq, Bp, Bq.

In~general, for each type of Heun class we give one derivation. The exception is
the type $\big(2\frac32\big)$, where,
 following again \cite{Ohyama}, we give two versions of derivations of deg-III$'$: one in the form of~$\big(2;\frac32\big)$, the other in the form of $\big(\frac32;2\big)$.

\subsection{From Heun (\underline1\underline{1}\underline{1};\underline{1}) to Painlev\'e VI} \label{painleve6}

Set
\begin{gather*}
\sigma(z)=z(z-1)(z-t),
\\
\tau(z)=(1-\kappa_0)(z-1)(z-t)+(1-\kappa_1)z(z-t)+(1-\kappa_t ) z(z-1),
\\
\eta(z)=\frac{\big((\kappa_0+\kappa_1+\kappa_t -1)^2-\kappa_{\infty}^2\big)
 z}{4}-c.
\end{gather*}
Heun $(\un \un \un ;\un )$ equation
\begin{gather*}
\partial_z^2\!+\bigg(\frac{1\!-\kappa_0}{z}\!+\frac{1\!-\kappa_1}{z\!-1}\!+\frac{1\!-\kappa_t }{z\!-t}\bigg)\partial_z\!+\frac{1}{z(z\!-1)(z\!-t)}\bigg(\frac{
\big((\kappa_0\!+\kappa_1\!+\kappa_t \!-1)^2\!-\kappa_{\infty}^2\big)}{4}z\!-c\bigg).
\end{gather*}
Deformed Heun $(\un \un \un ;\un )$ equation
\begin{gather*} 
\partial_z^2+\bigg(\frac{1-\kappa_0}{z}+\frac{1-\kappa_1}{z-1}
+\frac{1-\kappa_t}{z-t}-\frac{1}{z-\lambda}\bigg)\partial_z
\\ \qquad
{}+\frac{1}{z(z-1)(z-t)}\bigg(\frac{\big((\kappa_0+\kappa_1+\kappa_t -1)^2-\kappa_{\infty}^2\big)}
 {4}(z-\lambda)-\mu^2\lambda(\lambda-1)(\lambda-t)
 \\ \qquad\qquad
+\big(\kappa_0(\lambda-1)(\lambda-t)+\kappa_1\lambda(\lambda-t)
 +\kappa_t \lambda(\lambda-1)\big)\mu
+\frac{\lambda(\lambda-1)(\lambda-t)\mu}{(z-\lambda)}\bigg).
\end{gather*}
Type A1 compatibility functions
\begin{gather*}
a(z)=\frac{(\lambda-t)z(z-1)}{t(t-1)(z-\lambda)},\qquad
b(z)=-\frac{\lambda(\lambda-1)(\lambda-t)\mu}{t(t-1)(z-\lambda)}.
\end{gather*}
Painlev\'e $(\un \un \un ;\un )$ Hamiltonian
\begin{gather*}
t(t-1)H=\lambda(\lambda-1)(\lambda-t)\mu^2
-\big(\kappa_0(\lambda-1)(\lambda-t)+\kappa_1\lambda(\lambda-t)+(\kappa_t -1)\lambda(\lambda-1)\big)\mu
\\ \hphantom{t(t-1)H=}
{}+\frac{\big((\kappa_0+\kappa_1+\kappa_t -1)^2-\kappa_{\infty}^2\big)
 (\lambda-t)}{4}.
\end{gather*}
Painlev\'e $(\un \un \un ;\un )$ equation
\begin{gather}\nonumber
\frac{\d^2\lambda}{\d t^2}=\frac{1}{2}\bigg(\frac{1}{\lambda}+\frac{1}{\lambda-1} +\frac{1}{\lambda-t}\bigg)\bigg(\frac{\d\lambda}{\d t}\bigg)^2 -\bigg(\frac{1}{t}+\frac{1}{t-1}+\frac{1}{\lambda-t}\bigg)\frac{\d \lambda}{\d t}
\\ \hphantom{\frac{\d^2\lambda}{\d t^2}=} {}+\frac{\lambda(\lambda-1)(\lambda-t)}{t^2(t-1)^2} \bigg(\alpha+\beta\frac{t}{\lambda^2}+\gamma\frac{t-1}{(\lambda-1)^2}+\delta\frac{t(t-1)}{(\lambda-t)^2}\bigg),
\label{P6}
\end{gather}
where
\begin{gather*}
\alpha=\frac{1}{2}\kappa_{\infty}^2,\qquad
\beta=-\frac{1}{2}\kappa_0^2,\qquad
\gamma=\frac{1}{2}\kappa_1^2,\qquad
\delta=\frac{1}{2}\big(1-\kappa_t ^2\big).
\end{gather*}
The standard name of~(\ref{P6}) is
Painlev\'e VI equation.

\subsection{From Heun (2\un;\un) to Painlev\'e V}

Set
\begin{gather*}
 \sigma(z)=(z-1)^2z,
 \\
 \tau(z)=(2-\chi_1)z(z-1)+(1-\kappa_0)(z-1)^2+ tz,
 \\
 \eta(z)=\frac{\big((\kappa_0+\chi_1-1)^2-\kappa_\infty^2\big)(z-1)}{4}-c.
\end{gather*}
Heun $(2\un ;\un )$ equation (with the singularity of rank 2 put at 1):
\begin{gather*}
 \partial_z^2+\bigg(\frac{2-\chi_1}{z-1}+\frac{1-\kappa_0}{z}+\frac{t}{(z-1)^2}\bigg)\partial_z +\frac{1}{(z-1)^2z}\bigg(\frac{(\kappa_0+\chi_1-1)^2-\kappa_\infty^2}{4}(z-1)-c\bigg).
 \end{gather*}
Deformed Heun $(2\un ;\un )$ equation:
\begin{gather*}
\partial_z^2+\bigg(\frac{t}{(z-1)^2}+\frac{2-\chi_1}{z-1} +\frac{1-\kappa_0}{z}-\frac{1}{z-\lambda}\bigg)\partial_z
\\ \qquad
{}+\frac{1}{(z-1)^2z}\bigg(\frac{(\kappa_0+\chi_1-1)^2-\kappa_\infty^2}{4}(z-\lambda)-
 ( \lambda-1)^2\lambda\mu^2
\\ \qquad\qquad
{}-\big({-}\kappa_0(\lambda-1)^2-\chi_1\lambda(\lambda-1)+ t\lambda\big)\mu+\frac{(\lambda-1)^2\lambda\mu}{(z-\lambda)}\bigg).
\end{gather*}
Type Ap compatibility functions
\begin{gather*}
a(z)=\frac{(\lambda-1) z(z-1)}{t(z-\lambda)},\qquad
 b(z)=-\frac{(\lambda-1)^2\lambda\mu}{t(z-\lambda)}.
\end{gather*}
Painlev\'e $(2\un ;\un )$ Hamiltonian
\begin{gather}\nonumber
tH= (\lambda-1)^2\lambda\mu^2-\big(\kappa_0(\lambda-1)^2+(\chi_1-1) \lambda(\lambda-1)- t \lambda\big)\mu
\\ \hphantom{tH=}
{}+\frac{\big((\kappa_0+\chi_1-1)^2-\kappa_\infty^2\big)(\lambda-1)}{4}.\label{P5h}
\end{gather}
Painlev\'e $(2\un ;\un )$ equation:
\begin{gather}
\frac{\d^2\lambda}{\d t^2}= \bigg(\frac{1}{2\lambda}+
\frac{1}{\lambda-1}\bigg)\bigg(\frac{\d\lambda}{\d t}\bigg)^2-\frac{1}{t} \frac{\d \lambda}{\d t}+\frac{(\lambda-1)^2}{t^2}\bigg(\alpha \lambda+\frac{\beta}{\lambda}\bigg)
+\gamma\frac{\lambda}{t}-\frac{\lambda(\lambda+1)}{2(\lambda-1)},\label{P5}
\end{gather}
where
\begin{gather*}
\alpha=\frac{1}{2}\kappa_{\infty}^2,\qquad
\beta=-\frac{1}{2}\kappa_0^2,\qquad
\gamma=\chi_1.
\end{gather*}
According to the standard terminology,
(\ref{P5}) is the nondegenerate case of the Painlev\'e V equation.

\subsection[From Heun ($\frac32$\un ;\un ) to degenerate Painlev\'e V]
{From Heun $\boldsymbol{\big(\frac32\un ;\un \big)}$ to degenerate Painlev\'e V}
Set
\begin{gather*}
\sigma(z)=(z-1)^2z,
\\
\tau(z)=(z-1)z+(1-\kappa_0)(z-1)^2,
\\
\eta(z)=-\frac{ t}{(z-1)}+\frac{\big(\kappa_0^2-\kappa_\infty^2\big)z}{4}-c.
\end{gather*}
Heun $\big(\frac321;1\big)$ equation
(with the singularity of rank 2 put at 1):
\begin{gather*}
\partial_z^2+\bigg(\frac{1}{z-1}+\frac{1-\kappa_0}{z}\bigg)\partial_z
+\frac{1}{z(z-1)^2}\bigg({-}\frac{ t}{(z-1)}+\frac{(\kappa_0^2-\kappa_\infty^2)}{4}(z-1)-c\bigg).
\end{gather*}
Deformed Heun $\big(\frac32\un ;\un\big)$ equation:
\begin{gather*}
\partial_z^2+\bigg(\frac{1}{z\!-1}+\frac{1\!-\kappa_0}{z}-\frac{1}{z\!-\lambda}\bigg)\partial_z
 +\frac{1}{(z-1)^2z}\bigg({-}\frac{ t }{(z-1)}+\frac{t }{(\lambda-1)}+\frac{(\kappa_0^2-\kappa_\infty^2)}{4}(z-\lambda)
\\ \qquad
-\lambda(\lambda-1)^2\mu^2+
 \big(\lambda(\lambda-1)+\kappa_0(\lambda-1)^2\big)\mu
 +\frac{(\lambda-1)^2\lambda\mu}{(z-\lambda)}\bigg).
\end{gather*}
Type Aq compatibility functions:
\begin{gather*}
a(z)=\frac{(\lambda-1) z(z-1)}{t(z-\lambda)},\qquad
 b(z)=-\frac{\lambda(\lambda-1)^2\mu}{t(z-\lambda)}.
 \end{gather*}
Painlev\'e $\big(\frac32\un ;\un \big)$ Hamiltonian:
\begin{gather}
 tH=\lambda(\lambda-1)^2\mu^2-\kappa_0(\lambda-1)^2\mu+
\frac{(\kappa_0^2-\kappa_\infty^2)(\lambda-1)}{4}
-\frac{ t\lambda}{(\lambda-1)}. \label{P5dh}
 \end{gather}
Painlev\'e $(\frac32\un ;\un )$ equation:
\begin{gather}\label{P5d}
\frac{\d^2\lambda}{\d t^2}= \bigg(\frac{1}{2\lambda}+
\frac{1}{\lambda-1}\bigg)\bigg(\frac{\d\lambda}{\d t}\bigg)^2-\frac{1}{t} \frac{\d \lambda}{\d t}+\frac{(\lambda-1)^2}{t^2}\bigg(\alpha \lambda+\frac{\beta}{\lambda}\bigg)
-2\frac{\lambda}{t},
\end{gather}
where
\begin{gather*}
\alpha=\frac{1}{2}\kappa_{\infty}^2,\qquad
\beta=-\frac{1}{2}\kappa_0^2.
\end{gather*}
According to the standard terminology~(\ref{P5d}) is
the degenerate Painlev\'e V equation.

\subsection{From Heun (2;2) to non-degenerate Painlev\'e III$'$}

Set
\begin{gather*}
 \sigma(z)=z^2,\qquad
 \tau(z)=t+(2-\chi_0)z- z^2,\qquad
 \eta(z)=\frac{(\chi_0+\chi_\infty-1)z}{2}-c.
\end{gather*}
Heun $(2;2)$ equation:
\begin{gather*}
\partial_z^2+\bigg(\frac{t}{z^2}+\frac{2-\chi_0}{z}-1\bigg)\partial_z
 +\frac{(\chi_0+\chi_\infty-1)}{2z}-\frac{c}{z^2}.
\end{gather*}
Deformed Heun $(2;2)$ equation:
\begin{gather*}
\partial_z^2+ \bigg(\frac{ t}{z^2}+\frac{2-\chi_0}{z}-1-\frac{1}{z-\lambda}\bigg)\partial_z
\\ \qquad
{}+\frac{1}{z^2}\bigg(\frac{(\chi_0+\chi_\infty-1)}{2}(z-\lambda)
-\lambda^2\mu^2-\big(t-\chi_0\lambda-\lambda^2\big)\mu+\frac{\lambda^2\mu}{(z-\lambda)}\bigg).
\end{gather*}
Type Ap compatibility functions:
\begin{gather*}
a(z):=\frac{\lambda z}{t(z-\lambda)},\qquad b(z)=-\frac{\lambda^2\mu}{t(z-\lambda)}.
\end{gather*}
Painlev\'e $(2;2)$ Hamiltonian:
\begin{gather}\label{P3h}
t H:= \lambda^2\mu^2-\big(\lambda^2+(\chi_0-1) \lambda- t\big)\mu+\frac{1}{2}(\chi_0+\chi_{\infty}-1)\lambda.
\end{gather}
Painlev\'e $(2;2)$ equation:
\begin{gather}\label{P3p}
\frac{\d^2\lambda}{\d t^2}=\frac{1}{\lambda}\bigg(\frac{\d\lambda}{\d t}\bigg)^2-\frac{1}{t}\frac{\d\lambda}{\d t}+
\frac{\alpha\lambda^2}{4t^2}+\frac{\lambda^3}{t^2}+\frac{\beta}{4t}-\frac{1}{\lambda},
\end{gather}
where
\begin{gather*}
\alpha=-4\chi_{\infty},\qquad
\beta=4\chi_0.
\end{gather*}
According to the standard terminology
(\ref{P3p}) could be called
the nondegenerate Painlev\'e III$'$ equation.

\subsection[From Heun $\big(2;\frac32\big)$ to degenerate Painlev\'e III$'$]
{From Heun $\boldsymbol{\big(2;\frac32\big)}$ to degenerate Painlev\'e III$'$}

We~set
\begin{gather*}
\sigma(z):=z^2,\qquad
 \tau(z)=t+z(2-\chi_0),\qquad
 \eta(z)=\frac{ z}{2}-c.
\end{gather*}
 Heun $\big(2;\frac32\big)$ equation:
\begin{gather*}
\partial_z^2+\bigg(\frac{t}{z^2}+\frac{2-\chi_0}{z}\bigg)\partial_z
 +\frac{1}{z^2}\bigg(\frac{1}{2}z-c\bigg).
\end{gather*}
Deformed Heun $\big(2;\frac32\big)$ equation:
\begin{gather*}
\partial_z^2+\bigg(\frac{t}{z^2}+\frac{2-\chi_0}{z}-\frac{1}{z-\lambda}\bigg)\partial_z
 +\frac{1}{z^2}\bigg(\frac{(z-\lambda)}{2}-\lambda^2\mu^2
-\big(t-\chi_0\lambda\big)\mu +\frac{\lambda^2\mu}{(z-\lambda)}\bigg).
\end{gather*}
Type Ap compatibility functions
\begin{gather*}
a(z):=\frac{\lambda z}{t(z-\lambda)},\qquad b(z)=-\frac{\lambda^2\mu}{t(z-\lambda)}.
\end{gather*}
Painlev\'e $\big(2;\frac32\big)$ Hamiltonian:
\begin{gather} \label{P3p1H}
tH=\lambda^2\mu^2+\big(1-\chi_0)\lambda+t\big)\mu+\frac{\lambda}{2}.
\end{gather}
Painlev\'e $\big(2;\frac32\big)$ equation:
\begin{gather}
\label{P3p1}
\frac{\d^2\lambda}{\d t^2}=\frac{1}{\lambda}\bigg(\frac{\d\lambda}{\d t}\bigg)^2-\frac{1}{t}\frac{\d\lambda}{\d t}-
\frac{\lambda^2}{t^2}+\frac{\beta}{4t}-\frac{1}{\lambda},
\end{gather}
where
\begin{gather*}
\beta=4\chi_0.
\end{gather*}
According to the standard terminology~(\ref{P3p1}) is
one of the forms of the degenerate Painlev\'e III$'$ equation.

\subsection[From Heun $\big(\frac32;2\big)$ to degenerate Painlev\'e III$'$]
{From Heun $\boldsymbol{\big(\frac32;2\big)}$ to degenerate Painlev\'e III$'$}
Set
\begin{gather*}
 \sigma(z)=z^2,\qquad
 \tau(z)=- z^2+z,\qquad
 \eta(z)=\frac{t}{2z}-c+\frac{\chi_\infty z}{2}.
\end{gather*}
Heun $\big(\frac32;2\big)$ equation:
\begin{gather*}
 \partial_z^2+\bigg({-}1 +\frac1z\bigg)\partial_z+
 \frac{1}{z^2}\bigg(\frac{t}{2z}-c+\frac{\chi_\infty}{2}z\bigg).
\end{gather*}
Deformed Heun $\big(\frac32;2\big)$ equation:
\begin{gather*}
\partial_z^2+\bigg({-}1+\frac{1}{z}-\frac{1}{z-\lambda}\bigg)\partial_z
 +\frac{1}{z^2}\bigg(\frac{t}{2z}- \frac{t}{2\lambda}+\frac{\chi_\infty}{2}(z-\lambda)
 -\lambda^2\mu^2 +\big(\lambda^2+\lambda\big)\mu+\frac{\mu\lambda^2}{(z-\lambda)}\bigg).
\end{gather*}
Type Aq compatibility functions:
\begin{gather*}
a(z):=\frac{\lambda z}{t(z-\lambda)},\qquad b(z)=-\frac{\lambda^2\mu}{t(z-\lambda)}.
\end{gather*}
Painlev\'e $\big(\frac32;2\big)$ Hamiltonian
\begin{gather}\label{P3p2H}
 tH= \lambda^2\mu^2-\lambda^2\mu+\frac{\chi_\infty \lambda}{2}+\frac{t}{2\lambda}.
 \end{gather}
Painlev\'e $\big(\frac32;2\big)$ equation:
\begin{gather}\label{P3p2}
\frac{\d^2\lambda}{\d t^2}=\frac{1}{\lambda}\bigg(\frac{\d\lambda}{\d t}\bigg)^2-\frac{1}{t}\frac{\d\lambda}{\d t}+
\frac{\alpha\lambda^2}{4t^2}+\frac{\lambda^3}{t^2}+\frac{1}{t},
\end{gather}
where
\begin{gather*}
\alpha=-4\chi_{\infty}.
\end{gather*}
According to the standard terminology~(\ref{P3p1}) is
one of the forms of the degenerate Painlev\'e III$'$ equation.

\subsection[From Heun $\big(\frac32;\frac32\big)$ to doubly degenerate Painlev\'e III$'$]
{From Heun $\boldsymbol{\big(\frac32;\frac32\big)}$ to doubly degenerate Painlev\'e III$'$}
We~set
\begin{gather*}
 \sigma(z)=z^2,\qquad
 \tau(z)=2z,\qquad
 \eta(z)=\frac{z}{2}-c+\frac{t}{2z}.
\end{gather*}
Heun $\big(\frac32;\frac32\big)$ equation
\begin{gather*}
 \partial_z^2+\frac2z\partial_z
 +\frac{1}{z^2}\Big(\frac{1}{2}z-c+\frac{t}{2z}\Big).
 \end{gather*}
Deformed Heun $\big(\frac32;\frac32\big)$ equation:
\begin{gather*}
\partial_z^2+\bigg(\frac{2}{z}-\frac{1}{z-\lambda}\bigg)\partial_z
 +\frac{1}{z^2}\bigg(\frac{t}{2z} - \frac{t}{2\lambda} +\frac{1}{2}(z-\lambda)
 -\lambda^2\mu^2 +\frac{\mu\lambda^2}{(z-\lambda)}\bigg).
\end{gather*}
Type Aq compatibility functions:
\begin{gather*}
a(z):=\frac{\lambda z}{t(z-\lambda)},\qquad
b(z)=-\frac{\lambda^2\mu}{t(z-\lambda)}.
\end{gather*}
Painlev\'e $\big(\frac32;\frac32\big)$ Hamiltonian
\begin{gather}\label{P3pph}
 tH= \lambda^2\mu^2+\lambda\mu+\frac{\lambda}{2}+\frac{t}{2\lambda}.
\end{gather}
Painlev\'e $\big(\frac32;\frac32\big)$ equation:
\begin{gather}\label{P3pp}
\frac{\d^2\lambda}{\d t^2}=\frac{1}{\lambda}\bigg(\frac{\d\lambda}{\d t}\bigg)^2-\frac{1}{t}\frac{\d\lambda}{\d t}-
\frac{\lambda^2}{t^2}+\frac{1}{t}.
\end{gather}
Following the standard terminology
(\ref{P3pp}) could be called the doubly degenerate Painlev\'e III$'$ equation.

\subsection{From Heun (\un;3) to Painlev\'e IV}

Set
\begin{gather*}
 \sigma(z)=z,\qquad
 \tau(z)=1-\kappa_0-tz-\frac{z^2}{2},\qquad
 \eta(z)=\frac{\theta_\infty z}{2}-c.
\end{gather*}
Heun $(\un ;3)$ equation
\begin{gather*}
 \partial_z^2+\bigg(\frac{1-\kappa_0}{z}-t-\frac{z}{2}\bigg)\partial_z+
 \frac{1}{z}\bigg( \frac{\theta_\infty}{2}z-c\bigg).
 \end{gather*}
Deformed Heun $(\un ;3)$ equation
\begin{gather*}
\partial_z^2\!+ \bigg(\frac{1\!-\kappa_0}{z}\!-t\!-\frac{z}{2}\!-\frac{1}{z\!-\lambda}\bigg)\partial_z
\!+\frac{1}{z}\bigg(\frac{\theta_\infty(z\!-\lambda)}{2}\!-\lambda\mu^2 \!+\bigg(\frac{1}{2}\lambda^2\!+t\lambda\!+\kappa_0\bigg)\mu\!+\frac{\lambda\mu}{(z\!-\lambda)}\bigg).
\end{gather*}
Type Bp compatibility functions:
\begin{gather*}
a(z):=\frac{2 z}{(z-\lambda)},\qquad b(z)=-\frac{2\lambda\mu}{(z-\lambda)}.
\end{gather*}
Painlev\'e $(\un ;3)$ Hamiltonian:
\begin{gather}\label{P4h}
H= 2\lambda\mu^2-(\lambda^2+2t \lambda+2\kappa_0)\mu+\theta_{\infty}\lambda.
\end{gather}
Painlev\'e $(\un ;3)$ equation
\begin{gather}\label{P4}
\frac{\d^2\lambda}{\d t^2}= \frac{1}{2\lambda} \bigg(\frac{\d\lambda}{\d t}\bigg)^2+\frac{3}{2}\lambda^3+4t \lambda^2+2\big(t^2-\alpha\big)\lambda+\frac{\beta}{\lambda},
\end{gather}
where
\begin{gather*}
\alpha=-\kappa_0+2\theta_{\infty}-1,\qquad
\beta=-2\kappa_0^2.
\end{gather*}
In~the standard terminology~(\ref{P4}) is called
the Painlev\'e IV equation.

\subsection[From Heun (\un;$\frac{5}{2}$) to Painlev\'e 34]
{From Heun $\boldsymbol{\big(\un;\frac{5}{2}\big)}$ to Painlev\'e 34}

Set
\begin{gather*}
 \sigma(z)=z,\qquad
 \tau(z)=1-\kappa_0,\qquad
 \eta(z)= -\frac12z^2-\frac{1}{2}tz-c.
\end{gather*}
Heun $\big(\un ;\frac{5}{2}\big)$ equation:
\begin{gather*}
\partial_z^2+\frac{1-\kappa_0}{z}\partial_z+\frac{1}{z}\bigg({-}\frac{1}{2}z^2-\frac{tz}{2}-c\bigg).
\end{gather*}
Deformed Heun $\big(\un ;\frac{5}{2}\big)$ equation:
\begin{gather*}
 \partial_z^2+\bigg(\frac{1-\kappa_0}{z}-\frac{1}{z-\lambda}\bigg)
 +\frac{1}{z}\bigg({-}\frac{z^2}{2}+\frac{\lambda^2}{2}-\frac{t}{2}(z-\lambda) -\lambda\mu^2+\kappa_0\mu+\frac{\lambda\mu}{z-\lambda}\bigg).
\end{gather*}
Type Bq compatibility functions:
\begin{gather*}
a(z)=\frac{z}{z-\lambda},\qquad
b(z)=-\frac{\lambda\mu}{z-\lambda}.
\end{gather*}
Painlev\'e $\big(\un ;\frac{5}{2}\big)$ Hamiltonian
\begin{gather}\label{P34h}
 H=\lambda\mu^2-\kappa_0\mu-\frac{\lambda^2}{2}-\frac{t\lambda}{2}.
\end{gather}
Painlev\'e $\big(\un ;\frac{5}{2}\big)$ equation:
\begin{gather}
\frac{\d^2\lambda}{\d t^2}=\frac{1}{2\lambda}\Big(\frac{\d\lambda}{\d t}\Big)^2+2\lambda^2+t\lambda-\frac{\alpha}{2\lambda},\label{P34}
\end{gather}
where $\alpha=\kappa_0^2$.

According to \cite{Ohyama}, the standard name
for~(\ref{P34}) is the Painlev\'e 34 equation.

\subsection{From Heun (;4) to Painlev\'e II}

Set
\begin{gather*}
 \sigma(z)=1,\qquad
 \tau(z)=-2z^2-t,\qquad
 \eta(z)=-(2\alpha+1)z-c.
\end{gather*}
Heun $(;4)$ equation:
\begin{gather*}
 \partial_z^2-\big(2z^2+t\big)\partial_z-(2\alpha+1)z-c.
\end{gather*}
Deformed Heun $(;4)$ equation:
\begin{gather*}
\partial_z^2+\bigg({-}2z^2-t-\frac{1}{z-\lambda}\bigg)\partial_z
-(2\alpha+1)(z-\lambda)-\mu^2+\big(2\lambda^2+t\big)\mu+\frac{\mu}{z-\lambda}.
\end{gather*}
Type Bp compatibility functions scaled with $\epsilon=-1$:
\begin{gather*}
a(z):=\frac{1}{2(z-\lambda)},\qquad
b(z)=-\frac{\mu}{2(z-\lambda)}.
\end{gather*}
Painlev\'e $(;4)$ Hamiltonian:
\begin{gather}
 \label{H2}
H=\frac{1}{2}\mu^2-\bigg(\lambda^2+\frac{t}{2}\bigg)\mu-\bigg(\alpha+\frac{1}{2}\bigg)\lambda.
\end{gather}
Painlev\'e $(;4)$ equation:
\begin{gather}\label{P2}
\frac{\d^2\lambda}{\d t^2}= 2\lambda^3+t \lambda+\alpha.
\end{gather}
According to the standard terminology~(\ref{P2}) is called
the Painlev\'e II equation

\subsection[From Heun (;$\frac72$) to Painlev\'e I]
{From Heun $\boldsymbol{\big(;\frac72\big)}$ to Painlev\'e I}\label{painleve1}

Set
\begin{gather*}
 \sigma(z)=1,\qquad
 \tau(z)=0,\qquad
 \eta(z)=-4z^3-2tz-c.
\end{gather*}
Heun $\big(;\frac72\big)$ equation:
\begin{gather*}
 \partial_z^2-4z^3-2tz-c.
\end{gather*}
Deformed Heun $\big(;\frac72\big)$ equation:
\begin{gather*} \partial_z^2-\frac{1}{(z-\lambda)}\partial_z-4z^3-2tz+4\lambda^3+2t\lambda-\mu^2+\frac{\mu}{(z-\lambda)}.
\end{gather*}
Type Bq compatibility functions scaled with $\epsilon=-6$:
\begin{gather*}
a(z):=\frac{1}{2(z-\lambda)},\qquad
b(z)=-\frac{\mu}{2(z-\lambda)}.
\end{gather*}
Painlev\'e $\big(;\frac72\big)$ Hamiltonian:
\begin{gather}
H=\frac{1}{2}\mu^2-2\lambda^3-t\lambda.\label{P1.h}
\end{gather}
Painlev\'e $\big(;\frac72\big)$ equation:
\begin{gather}\label{P1.}
\frac{\d^2\lambda}{\d t^2}=6\lambda^2+t.
 \end{gather}
In~the standard terminology~(\ref{P1.}) is called the Painlev\'e I equation.

\section{Five supertypes of Painlev\'e equation}\label{FivesupertypesofPainleveequation}

\subsection{Overview of five supertypes}

Recall that the ten types of the Heun class equation can be grouped into
five supertypes, as described in Section~\ref{Classification of Heun class equations}.
The ten types of Painlev\'e equation
can be also grouped into five supertypes. There is an exact correspondence between the supertypes of Heun class and Painlev\'e equations:

\begin{itemize}\itemsep=0pt
\item Painlev\'e VI or $(\un \un \un \un )$.
\item Painlev\'e V or $(\un \un \underline2)$.
 \begin{itemize}\itemsep=0pt
 \item Painlev\'e ndeg-V or $(\un \un 2)$.
 \item Painlev\'e deg-V or $\big(\un \un \frac32\big)$.
 \end{itemize}
\item Painlev\'e III$'$ or $(\underline2\underline2)$.
 \begin{itemize}\itemsep=0pt \item Painlev\'e ndeg-III$'$ or $(22)$.
 \item Painlev\'e deg-III$'$ or $\big(\frac322\big)$.
 \item Painlev\'e ddeg-III$'$ or $\big(\frac32\frac32\big)$.
 \end{itemize}
 \item Painlev\'e IV-34 or $(\un \underline3)$.
\begin{itemize}\itemsep=0pt \item Painlev\'e IV or $(\un 3)$.
 \item Painlev\'e 34 or $\big(\un \frac52\big)$.
\end{itemize}
 \item Painlev\'e I--II or $(\underline4)$.
\begin{itemize}\itemsep=0pt \item Painlev\'e II or $(4)$.
 \item Painlev\'e I or $\big(\frac72\big)$.
\end{itemize}
\end{itemize}

In~what follows we discuss this classification.
We~describe the minimal set of parameters that can be used in a given type and various equivalences. In~our discussion we
try to include the Hamiltonian aspect, whenever it is possible.

The above classification of Painlev\'e equation was pointed out by Ohyama--Okumura, see the beginning of Section 2 of \cite{Ohyama}.
(In~that reference the authors use the word ``type'' both for what we call ``supertype'' and ``type''.)
The discussion in \cite{Ohyama}, however,
concentrated on the second order equations. Less space was devoted to
 the
Hamiltonian form of the five supertypes.

The first supertype is Painlev\'{e} VI or $(\un \un \un \un )$, which contains only one type. All the four remaining supertypes contain at least two types. We~discuss them in the following subsections.

 In~each of the following subsections we start with a general form of
the given supertype of the Painlev\'e equation. It will be
indicated by $\diamond$. It always has the form
\begin{gather*}
\frac{\d^2\lambda}{\d t^2}=F\bigg(t,\frac{\d\lambda}{\d t},\lambda\bigg).
\end{gather*}
The corresponding differential (nonlinear) operator will be denoted
\begin{gather*}
P(t,\lambda):=\frac{\d^2\lambda}{\d t^2}-F\bigg(t,\frac{\d\lambda}{\d t},\lambda\bigg).
\end{gather*}
We~also introduce the corresponding Hamiltonian~$H(t,\lambda,\mu).$
Both $P(t,\lambda)$ and $H(t,\lambda,\mu)$ depend on several parameters, put as subscripts.
Next we give the scaling properties of the equation and the Hamiltonian.

Then we list various nontrivial types that belong to a given
supertype, marking them with $\ast$.

Finally, we discuss the relationship between various types. In~particular, show how to reduce the number of parameters using scaling.

Each supertype contains one generic type, which we call non-degenerate.
Besides, it may contain one or more degenerate types. The Hamiltonian that covers the non-degenerate type, does not always allow us to describe all
types that belong to a given supertype.
This can be viewed as a drawback of the Hamiltonian approach.

\subsection{Painlev\'e V or (\un \un \underline{2})}\label{PainleveV}

As noted in \cite{Ohyama}, the usual form of the Painlev\'e V equation, depending on 4 parameters, should be treated not as a single type, but as a supertype. It is invariant with respect to a scaling transformation.
It includes two nontrivial types: nondegenerate V depending on 3 parameters and degenerate V depending on 2 parameters. There exists also a trivial type, solvable in quadrature.

In~this subsection we discuss the supertype Painlev\'e V in detail. Note that we prefer to denote it by $(\un \un \underline{2})$, since it corresponds to
the supertype $(\un \un \underline{2})$ of the Heun class.

\begin{itemize}\itemsep=0pt
\item[$\diamond$] Painlev\'e V or $(\un \un \underline{2})$ equation and Hamiltonian
\begin{gather}
\frac{\d^2\lambda}{\d t^2}= \bigg(\frac{1}{2\lambda}\!+
\frac{1}{\lambda\!-1}\bigg)\bigg(\frac{\d\lambda}{\d t}\bigg)^2-\frac{1}{t} \frac{\d \lambda}{\d t}+\frac{(\lambda\!-1)^2}{t^2}\bigg(\alpha \lambda+\frac{\beta}{\lambda}\bigg)
+\gamma\frac{\lambda}{t}+\delta\frac{\lambda(\lambda+1)}{\lambda-1},\!\!
\label{painleve5}
\\
\nonumber
tH= (\lambda-1)^2\lambda\mu^2-\big(\kappa_0(\lambda-1)^2+(\chi_1-1) \lambda(\lambda-1)-\eta t \lambda\big)\mu
\\ \hphantom{\frac{\d^2\lambda}{\d t^2}=}
{}+\frac{\big((\kappa_0+\chi_1-1)^2-\kappa_\infty^2\big)(\lambda-1)}{4},
\label{PH5}
\end{gather}
where $\alpha=\frac{1}{2}\kappa_{\infty}^2$,
$\beta=-\frac{1}{2}\kappa_0^2$, $\gamma=\chi_1\eta$, $\delta=-\frac{1}{2}\eta^2$.\vspace{1ex}

Scaling properties
\begin{gather*}
 \epsilon^2 P_{\alpha,\beta,\gamma,\delta}(\epsilon t,\lambda)
 = P_{\alpha,\beta,\epsilon\gamma,\epsilon^2\delta}(t,\lambda),
 \\
\epsilon H_{\kappa_0,\kappa_\infty,\chi_1,\eta}(\epsilon t,\lambda,\mu)=
H_{\kappa_0,\kappa_\infty,\chi_1,\frac\eta\epsilon}( t,\lambda,\mu).
\end{gather*}

\item[$\ast$] Painlev\'e ndeg-V or $(2\un \un )$ equation and Hamiltonian recalled from~(\ref{P5}) and~(\ref{P5h}):
\begin{gather*}
\frac{\d^2\lambda}{\d t^2}= \bigg(\frac{1}{2\lambda}+
\frac{1}{\lambda-1}\bigg)\bigg(\frac{\d\lambda}{\d t}\bigg)^2-\frac{1}{t} \frac{\d \lambda}{\d t}+\frac{(\lambda-1)^2}{t^2}\bigg(\alpha \lambda+\frac{\beta}{\lambda}\bigg)
+\gamma\frac{\lambda}{t}-\frac{\lambda(\lambda+1)}{2(\lambda-1)},
\\
tH= (\lambda-1)^2\lambda\mu^2-\big(\kappa_0(\lambda-1)^2+(\chi_1-1) \lambda(\lambda-1)- t \lambda\big)\mu
\\ \hphantom{tH= }
{}+\frac{\big((\kappa_0+\chi_1-1)^2-\kappa_\infty^2\big)(\lambda-1)}{4},
\end{gather*}
where $\alpha=\frac{1}{2}\kappa_{\infty}^2$, $\beta=-\frac{1}{2}\kappa_0^2$, $\gamma=\chi_1$.\vspace{1ex}

\item[$\ast$] Painlev\'e deg-V or $\big(\frac32\un \un \big)$ equation and Hamiltonian, recalled from~(\ref{P5d}) and~(\ref{P5dh}):
\begin{gather}
\frac{\d^2\lambda}{\d t^2}= \bigg(\frac{1}{2\lambda}+
\frac{1}{\lambda-1}\bigg)\bigg(\frac{\d\lambda}{\d t}\bigg)^2-\frac{1}{t} \frac{\d \lambda}{\d t}+\frac{(\lambda-1)^2}{t^2}\bigg(\alpha \lambda+\frac{\beta}{\lambda}\bigg)-2\frac{\lambda}{t},\nonumber
\\
 tH=\lambda(\lambda-1)^2\mu^2-\kappa_0(\lambda-1)^2\mu+
\frac{(\kappa_0^2-\kappa_\infty^2)(\lambda-1)}{4}-\frac{ t\lambda}{(\lambda-1)},
\label{P5dh-}
\end{gather}
where
$\alpha=\frac{1}{2}\kappa_{\infty}^2$, $\beta=-\frac{1}{2}\kappa_0^2$.\vspace{1ex}
\end{itemize}

Let us discuss special cases:
\begin{itemize}\itemsep=0pt
\item Let
 $\delta\neq0$. In~the Hamiltonian form it corresponds to $\eta\neq0$.
 By scaling we can set $\delta=-\frac12$, and in the Hamiltonian form $\eta=1$.
 We~obtain the Painlev\'e $(\un\un2)$ equation and Hamiltonian.
\item Let $\delta=0$, $\gamma\neq0$.
 By scaling we can set $\gamma=-2$.
 We~obtain the Painlev\'e $\big(\un \un \frac32\big)$ equation.

 However, on the Hamiltonian level this reduction does not work: we cannot directly use~(\ref{PH5}) to obtain the Painlev\'e $\big(\un \un \frac32\big)$ Hamiltonian.
\item Let $\delta=0$, $\gamma=0$. On the Hamiltonian level, $\eta=0$. The Hamiltonian becomes
\begin{gather}\nonumber
 tH= (\lambda-1)^2\lambda\mu^2-\big(\kappa_0(\lambda-1)^2+(\chi_1-1) \lambda(\lambda-1)\big)\mu
 \\ \hphantom{ tH=}
 {}+\frac{\big((\kappa_0+\chi_1-1)^2-\kappa_\infty^2\big)(\lambda-1)}{4}.\label{PH5-}
 \end{gather}
This case is solvable in quadratures by the method of Section~\ref{hammo1}.

Note that the corresponding 2nd order equation does not depend on the parameter $\chi_1$. On the Hamiltonian level it can be seen by using the canonical transformation
$\tilde\mu=\mu-\frac{\chi_1-1}{2(\lambda-1)}$, which transforms~(\ref{PH5-}) into
\begin{gather*}
 tH=(\lambda-1)^2\lambda\tilde\mu^2-\kappa_0(\lambda-1)^2\tilde\mu +\frac{(\kappa_0^2-\kappa_\infty^2)(\lambda-1)}{4}-\frac{(\chi_1-1)^2}{4},
 \end{gather*}
where the dependence on $\chi_1$ remains only in the free term.
\end{itemize}

\begin{Remark}
It is well known that
 the Painlev\'e deg-V or $\big(\un \un \frac32\big)$ and ndeg-III$'$ or $(22)$ equations are equivalent~\cite{Ohyama}. Below we will show this by describing a canonical transformation that connects the corresponding Hamiltonians.

Let us insert the canonical transformation
\begin{gather*}
 \tilde\lambda=1-\frac1\mu,\quad\
 \mu=\frac{1}{1-\tilde\lambda},\quad\
 \tilde\mu=\mu^2\lambda-\frac{(\chi_0-1)\mu}{2},\quad\
 \lambda =\big(1-\tilde\lambda\big)^2\tilde\mu+\frac{(\chi_0-1)}{2}\big(1-\tilde\lambda\big),
\end{gather*}
into the $(22)$ Hamiltonian~(\ref{P3h-}). We~obtain
\begin{gather*}
tH= \tilde\lambda\big(\tilde\lambda\!-1\big)^2\tilde\mu^2\! +\!\frac{-\chi_0\!+\chi_\infty\!+1}{2}\big(\tilde\lambda\!-1\big)^2\tilde\mu-
\frac{\chi_\infty(\chi_0\!-1)}{2}\big(\tilde\lambda\!-1\big)
\!-\!\frac{ t\tilde\lambda}{\big(\tilde\lambda\!-1\big)}
\!-\!\frac{(\chi_0\!-1)^2}{4}\!+t,
\end{gather*}
which after appropriate identification of parameters coincides with
the $\big(\un \un \frac32\big)$ Hamiltonian~(\ref{P5dh-}) up to a free term.
\end{Remark}

\subsection{Painlev\'e III$'$ or (\underline{2}\underline{2})}
\label{PainleveIII'}

As noted in \cite{Ohyama},
the usual Painlev\'e III$'$ equation, depending on 4 parameters, should be treated as a supertype. It is invariant with respect to two distinct scaling transformations. It includes 3 nontrivial types: nondegenerate III$'$ depending on 2 parameters, degenerate III$'$ (in two forms) depending on 1 parameter, doubly degenerate III$'$ with no parameters. There are also trivial forms solvable in quadratures.

 In~this subsection we discuss Painlev\'e III$'$ in detail.
We~prefer to denote it by the sym\-bol~$(\underline2\underline2)$, because it corresponds to the supertype $(\underline2\underline2)$ of the Heun class.

\begin{itemize}\itemsep=0pt
\item[$\diamond$] Painlev\'e III$'$ or $(\underline{2}\underline{2})$ equation and Hamiltonian:
 \begin{gather}
\frac{\d^2\lambda}{\d t^2}=\frac{1}{\lambda}\bigg(\frac{\d\lambda}{\d t}\bigg)^2-\frac{1}{t}\frac{\d\lambda}{\d t }+
\frac{\alpha\lambda^2+\gamma\lambda^3}{4t^2}+\frac{\beta}{4t}+\frac{\delta}{4\lambda},\nonumber
\\
 t H= \lambda^2\mu^2-\big(\eta_{\infty}\lambda^2+(\chi_0-1) \lambda-\eta_0 t\big)\mu+\frac{1}{2}\eta_{\infty}(\chi_0+\chi_{\infty}-1)\lambda,
\label{painleve3} 
\end{gather}
where
$\alpha=-4\eta_{\infty}\chi_{\infty}$, $\beta=4\eta_0\chi_0$, $\gamma=4\eta_{\infty}^2$, $\delta=-4\eta_0^2$.
Scaling properties
\begin{gather*}
 \frac{ \epsilon^{2}}{\omega} P_{\alpha,\beta,\gamma,\delta}(\epsilon t,\omega\lambda) = P_{\omega\alpha,\frac{\epsilon}{\omega}\beta,\omega^2\gamma,\frac{\epsilon^2}{\omega^2}\delta}(t,\lambda),
 \\
\epsilon H_{\eta_0,\eta_\infty,\chi_0,\chi_\infty}\bigg(\epsilon t,\omega\lambda,\frac{\mu}{\omega}\bigg)
= H_{\frac{\epsilon}{\omega}\eta_0,\omega\eta_\infty,\chi_0,\chi_\infty}( t,\lambda,\mu).
\end{gather*}

\item[$\ast$] Painlev\'e ndeg-III$'$ or $(22)$ equation and Hamiltonian, recalled from~(\ref{P3h}) and~(\ref{P3p}):
\begin{gather}
\frac{\d^2\lambda}{\d t^2}=\frac{1}{\lambda}\bigg(\frac{\d\lambda}{\d t}\bigg)^2-\frac{1}{t}\frac{\d\lambda}{\d t}+
\frac{\alpha\lambda^2}{4t^2}+\frac{\lambda^3}{t^2}+\frac{\beta}{4t}-\frac{1}{\lambda},\nonumber
\\ \label{P3h-}
t H:= \lambda^2\mu^2-\big(\lambda^2+(\chi_0-1) \lambda- t\big)\mu+\frac{1}{2}(\chi_0+\chi_{\infty}-1)\lambda,
\end{gather}
where
$\alpha=-4\chi_{\infty}$, $\beta=4\chi_0$.

\item[$\ast$] Painlev\'e deg-III$'$-1 or $\big(2;\frac32\big)$ equation and Hamiltonian, recalled from~(\ref{P3p1}) and~(\ref{P3p1H}):
\begin{gather*}
\frac{\d^2\lambda}{\d t^2}=\frac{1}{\lambda}\bigg(\frac{\d\lambda}{\d t}\bigg)^2-\frac{1}{t}\frac{\d\lambda}{\d t}-
\frac{\lambda^2}{t^2}+\frac{\beta}{4t}-\frac{1}{\lambda},\nonumber
\\
tH=\lambda^2\mu^2+\big((1-\chi_0)\lambda+t\big)\mu+\frac{\lambda}{2},
 \end{gather*}
where $\beta=4\chi_0$.

\item[$\ast$] Painlev\'e deg-III$'$-2 or $\big(\frac32;2\big)$ equation and Hamiltonian recalled from
(\ref{P3p2}) and~(\ref{P3p2H}):
\begin{gather*}
\frac{\d^2\lambda}{\d t^2}=\frac{1}{\lambda}\bigg(\frac{\d\lambda}{\d t}\bigg)^2-\frac{1}{t}\frac{\d\lambda}{\d t}+
\frac{\alpha\lambda^2}{4t^2}+\frac{\lambda^3}{t^2}+\frac{1}{t},
\\
 tH= \lambda^2\mu^2-\lambda^2\mu+\frac{\chi_\infty \lambda}{2}+\frac{t}{2\lambda},
 \end{gather*}
where
$\alpha=-4\chi_{\infty}.$

\item[$\ast$] Painlev\'e ddeg-III$'$ or $\big(\frac32\frac32\big)$ equation and
Hamiltonian, recalled from~(\ref{P3pp}) and~(\ref{P3pph}):
\begin{gather*}
\frac{\d^2\lambda}{\d t^2}=\frac{1}{\lambda}\bigg(\frac{\d\lambda}{\d t}\bigg)^2-\frac{1}{t}\frac{\d\lambda}{\d t}-\frac{\lambda^2}{t^2}+\frac{1}{t},
\\
 tH= \lambda^2\mu^2+\lambda\mu+\frac{\lambda}{2}+\frac{t}{2\lambda}.
\end{gather*}
\end{itemize}

Let us discuss special cases:
\begin{itemize}\itemsep=0pt
\item
 Let
$\delta\neq0$, $\gamma\neq0$. On the level of the Hamiltonian it means $\eta_0\neq0$, $\eta_\infty\neq0$. By scaling we
can set on the Hamiltonian level $\eta_0=1$, $\eta_\infty=1$, which corresponds to $\gamma=4$, $\delta=-4$. We~obtain the Painlev\'e ndeg-III$'$ or $(22)$
equation~(\ref{P3p}) and Hamiltonian~(\ref{P3h}).
\item Let $\delta,\alpha\neq0$, $\gamma=0$. By scaling we can make $\delta=-4$, $\alpha=-4$.
 We~obtain the Painlev\'e $\big(2;\frac32\big)$ equation.
 This reduction does not work for the
 Painlev\'e $\big(2;\frac32\big)$
 Hamiltonian.
\item Let $\delta=0$, $\gamma,\beta\neq0$. By scaling we can make $\gamma=4$,
 $\beta=4$.
 We~obtain the Painlev\'e $\big(\frac32;2\big)$ equation.
 This reduction does not work for the Painlev\'e $\big(\frac32;2\big)$ Hamiltonian.

 The equations $\big(\frac32;2\big)$ and $\big(2;\frac32\big)$ are two equivalent forms of $\big(2\frac32\big)$ see Proposition~\ref{equi}.
\item
 Let $\delta=\gamma=0$, $\alpha,\beta\neq0$. By scaling we can make
 $\alpha=-4$, $\beta=4$. We~obtain the Pain\-lev\'e $\big(\frac32\frac32\big)$ equation~(\ref{P3pp}). This reduction does not work for the Painlev\'e $\big(\frac32\frac32\big)$ Hamiltonian~(\ref{P3pph}).
\item Let $\alpha=\gamma=0$. On the Hamiltonian level, $\eta_\infty=0$.
 The Hamiltonian is
 \begin{gather}\label{p3sol}
 t H= \lambda^2\mu^2-\big((\chi_0-1) \lambda-\eta_0 t\big)\mu.
 \end{gather}
 We~can apply to~(\ref{p3sol}) the time-dependent canonical transformation
 \begin{gather}\label{time}
 \tilde H=H-\frac{\tilde\lambda\tilde\mu}{t},\qquad
 \tilde\lambda=\frac{\lambda}{t},
\quad \tilde\mu=t\mu,
\end{gather}
 obtaining
 \begin{gather*}
 t\tilde H= \tilde\lambda^2\tilde\mu^2-\big(\chi_0\tilde\lambda-\eta_0\big)\tilde\mu.
 \end{gather*}
 It is solvable by quadratures by Section~\ref{hammo1}.
\item
Let $\beta=\delta=0$. On the Hamiltonian level it corresponds to $\eta_0=0$. The Hamiltonian is
 \begin{gather*}
t H= \lambda^2\mu^2-\big(\eta_{\infty}\lambda^2+(\chi_0-1) \lambda\big)\mu+\frac{1}{2}\eta_{\infty}(\chi_0+\chi_\infty-1)\lambda.
 \end{gather*}
 It is solvable by quadratures by Section~\ref{hammo1}.
\end{itemize}

\begin{Proposition}\label{equi}
The Painlev\'e $\big(2;\frac32\big)$ and $\big(\frac32;2\big)$ equations are equivalent.
\end{Proposition}

\begin{proof}
First we apply to the Painlev\'e $\big(2;\frac32\big)$ Hamiltonian the time-dependent canonical transformation~(\ref{time}) obtaining
\begin{gather*}
t \tilde H=\tilde\lambda^2\tilde\mu^2+\big({-}\chi_0\tilde\lambda+1\big)\tilde\mu+\frac{t\tilde\lambda}{2}.
\end{gather*}
Next we apply the time independent canonical transformation
\begin{gather*} 
\tilde\lambda= \frac{1}{\lambda},\qquad
\tilde\mu=-\mu\lambda^2+\frac{\chi_0\lambda}{2},
\end{gather*}
obtaining
\begin{gather*}
 t\tilde H = \lambda^2\mu^2-\lambda^2\mu+\frac{\chi_0 \lambda}{2}+\frac{t}{2\lambda}-\frac{\chi_0^2}{4},
\end{gather*}
which is the Painlev\'e $\big(\frac32;2\big)$ Hamiltonian for $\chi_\infty=\chi_0$ minus $\frac{\chi_0^2}{4}$.
\end{proof}

\begin{Remark}
The standard Painlev\'e III equation is given by
\begin{gather}\label{P3}
\frac{\d^2\tilde\lambda}{\d \tilde t^2}=\frac{1}{\lambda}\bigg(\frac{\d\tilde\lambda}{\d \tilde t}\bigg)^2-\frac{1}{\tilde t}\frac{\d\tilde\lambda}{\d \tilde t}+
\frac{\alpha\tilde\lambda^2+\beta}{\tilde t}+\gamma \tilde\lambda^3+\frac{\delta}{\tilde\lambda}.
\end{gather}
The Painlev\'e III$'$ equation is obtained from~(\ref{P3}) by $\tilde\lambda=t\lambda$, $\tilde t=t^2$.
\end{Remark}

\subsection{Painlev\'e IV-34 or (\un \underline{3})}\label{Painleve IV-34}

It has been noted in \cite{Ohyama} that it is natural to consider Painlev\'e IV
together with the so-called Painlev\'e 34. The latter is equivalent to Painlev\'e II, and therefore is not so well known. Together they can be treated as special cases of a supertype, which in \cite{Ohyama} is denoted $4\underline{\;\;}34$, and we denote by IV-34, or preferably by $(\un \underline{3})$, since it is related to the supertype $(\un \underline{3})$ of the Heun class. In~this subsection we discuss this supertype of Painlev\'e in detail.

Painlev\'e IV-34 depends on 3 parameters. It is invariant with respect to a scaling transformation. It contains Painlev\'e IV depending on 2 parameters and Painlev\'e 34 depending on 1 parameter, as well as a trivial type solvable in quadratures.
\begin{itemize}\itemsep=0pt
\item[$\diamond$] Painlev\'e IV-34 or $(\un \underline{3})$ equation and Hamiltonian:
 \begin{gather}
 \frac{\d^2\lambda}{\d t^2}= \frac{1}{2\lambda}\bigg(\frac{\d \lambda}{\d t}\bigg)^2+
 \rho\lambda(2\lambda+t)+\gamma\lambda(\lambda+t)(3\lambda+t)+ \frac{\beta}{4\lambda},\nonumber
\\
 H=\lambda\mu^2-\big(\eta\lambda^2+\eta t\lambda-\theta\lambda+\kappa_0\big)\mu+
 \bigg(\frac{\theta^2}{4}+\frac{(\kappa_0-1)\eta}{2}\bigg)\lambda,
 \label{painleve434}
 \end{gather}
 where
 $\beta=-2\kappa_0^2$, $\rho=-\eta\theta$, $\gamma=\frac12\eta^2$.
Scaling properties
 \begin{gather*}
 \epsilon P_{\beta,\rho,\gamma}(\epsilon t,\epsilon\lambda)
 =P_{\beta,\epsilon^3\rho,\epsilon^4\gamma}( t,\lambda),
 \\
 \epsilon H_{\eta,\theta,\kappa_0}\big(\epsilon t,\epsilon\lambda,\epsilon^{-1}\mu\big)
 = H_{\epsilon^2\eta,\epsilon\theta,\kappa_0}(t,\lambda,\mu).
 \end{gather*}

\item[$\ast$] Painlev\'e IV or $(\un 3)$ equation and Hamiltonian, recalled from~(\ref{P4}) and~(\ref{P4h}):
\begin{gather*}
\frac{\d^2\lambda}{\d t^2}= \frac{1}{2\lambda} \left(\frac{\d\lambda}{\d t}\right)^2+\frac{3}{2}\lambda^3+4t \lambda^2+2\big(t^2-\alpha\big)\lambda+\frac{\beta}{\lambda},
\\
H= 2\lambda\mu^2-\big(\lambda^2+2t \lambda+2\kappa_0\big)\mu+\theta_{\infty}\lambda,
\end{gather*}
where
$\alpha=-\kappa_0+2\theta_{\infty}-1$, $\beta=-2\kappa_0^2$.

\item[$\ast$] Painlev\'e 34 or $\big(\un \frac{5}{2}\big)$ equation and Hamiltonian, recalled
from~(\ref{P34}) and~(\ref{P34h}):
\begin{gather}
\frac{\d^2\lambda}{\d t^2}=\frac{1}{2\lambda}\Big(\frac{\d\lambda}{\d t}\Big)^2+2\lambda^2+t\lambda+\frac{\beta}{4\lambda},\nonumber
\\
H=\lambda\mu^2-\kappa_0\mu-\frac{\lambda^2}{2}-\frac{t\lambda}{2},
\label{P34h-}
\end{gather}
where $\beta=-2\kappa_0^2$.
\end{itemize}

Let us discuss special cases:
 \begin{itemize}\itemsep=0pt
 \item
 Let $\gamma\neq0$. By scaling we can set $\gamma=\frac18$. We~change the time $\tilde t=\frac12 t+\rho$. We~obtain the Painlev\'e $(\un3)$ equation with $\alpha=4\rho^2$.

 The equivalent reduction on the Hamiltonian level: For $\eta\neq0$,
 by scaling, we can make $\eta=\frac12$. We~multiply the Hamiltonian by $2$ and change $t$ to $2t$:
 \begin{gather*}
 2 H (2t,\lambda,\mu)=2\lambda\mu^2-\big(\lambda^2+2t\lambda-2\theta\lambda+\kappa_0\big)\mu +\frac{\theta^2+\kappa_0-1}{2}\lambda,
 \end{gather*}
 which is the Painlev\'e $(\un3)$ Hamiltonian with $\theta_\infty=\frac{\theta^2+\kappa_0-1}{2}$ and $\tilde t=t-\theta$.

 \item Let $\gamma=0$, $\rho\neq0$.
 By scaling we can make $\rho=1$.
 We~obtain the Painlev\'e $34$ or $\big(\un \frac52\big)$ equation. 
 This reduction does not work for the Painlev\'e $34$ or $\big(\un \frac52\big)$ Hamiltonian.
 \item Let $\gamma=\rho=0$, which on the Hamiltonian level corresponds to $\eta=0$. We~have
 the Hamiltonian\vspace{-.5ex}
 \begin{gather*}
 H=\lambda\mu^2- (-\theta\lambda+\kappa_0)\mu+\frac{\theta^2}{4}\lambda,
 \end{gather*}
 which is solvable by quadratures (see Section~\ref{hammo1}).
 \end{itemize}

 \begin{Remark}
 Note that Painlev\'e 34 or $\big(\un\frac52\big)$ is equivalent to Painlev\'e II or $(4)$.

 Let us show this on the Hamiltonian level.
After an application of the canonical transfor\-mation\vspace{-.5ex}
 \begin{gather*}
 \tilde \mu=\lambda,\qquad \tilde\lambda=-\mu,
 \end{gather*}
 the Painlev\'e $\big(\un\frac52\big)$ Hamiltonian~(\ref{P34h-}) becomes\vspace{-.5ex}
 \begin{gather*}
 H=-\frac{\tilde\mu^2}{2}-\tilde\lambda^2\tilde\mu-\frac{t\tilde\mu}{2} +\kappa_0\tilde\lambda.
 \end{gather*}
 Then we change $t$ into $-t$ and multiply the Hamiltonian by $-1$, obtaining the Painlev\'e $(4)$ Hamiltonian~(\ref{H2-}) with $\kappa_0=\alpha+\frac12$.
\end{Remark}

\subsection{Painlev\'e I--II or (\underline{4})}
Usually Painlev\'e I and II equations are treated separately.
However, it has been noted already by Painlev\'e and elaborated in \cite{Ohyama} that it is natural to join them
in a single supertype. \cite{Ohyama}~denotes it $1\underline{\;\;}2$, we denote it I--II, or preferably $(\underline4)$, since it corresponds to the supertype $(\underline4)$
of the Heun class. In~this subsection we discuss this supertype of Painlev\'e in detail.

Painlev\'e I--II depends on 2 parameters.
It is invariant with respect to a scaling transformation. It contains Painlev\'e II, depending on 1 parameter, Painlev\'e I with no parameters and a trivial type solvable in quadratures.

\begin{itemize}\itemsep=0pt
\item[$\diamond$]
Painlev\'e I--II or $(\underline{4})$ equation and Hamiltonian:
\begin{gather*}
 \frac{\d^2\lambda}{\d t^2}=\gamma\big(2\lambda^3+t\lambda\big)+\beta\big(6\lambda^2+t\big),
 \\
 H=\frac12\mu^2-\bigg(\eta\lambda^2+\frac12\eta t\bigg)\mu-2\beta\lambda^3-t\beta\lambda-\frac12\eta\lambda,
 \end{gather*}
 where $\gamma=\eta^2$.

 Scaling properties:
 \begin{gather*}
 \epsilon^3 P_{\gamma,\beta}\big(\epsilon^2 t,\epsilon\lambda\big)=P_{\epsilon^6\gamma,\epsilon^5\beta}(t,\lambda),
 \\
 \epsilon^2 H_{\eta,\beta}\big(\epsilon^2 t,\epsilon\lambda,\epsilon^{-1}\mu\big)=
 H_{\epsilon^3\eta,\epsilon^5\beta}(t,\lambda,\mu).
\end{gather*}

\item[$\ast$]
 Painlev\'e II or $(4)$ equation and Hamiltonian, recalled from~(\ref{P2}) and~(\ref{H2}):
 \begin{gather}
\frac{\d^2\lambda}{\d t^2}= 2\lambda^3+t \lambda+\alpha,\nonumber
\\
\label{H2-}
H=\frac{1}{2}\mu^2-\bigg(\lambda^2+\frac{t}{2}\bigg)\mu-\bigg(\alpha+\frac{1}{2}\bigg)\lambda.
\end{gather}

\item[$\ast$]
 Painlev\'e I or $\big(\frac72\big)$ equation and Hamiltonian, recalled from~(\ref{P1.}) and~(\ref{P1.h}):
 \begin{gather*}
 \frac{\d^2\lambda}{\d t^2}=6\lambda^2+t,
 \\
 H=\frac{1}{2}\mu^2-2\lambda^3-t\lambda.
\end{gather*}
\end{itemize}

Let us discuss special cases:
 \begin{itemize}\itemsep=0pt\item
Let $\eta\neq0$, in both the equation and the Hamiltonian. By scaling we can set $\eta=1$.
We~apply the canonical transformation
\begin{gather*}
t=\tilde t+6\beta^2,\qquad
\lambda=\tilde\lambda-\beta,\qquad \mu=\tilde\mu-2\beta\tilde\lambda+4\beta^2.
\end{gather*}
We~obtain
\begin{gather*}
H=\frac{\tilde\mu^2}{2}-\bigg(\tilde\lambda^2+\frac{\tilde t}{2}\bigg)\mu-\bigg(2\beta^3+\frac12\bigg)\tilde\lambda+\frac12\beta-\beta^2\tilde t.
\end{gather*}
Thus up to free terms we obtain the Painlev\'e $(4)$ Hamiltonian with
 $\alpha=2\beta^3$, and hence also the Painlev\'e $(4)$ equation.
 \item
 Let $\eta=0$, $\beta\neq0$. By
 scaling we can set $\beta=1$. The Hamiltonian becomes
 \begin{gather*}
 H=\frac{\mu^2}{2}-2\lambda^3-t\lambda.
 \end{gather*}
 Thus we obtain the Painlev\'e $\big(\frac72\big)$ Hamiltonian, and hence also the Painlev\'e $\big(\frac72\big)$ equation.
 \item
 Let $\eta=0$, $\beta=0$. The Hamiltonian becomes $H=\frac12\mu^2$,
 which is trivial.
\end{itemize}

\appendix
\section{Proof of Theorems \ref{main0}, \ref{main1} and \ref{main}}

\subsection{Preparation for the proof of Theorem \ref{main0}}\label{ap0}

Recall from~(\ref{heudef1}),~(\ref{heudef}) and~(\ref{compa3}) that we set
 \begin{gather*}
p(z):=p_0(z)-\frac{1}{z-\lambda},
\qquad
p_0(z):=\frac{\tau(z)}{\sigma(z)},
\\
q(z):=q_0(z)+\frac{1}{\sigma(z)}\bigg({-}\eta(\lambda)-\mu\big(\tau(\lambda)-\sigma'(\lambda)\big) -\mu^2\sigma(\lambda) +\frac{\mu\sigma(\lambda)}{(z-\lambda)}\bigg),
 \qquad
 q_0(z):=\frac{\eta(z)}{\sigma(z)},
 \\
 a(z):=\frac{c(z)}{z-\lambda},\qquad b(z):=-\frac{c(\lambda)\mu}{z-\lambda},
 \end{gather*}
where $c(z)$ is a
$t,\lambda$-dependent polynomial of degree $\leq2$.
 Recall that the prime is synonymous with $\partial_z$ and the dot with $\partial_t$.

We~will find conditions on $c$
and the time variable $t$ so that the compatibility conditions~(\ref{compa2.}) and~(\ref{compa1.}), that is,
\begin{gather}
 \dot p-ap'+2b'-pa'+a''=0,
 \label{compa2}
 \\
\label{compa1}
 \dot q+pb'-2qa'-aq'+b''=0,
\end{gather}
 are satisfied.

 The following simple identities will be useful in our calculations:
 \begin{Lemma}
 \begin{gather}
 \frac{(\lambda-s)}{(z-\lambda)}\bigg(\frac1{z-s}-\frac1{\lambda-s}+\frac{z-\lambda}{(\lambda-s)^2}\bigg)=
 \frac{1}{\lambda-s}-\frac1{z-s},\nonumber
\\
 \frac{2}{z-s}-\frac2{\lambda-s}+(z-\lambda)\bigg(\frac{1}{(z-s)^2}+\frac{1}{(\lambda-s)^2}\bigg)=
 \frac{(z-\lambda)^3}{(z-s)^2(\lambda-s)^2}.\label{simple}
 \end{gather}
 Let $\psi:=\frac{\xi(z)}{z-s}$.
 If $\xi$ be a polynomial with $\deg\xi\leq2$, then
 \begin{gather}
 \xi(z)-\xi(\lambda)-(z-\lambda)\xi'(z)=
 - \frac{\xi''}{2}(z-\lambda)^2,\nonumber
 \\ \label{muu2-}
 \psi(z)-\psi(\lambda)-(z-\lambda)\psi'(z)
 =-\frac{\xi(s)(z-\lambda)^2}{(z-s)^2(\lambda-s)}.
 \end{gather}
 If $\xi$ is a polynomial with $\deg\xi\leq3$, then
 \begin{gather} 2\xi(z)-2\xi(\lambda)-(z-\lambda)\big(\xi'(z)+\xi'(\lambda)\big)=-\frac{\xi'''}{6}(z-\lambda)^3,
 \label{muus}
 \\
 2\psi(z)-2\psi(\lambda)-(z-\lambda)\left(\psi'(z)+ \psi'(\lambda)\right)=\xi(s)
 \frac{(z-\lambda)^3}{(z-s)^2(\lambda-s)^2}.\label{simple.2}
 \end{gather}
\end{Lemma}

\subsection{First compatibility condition}\label{ap1}

\begin{Proposition} \label{prop1}
Suppose that the following equation of motion for $\lambda$ holds:
 \begin{gather}
 \dot\lambda=2 c(\lambda)\mu-c'(\lambda)+\frac{c\tau}{\sigma}(\lambda),
\label{fir}
\end{gather}
and we have the condition
\begin{gather}
 0=I:= \frac{\dot\tau}{\sigma}(z)-\frac{\tau\dot\sigma}{\sigma^2}(z) + \frac{\frac{c\tau}{\sigma}(z)-\frac{c\tau}{\sigma}(\lambda)-(z-\lambda)\big(\frac{c\tau}{\sigma}\big)'(z)}
{(z-\lambda)^2}.
\label{mu1}
\end{gather}
Then~\eqref{compa2} is satisfied.
\end{Proposition}

 \begin{proof}
 \begin{gather*}
 0= \dot p(z)-a(z)p'(z)+2b'(z)-p(z)a'(z)+a''(z)
 \\ \phantom{0}
{} =\dot p_0(z)-\frac{\dot\lambda}{(z-\lambda)^2}
 -\frac{c(z)}{z-\lambda}\bigg( p_0'(z)+\frac{1}{(z-\lambda)^2}\bigg)
 +\frac{2c(\lambda)\mu}{(z-\lambda)^2}
 \\ \phantom{0=}
 {}-\bigg(p_0(z)-\frac{1}{(z-\lambda)}\bigg) \bigg(\frac{c'(z)}{z-\lambda}-\frac{c(z)}{(z-\lambda)^2}\bigg)
 +\frac{c''(z)}{z-\lambda}-\frac{2c'(z)}{(z-\lambda)^2}+\frac{2c(z)}{(z-\lambda)^3}
 \\ \phantom{0}
 {}=\frac{1}{(z-\lambda)^2}\big({-} \dot\lambda+ 2c(\lambda)\mu-c'(z)+(cp_0)(z)\big)
 +\frac{1}{(z-\lambda)}\big({-}c(z)p_0'(z)-p_0(z)c'(z)+c''(z)\big)
 \\ \phantom{0=}
 {}+\dot p_0(z).
 \end{gather*}
 We~rearrange this as
 \begin{gather}\label{rea1}
 \phantom{0}= \frac{1}{(z-\lambda)^2}\big({-} \dot\lambda+2 c(\lambda)\mu-c'(\lambda)+(cp_0)(\lambda)\big)
 \\ \phantom{0= }
 \label{rea2}
 {}+\frac{ -c'(z)+c'(\lambda)+c''(z)(z-\lambda)}{(z-\lambda)^2}
 \\ \phantom{0= }
 \label{rea3}
 +\frac{(cp_0)(z)-(cp_0)(\lambda)-(z-\lambda)\big(cp_0)'(z)}
{(z-\lambda)^2} + \dot p_0(z).
 \end{gather}
(\ref{rea1}) is proportional to $\frac{1}{(z-\lambda)^2}$ and the last two lines are regular at $z=\lambda$. Therefore,~(\ref{rea1}) has to vanish separately, yielding the condition~(\ref{fir}).
(\ref{rea2}) vanishes automatically, because $c$ is a~polynomial of degree $\leq2$ in $z$.
(\ref{rea3}) yields the condition~(\ref{mu1}).
\end{proof}

\subsection{Second compatibility condition}\label{ap2}

It is much more difficult to analyze the second compatibility condition.

\begin{Proposition}\label{prop2}
Suppose that
the equation for $\lambda$~\eqref{fir} holds together with the equation for $\mu$
 \begin{gather}
 \dot\mu=-\frac{c\eta'}{\sigma}(\lambda)-\mu\bigg(\frac{c\tau'}{\sigma}(\lambda)-
 \frac{c\sigma''}{2\sigma}(\lambda)-\frac{c''}{2}\bigg)-\mu^2\frac{\sigma'c}{\sigma}(\lambda).
 \label{second}
 \end{gather}
 Assume also that $I=0$ $($see \eqref{mu1}$)$ and
 the following conditions are satisfied:
\begin{gather*} 0=II:=-\frac{\dot\sigma}{\sigma}(z)\big(\eta(z)-\eta(\lambda)\big)+\dot\eta(z)-\dot\eta(\lambda)
\\ \hphantom{ 0=II:=}
{} +\frac{\eta(z)-\eta(\lambda)}{(z-\lambda)}\bigg(\frac{c\sigma'}{\sigma}(z)-c'(z)\bigg)
-\eta'(\lambda)\bigg(\frac{\sigma'c}{\sigma}(\lambda)-c'(\lambda)\bigg)
\\ \hphantom{ 0=II:=}
{}+ \frac{2c\eta(z)-2c\eta(\lambda)-\big((c\eta)'(z)+(c\eta)'(\lambda)\big)(z-\lambda)}{(z-\lambda)^2};
\\
0=III:= \frac{\dot\sigma}{\sigma}(z)- \frac{\dot\sigma}{\sigma}(\lambda)-
 \frac{\frac{c\sigma'}{\sigma}(z)-\frac{c\sigma'}{\sigma}(\lambda)
-(\frac{c\sigma'}{\sigma})'(\lambda)(z-\lambda)} {(z-\lambda)}.
 \end{gather*}
Then~\eqref{compa1} is true.
 \end{Proposition}

 \begin{proof}
 \begin{gather*}
 0= \dot q(z)+p(z)b'(z)-2q(z)a'(z)-a(z)q'(z)+b''(z)
 \\ \hphantom{ 0}
 {}=\frac{1}{\sigma(z)}\bigg({-}\tau(\lambda)+\sigma'(\lambda)-2\mu\sigma(\lambda)
 +\frac{\sigma(\lambda)}{(z-\lambda)}\bigg)\dot\mu
 \\ \hphantom{0=}
 {}+\frac{1}{\sigma(z)}\bigg({-}\eta'(\lambda)-\mu(\tau'(\lambda)-\sigma''(\lambda))-\mu^2\sigma'(\lambda)+
 \frac{\mu\sigma'(\lambda)}{(z-\lambda)} +\frac{\mu\sigma(\lambda)}{(z-\lambda)^2}\bigg)\dot\lambda
 \\ \hphantom{0=}
{}-\frac{\dot\sigma(z)}{\sigma^2(z)}\bigg(\eta(z)-\eta(\lambda)-\mu(\tau(\lambda)-\sigma'(\lambda)) -\mu^2\sigma(\lambda)+ \frac{\mu\sigma(\lambda)}{(z-\lambda)}\bigg)
 \\ \hphantom{0=}
{}+\frac{1}{\sigma(z)}\bigg(\dot\eta(z)\!-\!\dot\eta(\lambda)\!-\mu(\dot\tau(\lambda)\!-\dot\sigma'(\lambda)) \!-\mu^2\dot\sigma(\lambda)\!+\frac{\mu\dot\sigma(\lambda)}{(z\!-\!\lambda)}\bigg)
\!+\bigg(\frac{\tau(z)}{\sigma(z)}\!-\frac{1}{(z\!-\!\lambda)}\bigg)\frac{c(\lambda)\mu}{(z\!-\!\lambda)^2}
\\ \hphantom{0=}
 {}-\frac{2}{\sigma(z)}\bigg(\eta(z)-\eta(\lambda)-\mu\big(\tau(\lambda)-\sigma'(\lambda)\big) -\mu^2\sigma(\lambda)+\frac{\mu\sigma(\lambda)}{(z-\lambda)}\bigg) \bigg(\frac{c'(z)}{(z-\lambda)}-\frac{c(z)}{(z-\lambda)^2}\bigg)
 \\ \hphantom{0=}
 {} +\frac{c(z)\sigma'(z)}{(z-\lambda)\sigma^2(z)}\bigg(\eta(z)-\eta(\lambda) -\mu\big(\tau(\lambda)-\sigma'(\lambda)\big)-\mu^2\sigma(\lambda)
 +\frac{\mu\sigma(\lambda)}{(z-\lambda)}\bigg)
 \\ \hphantom{0=}
{}-\frac{c(z)}{(z-\lambda)\sigma(z)}\bigg(\eta'(z)-\frac{\mu\sigma(\lambda)}{(z-\lambda)^2}\bigg)-
2\frac{c(\lambda)\mu}{(z-\lambda)^3}.
 \end{gather*}
 Next we collect the terms that contain an inverse power of $z-\lambda$.
 These terms are grouped in~seve\-ral categories. In~these terms we also insert~(\ref{fir}). We~obtain
 \begin{gather*}
\phantom{0} =
\frac{1}{\sigma(z)}\big({-}\tau(\lambda)\!+\sigma'(\lambda)-2\mu\sigma(\lambda) \big)\dot\mu
+ \frac{1}{\sigma(z)}\big({-}\eta'(\lambda) -\mu(\tau'(\lambda)-\sigma''(\lambda))-\mu^2\sigma'(\lambda)\big)\dot\lambda
\\ \hphantom{0 =}
{} -\frac{\dot\sigma(z)}{\sigma^2(z)}\big(\eta(z)-\eta(\lambda)-\mu(\tau(\lambda)-\sigma'(\lambda)) -\mu^2\sigma(\lambda)\big)
\\ \hphantom{0 =}
{}+\frac{1}{\sigma(z)}\big(\dot\eta(z)-\dot\eta(\lambda)-\mu(\dot\tau(\lambda)-\dot\sigma'(\lambda)) -\mu^2\dot\sigma(\lambda)\big)
\\ \hphantom{0 =}
{}+\frac{1}{(z-\lambda)\sigma(z)}\bigg({-}\mu\frac{\dot\sigma(z)\sigma(\lambda)}{\sigma(z)} +\mu\dot\sigma(\lambda)+\dot\mu\sigma(\lambda)\bigg)
+\frac{1}{(z-\lambda)^2\sigma(z)}2c(z)\big(\eta(z)-\eta(\lambda)\big)
\\ \hphantom{0 =}
{}+\frac{1}{(z-\lambda)\sigma(z)} \bigg({-}2c'(z)\big(\eta(z)-\eta(\lambda)\big)+
 c(z)\frac{\sigma'(z)}{\sigma(z)}\big(\eta(z)-\eta(\lambda)\big)- c(z)\eta'(z)\bigg)
 \\ \hphantom{0 =}
{}+\frac{\mu}{(z-\lambda)^2\sigma(z)}
 \big(c(\lambda)\tau(\lambda)+c(\lambda)\tau(z)-2c(z)\tau(\lambda)\big)
 \\ \hphantom{0 =} {}+\frac{\mu}{(z-\lambda)\sigma(z)}\bigg(c(\lambda)\frac{\sigma'(\lambda)}{\sigma(\lambda)}\tau(\lambda) -c(z)\frac{\sigma'(z)}{\sigma(z)} \tau(\lambda) +2c'(z)\tau(\lambda)\bigg)
 \\ \hphantom{0 =}
{}+ \frac{\mu}{(z-\lambda)^3\sigma(z)}\big(3c(z)\sigma(\lambda)-3c(\lambda)\sigma(z)\big)
\\ \hphantom{0 =}
{}+ \frac{\mu}{(z-\lambda)^2\sigma(z)}\bigg({-}c'(\lambda)\sigma(\lambda)-2c'(z)\sigma(\lambda) +2c(z)\sigma'(\lambda) +c(z)\frac{\sigma'(z)}{\sigma(z)}\sigma(\lambda)\bigg)
 \\ \hphantom{0 =}
 {}+\frac{\mu}{(z-\lambda)\sigma(z)}\bigg({-}c'(\lambda)\sigma'(\lambda)-2c'(z)\sigma'(\lambda)+
 c(z)\frac{\sigma'(z)}{\sigma(z)}\sigma'(\lambda)\bigg)
 \\ \hphantom{0 =}
{}+\frac{\mu^2}{(z-\lambda)^2\sigma(z)} \big(2c(\lambda)\sigma(\lambda)-2c(z)\sigma(\lambda)\big)
 \\ \hphantom{0 =}
 {}+\frac{\mu^2}{(z-\lambda)\sigma(z)}\bigg(2c(\lambda)\sigma'(\lambda)+2c'(z)\sigma(\lambda)
 -c(z)\frac{\sigma'(z)}{\sigma(z)}\sigma(\lambda) \bigg)
 \\ \hphantom{0}
 {}=-\frac{1}{\sigma(z)}\big(\eta'(\lambda)+\mu\big(\tau'(\lambda)-\sigma''(\lambda)\big) +\mu^2\sigma'(\lambda)\big)\dot\lambda
 -\frac{1}{\sigma(z)}\big(\tau(\lambda)-\sigma'(\lambda)+2\mu\sigma(\lambda)\big)\dot\mu
 \\ \hphantom{0 =}
{} +\frac{\sigma(\lambda)}{\sigma(z)(z-\lambda)}\dot\mu+\frac{1}{\sigma(z)}\bigg(
 -\frac{\dot\sigma}{\sigma}(z)\big(\eta(z)-\eta(\lambda)\big)+\dot\eta(z)-\dot\eta(\lambda)\bigg)
 \\ \hphantom{0 =}
{}+\frac{\mu}{\sigma(z)}\bigg((\lambda)\frac{\dot\sigma}{\sigma}(z)-\dot\tau(\lambda)\bigg)
+\frac{\mu}{\sigma(z)}\bigg({-}\sigma'(\lambda)\frac{\dot\sigma}{\sigma}(z)+\dot\sigma'(\lambda)-
 \sigma(\lambda)\frac{\big(\frac{\dot\sigma}{\sigma}(z)-
 \frac{\dot\sigma}{\sigma}(\lambda)\big)}{(z-\lambda)}\bigg)
 \\ \hphantom{0 =}
 {}+\frac{\mu^2}{\sigma(z)}\sigma(\lambda)\bigg(\frac{\dot\sigma}{\sigma}(z)-
 \frac{\dot\sigma}{\sigma}(\lambda)\bigg)
 +\frac{1}{(z-\lambda)\sigma(z)}c(\lambda)\eta'(\lambda)
 \\ \hphantom{0 =}
{} +\frac{1}{\sigma(z)}\bigg( \frac{
2c\eta(z)\!-\!2c\eta(\lambda)\!-\!\big((c\eta)'(z)\!+\!(c\eta)'(\lambda)\big)(z\!-\!\lambda)}{(z\!-\!\lambda)^2}
 \!+\!\frac{\eta(z)\!-\!\eta(\lambda)}{(z\!-\!\lambda)}\bigg(c\frac{\sigma'}{\sigma}(z)\!-\!c'(z)\bigg)\!\bigg)
 \\ \hphantom{0 =}
{} +\frac{\mu}{\sigma(z)(z-\lambda)}c(\lambda)\tau'(\lambda)
 +\frac{\mu}{\sigma(z)}\bigg(\frac{(\frac{\sigma'c}{\sigma}(\lambda)-
 \frac{\sigma'c}{\sigma}(z))}{(z-\lambda)}\tau(\lambda)+c(\lambda)\frac{\tau''}{2}+c''\tau(\lambda)\bigg)
 \\ \hphantom{0 =}
{}+\frac{\mu}{(z-\lambda)\sigma(z)}\bigg({-}\frac{c''\sigma(\lambda)}{2}-
 \frac{c(\lambda)\sigma''(\lambda)}{2}\bigg)
 \\ \hphantom{0 =}
 {}+ \frac{\mu}{\sigma(z)} \Bigg({-}c''\sigma'(\lambda)+\frac{\big(\frac{c\sigma'}{\sigma}(z)-\frac{c\sigma'}{\sigma}(\lambda) -(z-\lambda)\big(\frac{c\sigma'}{\sigma}\big)'(\lambda)\big)}{(z-\lambda)^2}\sigma(\lambda)
 \\ \hphantom{0 =+ \frac{\mu}{\sigma(z)} \Bigg(}
 {}+\frac{\big(\frac{c\sigma'}{\sigma}(z)-\frac{c\sigma'}{\sigma}(\lambda)\big)}{z-\lambda}
 \sigma'(\lambda)-\frac{\sigma'''c(\lambda)}{2}\Bigg)
 \\ \hphantom{0 =}
{} +\frac{\mu^2}{\sigma(z)(z-\lambda)}c(\lambda)\sigma'(\lambda)
 +\frac{\mu^2}{\sigma(z)}\bigg(c''\sigma(\lambda)-\sigma(\lambda) \frac{\big(\frac{\sigma'c}{\sigma}(z)-\frac{\sigma'c}{\sigma}(\lambda)\big)}{(z-\lambda)}\bigg). \end{gather*}
 The singular term equals
 \begin{gather*}
\frac{1}{\sigma(z)(z-\lambda)}\bigg(\sigma(\lambda)\dot\mu
 + c(\lambda)\eta'(\lambda)+\mu\big(c(\lambda)\tau'(\lambda)- \frac{1}{2}c(\lambda)\sigma''(\lambda)-\frac{1}{2}c''\sigma(\lambda)\big) +\mu^2c(\lambda)\sigma'(\lambda)\bigg).
 \end{gather*}
 It yields the equation for $\dot\mu$, that is~(\ref{second}).
 After inserting~(\ref{fir}) and~(\ref{second}) the first two lines become
 \begin{gather*}
 \frac{1}{\sigma(z)\sigma(\lambda)} \eta'(\lambda)\big({-}c(\lambda)\sigma'(\lambda)+c'(\lambda)\sigma(\lambda)\big)
 \\ \qquad
 {}+\frac{\mu}{\sigma(z)\sigma(\lambda)}\bigg(\tau'(\lambda) \big({-}c(\lambda)\sigma'(\lambda)+c'(\lambda)\sigma(\lambda)\big)
 +\tau(\lambda)\bigg({-}\frac{c''\sigma(\lambda)}{2}
{}+ \frac{c(\lambda)\sigma''(\lambda)}{2}\bigg)\!\bigg)
\\ \qquad
+\frac{\mu}{\sigma(z)\sigma(\lambda)}\bigg(\frac{c''\sigma'(\lambda)\sigma(\lambda)}{2}+
\frac{c(\lambda)\sigma'(\lambda)\sigma''(\lambda)}{2}-
c'(\lambda)\sigma(\lambda)\sigma''(\lambda)\bigg)
\\ \qquad
{}+\frac{\mu^2}{\sigma(z)\sigma(\lambda)}\big({-}c(\lambda)\sigma'(\lambda)^2 +c'(\lambda)\sigma(\lambda)\sigma'(\lambda)-
c''\sigma(\lambda)^2+c(\lambda)\sigma''(\lambda)\sigma(\lambda)\big).
 \end{gather*}
 Finally, we obtain
 \begin{gather}
\phantom{0} =\frac{1}{\sigma(z)}\Bigg({-}\frac{\dot\sigma}{\sigma}(z)\big(\eta(z)-\eta(\lambda)\big) +\dot\eta(z)-\dot\eta(\lambda)
+\frac{\eta(z)-\eta(\lambda)}{(z-\lambda)}\bigg(\frac{c\sigma'}{\sigma}(z)-c'(z)\bigg)\nonumber
\\ \hphantom{??? =\frac{1}{\sigma(z)}\Bigg(}
{}-\eta'(\lambda)\bigg(\frac{\sigma'c}{\sigma}(\lambda)-c'(\lambda)\bigg)
+ \frac{2c\eta(z)-2c\eta(\lambda)-\big((c\eta)'(z)+(c\eta)'(\lambda)\big)(z-\lambda)}{(z-\lambda)^2} \Bigg)\nonumber
\\ \hphantom{0 =}
{}+\frac{\mu}{\sigma(z)}\Bigg({-}\dot\tau(\lambda) +c(\lambda)\frac{\tau''}{2}
-\tau'(\lambda)\bigg(\frac{\sigma'c}{\sigma}(\lambda)-c'(\lambda)\bigg)\nonumber
\\ \hphantom{??? =+\frac{\mu}{\sigma(z)}\Bigg(}
{}+\tau(\lambda)\bigg(\frac{\dot\sigma}{\sigma}(z)+\frac{c''}{2}
+ \frac{c(\lambda)\sigma''(\lambda)}{2\sigma(\lambda)}
-\frac{\big(\frac{\sigma'c}{\sigma}(z)-\frac{\sigma'c}{\sigma}(\lambda)\big)}{(z-\lambda)}\bigg)\Bigg)
\nonumber
\\ \hphantom{0 =}
{}+\frac{\mu}{\sigma(z)}\Bigg({-}\sigma'(\lambda)\frac{\dot\sigma}{\sigma}(z)+\dot\sigma'(\lambda)-
\sigma(\lambda)\frac{\big(\frac{\dot\sigma}{\sigma}(z)-
\big(\frac{\dot\sigma}{\sigma}(\lambda)\big)}{(z-\lambda)}\nonumber
\\ \hphantom{??? =+\frac{\mu}{\sigma(z)}\Bigg(}
{}-\frac{c''}{2}\sigma'(\lambda)+\frac{c(\lambda)}{2}\frac{\sigma'(\lambda)}{\sigma(\lambda)}
\sigma''(\lambda)-c'(\lambda)\sigma''(\lambda)-\frac{\sigma'''c(\lambda)}{2}\nonumber
\\ \hphantom{??? =+\frac{\mu}{\sigma(z)}\Bigg(}
{}+\frac{\big(\frac{c\sigma'}{\sigma}(z)-\frac{c\sigma'}{\sigma}(\lambda) -(z-\lambda)\big(\frac{c\sigma'}{\sigma}\big)'(\lambda)\big)}{(z-\lambda)^2}\sigma(\lambda) +\frac{\big(\frac{c\sigma'}{\sigma}(z)-\frac{c\sigma'}{\sigma}(\lambda)\big)}{z-\lambda}
\sigma'(\lambda)\Bigg)\nonumber
\\ \hphantom{0 =}
{}+ \frac{\mu^2\sigma(\lambda)}{\sigma(z)}\bigg(\frac{\dot\sigma}{\sigma}(z)-
 \frac{\dot\sigma}{\sigma}(\lambda)-
 \frac{\frac{c\sigma'}{\sigma}(z)-\frac{c\sigma'}{\sigma}(\lambda)
-(\frac{c\sigma'}{\sigma})'(\lambda)(z-\lambda)} {(z-\lambda)}\bigg)\nonumber
 \\ \hphantom{0}
{} =\frac{1}{\sigma(z)}II +\bigg({-}\frac{\mu\sigma(\lambda)}{\sigma(z)(z-\lambda)}
 +\frac{\mu}{\sigma(z)}\big(\tau(\lambda)-\sigma'(\lambda)\big)
 + \frac{\mu^2\sigma(\lambda)}{\sigma(z)}\bigg)III\nonumber
 \\ \hphantom{0 =}
{} +\frac{\mu}{\sigma(z)}\bigg((\tau(\lambda)-\sigma'(\lambda))\frac{\dot\sigma}{\sigma}(\lambda) -\dot\tau(\lambda)+\dot\sigma'(\lambda)\bigg)\nonumber
\\ \hphantom{0 =}
 {} +\frac{\mu}{\sigma(z)}\Bigg( \frac12\big((\tau-\sigma')c\big)''(\lambda)
-(\tau-\sigma')'(\lambda)\frac{c\sigma'}{\sigma}(\lambda)\nonumber
\\ \hphantom{???=+\frac{\mu}{\sigma(z)}\Bigg(}
{}+(\tau-\sigma')(\lambda)\bigg({-}\frac{c'\sigma'}{\sigma}(\lambda)+\frac{c\sigma'\sigma'}{\sigma^2}(\lambda) -\frac{c\sigma''}{2\sigma}(\lambda)\bigg) \Bigg)\nonumber
 \\ \hphantom{0}
 {}=\frac{1}{\sigma(z)}II
 +\bigg({-}\frac{\mu\sigma(\lambda)}{\sigma(z)(z-\lambda)}
 +\frac{\mu}{\sigma(z)}\big(\tau(\lambda)-\sigma'(\lambda)\big)
 + \frac{\mu^2\sigma(\lambda)}{\sigma(z)}\bigg)III\nonumber
 \\ \hphantom{0 =}\label{powq2}
{}+\frac{\mu\sigma(\lambda)}{\sigma(z)}
 \bigg({-}\partial_t\frac{\tau}{\sigma}(\lambda)+\partial_t\frac{\sigma'}{\sigma}(\lambda)
 +\frac12\big(\frac{\tau c}{\sigma}\big)''(\lambda)
 -\frac12\big(\frac{\sigma' c}{\sigma}\big)''(\lambda) \bigg).
 \end{gather}
 We~have
 \begin{gather}\label{powq3}
 \lim_{z\to\lambda}I=\partial_t\frac{\tau}{\sigma}(\lambda)-\frac12
\big(\frac{\tau c}{\sigma}\big)''(\lambda),
\\
\label{powq4}
\lim_{z\to\lambda}\frac{III}{(z-\lambda)}=
 \partial_t\frac{\sigma'}{\sigma}(\lambda)
 -\frac12\big(\frac{\sigma' c}{\sigma}\big)''(\lambda).
 \end{gather}
 Therefore, if $I$ and $III$ vanish, then so do
(\ref{powq3}) and~(\ref{powq4}), and hence also~(\ref{powq2})
 \end{proof}

Propositions~\ref{prop1} and~\ref{prop2} prove Theorem~\ref{main0}.

Next we would like to prove Theorems~\ref{main1} and~\ref{main}.
The proof will be divided into three subsections. In~the first two we
consider Cases A and in the third Case B.

 From now on we assume that $\sigma$, $\tau$, $\eta$ correspond to a Heun class equation, that is $\deg\sigma\leq3$,
$\deg\tau\leq2$ and $\deg\sigma\eta\leq4$.

\subsection{Case A, Part I}\label{apa}

Assume that $\sigma$ has a zero at $z=s$ so that
$\sigma(z)=(z-s)\rho(z)$. Clearly, $\deg\rho\leq2$.

We~make the ansatz\vspace{-.5ex}
 \begin{gather} \label{ansa}
 c(z):=m(\lambda-s)\rho(z),
\vspace{-.5ex}
 \end{gather}
where $m$ is a function only of $t$.

The equations \eqref{fir} and \eqref{second} can be rewritten as
 \begin{gather}\label{fir4}
\dot\lambda=m\big(2 \sigma(\lambda)\mu-(\lambda-s)\rho'(\lambda)+\tau(\lambda)\big),
\\
\dot\mu=-m\big(\eta'(\lambda)+\mu(\tau'(\lambda)-\rho'(\lambda)-(\lambda-s)
 \rho''(\lambda))+\mu^2\sigma'(\lambda)\big). \label{second4}
 \end{gather}
It is easy to check that \eqref{fir4} and \eqref{second4} are the Hamilton equations
for the Hamiltonian
 \begin{gather*}
 H=m\big(\eta(\lambda)+\big(\tau(\lambda)-(\lambda-s)\rho'(\lambda)\big)\mu + \sigma(\lambda)\mu^2\big).
\end{gather*}

Another consequence of \eqref{ansa} is
 \begin{gather*}
\frac{c\tau}{\sigma}(z)=\frac{m(\lambda-s)\tau(z)}{(z-s)}.
\end{gather*}
Therefore, using \eqref{muu2-} we obtain
 \begin{gather*} \frac{\frac{c\tau}{\sigma}(z)-\frac{c\tau}{\sigma}(\lambda)-(z-\lambda)\big(\frac{c\tau}{\sigma}\big)'(z)}
{(z-\lambda)^2}=-\frac{m\tau(s)}{(z-s)^2}.
\end{gather*}
Hence,
 \begin{gather*}
 I=\partial_t \frac{\tau}{\sigma}(z)-\frac{m\tau(s)}{(z-s)^2}.\vspace{-.5ex}
\end{gather*}

Using
 \begin{gather*}
\frac{c\sigma'}{\sigma}(z)-c'(z)=\frac{m(\lambda-s)\rho(z)}{z-s},\vspace{-.5ex}
\end{gather*}
we obtain
\begin{gather*}
\frac{\eta(z)-\eta(\lambda)}{(z-\lambda)}\bigg(\frac{c\sigma'}{\sigma}(z)-c'(z)\bigg)
-\eta'(\lambda)\bigg(\frac{\sigma'c}{\sigma}(\lambda)-c'(\lambda)\bigg)
\\ \qquad
{}=\frac{m\big(\eta(z)-\eta(\lambda)\big)(\lambda-s)\rho(z)} {(z-\lambda)(z-s)}-m\eta'(\lambda)\rho(\lambda).
 \end{gather*}

Therefore,
\begin{gather*} II=-\frac{\dot\sigma}{\sigma}(z)\big(\eta(z)-\eta(\lambda)\big)+\dot\eta(z)-\dot\eta(\lambda)
+\frac{m\big(\eta(z)-\eta(\lambda)\big)(\lambda-s)\rho(z)}{(z-\lambda)(z-s)}-m\eta'(\lambda)\rho(\lambda)
\\ \hphantom{II=}
{}+\frac{m(\lambda-s)}{(z-\lambda)^2}\big(
2\rho\eta(z)-2\rho\eta(\lambda)-\big((\rho\eta)'(z)+(\rho\eta)'(\lambda)\big)(z-\lambda)\big) .
\end{gather*}

Finally, we have
\begin{gather*}
 \frac{c\sigma'(z)}{\sigma(z)}=\frac{m(\lambda-s)(\rho(z)+(z-s)\rho'(z)) }{(z-s)}.\vspace{-.5ex}
\end{gather*}
Using~\eqref{muu2-} with $\xi(z)=-\rho(z)-(z-s)\rho'(z)$ and the interchanged role of $\lambda$
and $z$ we obtain
\begin{gather*}
 \frac{\frac{c\sigma'}{\sigma}(z)-\frac{c\sigma'}{\sigma}(\lambda)
-(\frac{c\sigma'}{\sigma})'(\lambda)(z-\lambda)} {(z-\lambda)}
 =m\rho(s)\frac{(z-\lambda)}{(z-s)(\lambda-s)}.
\end{gather*}
Therefore,
\begin{gather*}
 III=\frac{\dot\sigma}{\sigma}(z)- \frac{\dot\sigma}{\sigma}(\lambda)-
 m\rho(s)\frac{(z-\lambda)}{(z-s)(\lambda-s)}.
\end{gather*}

\noindent{\bf Subcase A1.}
We~assume that the root $s$ is single and the time variable is
chosen as $t=s$. We~can write
$\sigma(z)=(z-t)\rho(z)$, $\rho(t)\neq0$, $\deg\rho\leq2$.

We~assume $\deg\eta_0\leq1$, $\deg\phi\leq1$, $\kappa, \alpha\in\cc$,\vspace{-.5ex}
\begin{gather*}
 \tau(z)=(1-\kappa)\rho(z)+\phi(z)(z-t),
 \qquad
\eta(z)=\frac{\alpha\rho(t)}{z-t}+\eta_0(z),\qquad
 \dot\kappa=\dot\rho=\dot\phi=\dot\eta_0=\dot\alpha=0.\vspace{-.5ex}
\end{gather*}

We~have
\begin{gather*}
\partial_t\frac{ \tau}{\sigma}(z)=\frac{1-\kappa}{(z-t)^2},\qquad
1-\kappa=\frac{\tau(t)}{\rho(t)}.\vspace{-.5ex}
\end{gather*}
Therefore,
\begin{gather*}
 I=\frac{1-\kappa}{(z-t)^2}\big(1-m\rho(t)\big).\vspace{-.5ex}
\end{gather*}

Using $\deg\eta_0\leq1$ and $\dot\eta_0=0$ we obtain
\begin{gather*} -\frac{\dot\sigma}{\sigma}(z)\big(\eta_0(z)-\eta_0(\lambda)\big)+\dot\eta_0(z)-\dot\eta_0(\lambda)
=\frac1{(z-t)}\eta_0'(z-\lambda),
\\
\frac{m\big(\eta_0(z)-\eta_0(\lambda)\big)(\lambda-t)\rho(z)}{(z-\lambda)(z-t)}-m\eta_0'(\lambda)\rho(\lambda)
= -\frac{m\rho(t) \eta_0' (z-\lambda)}{(z-t)}+m\eta_0'(\lambda-t)\frac{\rho''}{2}(z-\lambda).
\end{gather*}
Using $\deg\eta_0\rho\leq3$ and~\eqref{muus} we get
\begin{gather*}
\frac{m(\lambda-t)}{(z-\lambda)^2}
\big(2\rho\eta_0(z)-2\rho\eta_0(\lambda)-\big((\rho\eta_0)'(z)+(\rho\eta_0)'(\lambda)\big)(z-\lambda)\big)
\\ \qquad
{}=-\frac{m(\lambda-t)(\rho\eta_0)'''(z-\lambda)}{6}
=-\frac{m(\lambda-t)\rho''\eta_0'(z-\lambda)}{2}.
 \end{gather*}

Set $\psi(z):=\frac{\rho(t)}{z-t}$.
 We~have
\begin{gather*}
 -\frac{\dot\sigma}{\sigma}(z)\big(\psi(z)-\psi(\lambda)\big)+\dot\psi(z)-\dot\psi(\lambda)
=\bigg(\frac{\rho(t)}{z-t}+\rho'(t)\bigg)\bigg(\frac1{z-t}-\frac1{\lambda-t}\bigg)
\\ \hphantom{-\frac{\dot\sigma}{\sigma}(z)\big(\psi(z)-\psi(\lambda)\big)+\dot\psi(z)-\dot\psi(\lambda)
=}
 {}+\rho(t)\bigg(\frac1{(z-t)^2}-\frac1{(\lambda-t)^2}\bigg),
 \\
 \frac{m\big(\psi(z)-\psi(\lambda)\big)(\lambda-t)\rho(z)}{(z-\lambda)(z-t)}-m\psi'(\lambda)\rho(\lambda)
 = -\frac{ m\rho(t)}{(z-t)^2} \bigg(\rho(t)+\rho'(t)(z-t)+\frac{\rho''}{2}(z-t)^2\bigg)
\\ \hphantom{ \frac{m\big(\psi(z)-\psi(\lambda)\big)(\lambda-t)\rho(z)}{(z-\lambda)(z-t)}-m\psi'(\lambda)\rho(\lambda)
 =}
{}+ \frac{ m\rho(t)}{(\lambda-t)^2} \bigg(\rho(t)\!+\rho'(t)(\lambda\!-t)\!+\frac{\rho''}{2}(\lambda\!-t)^2\bigg).
 \end{gather*}
Using \eqref{simple.2} we get
\begin{gather*}
\frac{m(\lambda-t)}{(z-\lambda)^2}
\big( 2\rho\psi(z)-2\rho\psi(\lambda)-\big((\rho\psi)'(z)+(\rho\psi)'(\lambda)\big)(z-\lambda)\big)
 =\frac{m\rho(t)^2(z-\lambda)}{(z-t)^2(\lambda-t)}
 \\ \hphantom{\frac{m(\lambda-t)}{(z-\lambda)^2}
\big( 2\rho\psi(z)-2\rho\psi(\lambda)-\big((\rho\psi)'(z)+(\rho\psi)'(\lambda)\big)(z-\lambda)\big)}
{}=-\frac{m\rho(t)^2}{(z-t)}\bigg(\frac1{z-t}-\frac1{\lambda-t}\bigg).
 \end{gather*}
Combining the above identities we obtain
\begin{gather*}
II =\bigg(\!\bigg({-}(\lambda-t)\eta_0'+\frac{\rho(t)}{z-t}+\rho'(t)\bigg) \bigg(\frac1{z-t}-\frac1{\lambda-t}\bigg)
+\rho(t)\bigg(\frac1{(z-t)^2}-\frac1{\lambda-t)^2}\bigg)\!\bigg)
\\ \hphantom{II =}
{}\times(1-m\rho(t)\big).
 \end{gather*}
We~have\vspace{-.5ex}
\begin{gather*}
\frac{\dot\sigma}{\sigma}(z)-\frac{\dot\sigma}{\sigma}(\lambda)
=-\frac{1}{z-t}+\frac{1}{\lambda-t}=\frac{z-\lambda}{(\lambda-t)(z-t)}.\vspace{-.5ex}
\end{gather*}
Therefore,\vspace{-.5ex}
\begin{gather*}
 III= \frac{z-\lambda}{(\lambda-t)(z-t)}\big(1-m\rho(t)\big).\vspace{-.5ex}
\end{gather*}
Thus $m=\frac{1}{\rho(t)}$ implies $I=II=III=0$.

\subsection{Case A, Part II}\label{apa1}

Assume that the root $s$ is at least double. Then $\rho(s)=0$ so that the normalization
$m=\frac{1}{\rho(s)}$ does not work. We~have to change the time variable.

Thus we assume $\sigma(z)=(z-s)^2\rho_1(z)$, where $\deg\rho_1\leq1$.
We~also assume $\dot{s}=\dot\rho_1=0$. Then we have
\begin{gather*}
I=\frac{\dot\tau}{\sigma}(z)-\frac{m\tau(s)}{(z-s)^2},
\\
II =\dot\eta(z)-\dot\eta(\lambda)+ \frac{m(\lambda-s)}{(z-\lambda)}\bigg(
 \rho_1 \eta (z)-\rho_1 \eta (\lambda)-(\rho_1 \eta)'(\lambda)(z-\lambda)-\frac{(\rho \eta)'''}{6}(z-\lambda)^2\bigg).
 \end{gather*}
We~consider separately two subcases: in the first the time variable is
contained in $p_0$ and in the second in $q_0$.
\medskip

\noindent {\bf Subcase Ap.}
We~assume that $\dot\tau_0=\dot\eta=0$, $\deg\tau_0\leq2$,
$\sigma\eta(s)=(\sigma\eta)'(s)=0$,
\begin{gather*}
\tau(z)=\tau_0(z)+t\rho_1(z),
\qquad
\tau_0(s)\neq0\quad \text{or}\quad \rho_1(s)\neq0.
\end{gather*}
We~have\vspace{-.5ex}
\begin{gather*}
\frac{\dot\tau}{\sigma}(z)=\frac{1}{(z-s)^2},\qquad \tau(s)=\tau_0(s)+t\rho_1(s)\not\equiv0.\vspace{-.5ex}
\end{gather*}
Hence,\vspace{-.5ex}
\begin{gather*}
 I=\frac{1}{(z-s)^2}\big(1-m(\tau_0(s)+t\rho_1(s))\big).\vspace{-.5ex}
\end{gather*}
Thus $m=\big(\tau_0(s)+t\rho_1(s)\big)^{-1}$ implies $I=0$.

$\sigma\eta(s)=(\sigma\eta)'(s)=0$ implies that $\rho_1\eta$ is a
polynomial of degree $\leq2$. Therefore,
$(\rho \eta)'''=\big(\rho_1 \eta)(z-s)\big)'''=3(\rho_1 \eta)''$ and
\begin{gather*}
II= \frac{m(\lambda-s)}{(z-\lambda)}\bigg(\rho_1 \eta (z)-\rho_1 \eta (\lambda)-(\rho_1 \eta)'(\lambda)(z-\lambda)-\frac{(\rho_1 \eta)''}{2} (z-\lambda)^2\bigg)=0.
\end{gather*}

\medskip

\noindent{\bf Subcase Aq.} We~assume that $\tau(s)=0$, $\sigma\eta_0(s)=0$, $\deg\sigma\eta_0\leq4$,
\begin{gather*}
\eta(z)=\frac{t}{z-s}+\eta_0(z),\qquad \dot\sigma=\dot\tau=\dot\eta_0=0,
\\
\rho_1(s)\neq0\qquad\text{or}\qquad(\sigma\eta_0)'\neq0.
\end{gather*}

Clearly, $I=0$.

Now
\begin{gather*}
\sigma\eta_0(z)=(z-s)(\sigma\eta_0)'(s)+(z-s)^2\psi(z),
\end{gather*}
where $\deg\psi\leq2$. Therefore,
\begin{gather*}
\eta_0\rho_1(z)=\frac{(\sigma\eta_0)'(s)}{z-s}+\psi(z),\\
\rho_1\eta(z)=\frac{(\sigma\eta_0)'(s)+\rho_1(s)t}{z-s}+\psi(z)+\rho_1't.
\end{gather*}
Now
$(\eta\rho)'''=\big((z-s)\rho_1\eta\big)'''=\big((z-s)\psi\big)'''=3\psi''$. Therefore,
\begin{gather}\nonumber
\rho_1 \eta (z)-\rho_1 \eta (\lambda)-(\rho_1 \eta)'(\lambda)(z-\lambda)-\frac{(\rho \eta)'''}{6}(z-\lambda)^2
\\ \qquad
{} =\psi(z)-\psi(\lambda)-(z-\lambda)\psi'(\lambda)-(z-\lambda)^2\frac{\psi''}{2}\label{pisa}
\\ \qquad\phantom{=}
{}+\big((\sigma\eta_0)'(s)+\rho_1(s)t\big) \bigg(\frac{1}{z-s}-\frac{1}{\lambda-s}+\frac{(z-\lambda)}{(\lambda-s)^2}\bigg).\label{pisa1}
\end{gather}
\eqref{pisa} vanishes, because $\deg\psi\leq2$. Using \eqref{pisa1}, \eqref{simple} and
\begin{gather*}
 \dot\eta(z)-\dot\eta(\lambda)=\frac{1}{z-s}-\frac{1}{\lambda-s}
\end{gather*}
we obtain
\begin{gather*}
 II=\bigg(\frac{1}{z-s}- \frac{1}{\lambda-s}\bigg) \big(1-m \big((\sigma\eta_0)'(s)+\rho_1(s)t\big) \big),
\end{gather*}
where $(\sigma\eta_0)'(s)+\rho_1(s)t\not\equiv0$.
 Hence if
 $m=\big((\sigma\eta_0)'(s)+\rho_1(s)t\big)^{-1}$, then $II=0$.

\subsection{Case B}\label{apb}

We~assume that $\deg\sigma\leq2$, $\deg\sigma\eta\leq3$ and $\dot\sigma=0$. We~set
 \begin{gather*}\label{caseb}
 c(z)=m\sigma(z),
 \end{gather*}
 where $m$ is a function just of $t$.
The equations \eqref{fir} and \eqref{second} can be rewritten as
\begin{gather}\label{fir5}
\dot\lambda=m\big(2 \sigma(\lambda)\mu-\sigma'(\lambda)+\tau(\lambda)\big),
\\
\dot\mu=-m\big(\eta'(\lambda)+\mu(\tau'(\lambda)-
 \sigma''(\lambda))+\mu^2\sigma'(\lambda)\big).\label{second5}
 \end{gather}
We~easily check that \eqref{fir5} and \eqref{second5} are the Hamilton equations
for the Hamiltonian
 \begin{gather*}
 H(t,\lambda,\mu):=m\big(\eta(\lambda)+\big(\tau(\lambda)-\sigma'(\lambda)\big) \mu
 +\sigma(\lambda)\mu^2\big).
 \end{gather*}

Using $\deg\tau\leq2$, we get
\begin{gather*}
\frac{\tau(z)-\tau(\lambda)-(z-\lambda)\tau'(z)}{(z-\lambda)^2}=-\frac{\tau''}{2}.
\end{gather*}
Therefore,
\begin{gather*}
I= \frac{\dot \tau(z)}{\sigma(z)}-m\frac{\tau''}{2}.
\end{gather*}

Using $\deg\sigma\eta\leq3$ we obtain
 \begin{gather*}
\frac{2\sigma\eta(z)-2\sigma\eta(\lambda)-\big((\sigma\eta)'(z)+(\sigma\eta)'(\lambda)\big) (z-\lambda)}{(z-\lambda)^2}=-\frac{(\sigma\eta)'''}{6}(z-\lambda).
 \end{gather*}
 Therefore,
 \begin{gather*}
 II= \dot\eta(z)-\dot\eta(\lambda) -m\frac{(\sigma\eta)'''}{6}(z-\lambda).
 \end{gather*}

$III$ is clearly $0$, because $\dot\sigma=0$ and $\frac{c\sigma'}{\sigma}$ is a polynomial of degree $\leq1$.

This ends the proof of Theorem~\ref{main}B.

Again, we consider two subcases, with $t$ contained in $p_0$ and in~$q_0$.

\medskip

\noindent{\bf Subcase Bp.}
$\deg\sigma\leq2$, $\deg\tau_0\leq2$, $\deg\sigma\eta\leq2$,
\begin{gather*}
\tau(z)=t\sigma(z)+\tau_0(z),\qquad \dot\sigma=\dot\tau_0= \dot\eta=0,
\\
\deg\sigma=2\qquad\text{or}\qquad\deg\tau_0=2.
\end{gather*}

$II=0$ is automatic. Moreover,
$\tau''=t\sigma''+\tau_0''\not\equiv0$ and
\begin{gather*}
I=1-m\bigg(t\frac{\sigma''}{2}+\frac{\tau''}{2}\bigg).
\end{gather*}
Hence $m=\big(t\frac{\sigma''}2+\frac{\tau_0''}2\big)^{-1}$ implies $I=0$.

\medskip

\noindent{\bf Subcase Bq.}
 $\deg\sigma\leq2$, $\deg\tau\leq1$, $ \deg\sigma\eta_0\leq3$,
\begin{gather*}
\eta(z)=tz+\eta_0(z),\qquad \dot\sigma=\dot\tau=\dot\eta_0=0,
\\
\deg\sigma=2\qquad\text{or} \qquad \deg\sigma\eta_0=3.
\end{gather*}

$I=0$ is automatic. Moreover,
$(\sigma\eta)'''=t3\sigma''+(\sigma\eta)'''\not\equiv0$ and
\begin{gather*}
II=(z-\lambda)\bigg(1-m\bigg(t\frac{\sigma''}{2}+\frac{(\sigma\eta_0)'''}{6}\bigg)\!\bigg).
\end{gather*}
Therefore
$m=\big(t\frac{\sigma''}{2}+\frac{(\sigma\eta_0)'''}{6}\big)^{-1}$ implies $II=0$.

\section{Hamilton equations}
\subsection{From Hamilton equations to second order equations}

 We~devote this appendix
to a few remarks about Hamilton equations.

Suppose that $H(t,\lambda,\mu)$ is a~time-dependent Hamiltonian. The equations
\begin{gather}\label{hamil0}
\frac{\d\lambda}{\d t}=\frac{\partial H}{\partial \mu}(t,\lambda,\mu),
\\
\frac{\d\mu}{\d t}=-\frac{\partial H}{\partial \lambda}(t,\lambda,\mu),
\label{hamil}
\end{gather}
are called the Hamilton equations generated by $H$.

 All Painlev\'e Hamiltonians have the form
 \begin{gather*}
 H(t,\lambda,\mu)=f(t,\lambda)\frac{\mu^2}{2} +\mu g(t,\lambda)+h(t,\lambda).
 \end{gather*}
 For such Hamiltonians it is easy to eliminate
 $\mu$ from the Hamilton equations. One obtains a~second order differential equation for $\lambda$ of the form
\begin{gather*}
\frac{\d^2\lambda}{\d t^2}=A(t,\lambda)\bigg(\frac{\d\lambda}{\d t}\bigg)^2+B(t,\lambda)\frac{\d \lambda}{\d t}+C(t,\lambda),
\\
A:= \frac{1}{2f} \frac{\partial f}{\partial\lambda},\qquad
B: = \frac{1}{f} \frac{\partial f}{\partial t},\qquad
C: = - \frac{g^2}{2f} \frac{\partial f}{\partial\lambda} - \frac{g}{f} \frac{\partial f}{\partial t} + g \frac{\partial g}{\partial\lambda} + \frac{\partial g}{\partial t} - f \frac{\partial h}{\partial\lambda}.
\end{gather*}

\subsection{Invariance of Hamilton equations}

The Hamilton equations are invariant with respect to various transformations.
\begin{itemize}\itemsep=0pt
\item The equations generated by $\epsilon
H(\epsilon t,\lambda,\mu)$ are equivalent to~(\ref{hamil0}) and~(\ref{hamil}).
\item Let $(\lambda,\mu)\mapsto\big(\tilde\lambda,\tilde\mu\big)$ be a (time-independent) canonical transformation, that means
 \begin{gather*}
 \frac{\partial\tilde\lambda}{\partial\lambda} \frac{\partial\tilde\mu}{\partial\mu}
 - \frac{\partial\tilde\lambda}{\partial\mu} \frac{\partial\tilde\mu}{\partial\lambda}=1.
 \end{gather*}
 Then the Hamilton equations in the new variables
 \begin{gather*}
\frac{\d\tilde\lambda}{\d t}=\frac{\partial H}{\partial \tilde\mu},
\qquad
\frac{\d\tilde\mu}{\d t}=-\frac{\partial H}{\partial \tilde\lambda},
\end{gather*}
 are equivalent to~(\ref{hamil}).
\item
 The Hamilton equations are invariant with respect to the following time-dependent transformation:
 \begin{gather*}
 \tilde\lambda=t^{-1}\lambda,\qquad \tilde\mu=t\mu,\qquad
 \tilde H=H-t^{-1}\tilde\lambda\tilde\mu.
 \end{gather*}
\end{itemize}

\subsection{Hamiltonian solvable in quadratures}\label{hammo1}
 Let us now consider a Hamiltonian of the form
 \begin{gather}
 H(t,\lambda,\mu)=m(t)\bigg(f(\lambda)\frac{\mu^2}{2} +\mu g(\lambda)+h(\lambda)\bigg).\label{hammo}
 \end{gather}
 We~will show that it is solvable in quadratures.

 First we change the time from $t$ to $s$, by solving
\begin{gather*}
\frac{\d s}{\d t}=m(t).
 \end{gather*}
 Using the time $s$ we can replace~(\ref{hammo}) by the time-independent Hamiltonian
 \begin{gather*}
 H(\lambda,\mu)=f(\lambda)\frac{\mu^2}{2} +\mu g(\lambda)+h(\lambda).
 \end{gather*}
Now
\begin{gather*}
E:=f(\lambda)\frac{\mu^2}{2} +\mu g(\lambda)+h(\lambda)
\end{gather*}
 is a constant of motion, hence we can express $\mu$ in terms of $\lambda$:
 \begin{gather*}
 \mu=\frac{-g(\lambda)\pm\sqrt{g(\lambda)^2-2f(\lambda)(h(\lambda)-E)}}{f(\lambda)} .
 \end{gather*}
 We~insert this into the first Hamilton equation
 \begin{gather*}
\frac{\d\lambda}{\d s}=\mu f(\lambda)+g(\lambda)
 \end{gather*}
 obtaining
 \begin{gather*}
\frac{\d\lambda}{\d s}=\pm\sqrt{g(\lambda)^2-2f(\lambda)(h(\lambda)-E)}.
 \end{gather*}
This is clearly solvable in quadratures.

\subsection*{Acknowledgements}
J.D.\ and A.I.\ would like to express their gratitude to Galina Filipuk for very useful discussions and remarks. A.I.\ acknowledges the support
by the Armenian Science Committee (SC Grant No.~20RF-171), and the Armenian National Science and Education Fund (ANSEF Grant No.~PS-5701).
The work of J.D.\ and A.L.\ was supported by National Science Center (Poland) under the grant UMO-2019/35/B/ST1/01651.

\LastPageEnding

\end{document}